\documentclass[reqno,a4paper,11pt,oneside]{amsart}    
\usepackage{fixmath}
\newcommand*{\QED}{}
\newcommand*{\Lap}{\upDelta}

\usepackage{float}
\usepackage{mathtools}
\usepackage{amssymb,amsmath}
\usepackage[nocompress]{cite}
\usepackage[noend]{algorithmic}
\usepackage{algorithm}
\usepackage{caption}
\usepackage{subcaption}
\usepackage{enumitem}
\usepackage{units}
\usepackage[final]{graphicx}
\usepackage{xcolor}
\usepackage{booktabs}
\usepackage[english]{babel}
\usepackage[utf8]{inputenc}
\usepackage[T1]{fontenc}
\usepackage{lmodern}

\usepackage{graphics}
\usepackage{epstopdf}

\usepackage{tikz}
\usepackage{pgfplots}

\pgfkeys{/pgf/images/include external/.code=\includegraphics{#1}}

\usetikzlibrary{shapes.misc}
\usetikzlibrary{patterns}
\tikzset{cross/.style={cross out, draw=black, minimum size=2*(#1-\pgflinewidth), inner sep=0pt, outer sep=0pt},
cross/.default={1pt}}
\pgfplotsset{every tick label/.append style={font=\tiny}}

\usepackage[final,stretch=10]{microtype}
\usepackage{psfrag}

\usepackage[colorlinks=true,linkcolor=black,urlcolor=black,citecolor=black,pdftex]{hyperref}
\hypersetup{final} 

\newcommand{\C}{\mathbb{C}}
\newcommand{\N}{\mathbb{N}}
\newcommand{\R}{\mathbb{R}}
\let\oldphi\phi
\renewcommand{\phi}{\varphi}
\newcommand{\eps}{\varepsilon}
\renewcommand{\d}{\mathop{}\!\mathrm{d}}
\newcommand{\de}{\mathop{}\!\mathrm{d}}

\DeclarePairedDelimiter{\abs}{\lvert}{\rvert}
\DeclarePairedDelimiter{\norm}{\lVert}{\rVert}
\DeclarePairedDelimiter{\pair}{\langle}{\rangle}
\DeclarePairedDelimiter{\inner}{(}{)}
\DeclarePairedDelimiter{\set}{\lbrace}{\rbrace}

\DeclareMathOperator{\dist}{dist}

\DeclareMathOperator{\Id}{Id}
\DeclareMathOperator{\supp}{supp}

\DeclareMathOperator*{\argmax}{arg\,max}
\DeclareMathOperator{\conv}{conv}
\DeclareMathOperator{\Ran}{Ran}
\DeclareMathOperator{\Ker}{Ker}

\newenvironment{curlyeq}{\left\{\quad\begin{aligned}}{\end{aligned}\right.}

\makeatletter
\newcommand*\rel@kern[1]{\kern#1\dimexpr\macc@kerna}
\newcommand*\widebar[1]{%
  \begingroup
  \def\mathaccent##1##2{%
    \rel@kern{0.8}%
    \overline{\rel@kern{-0.8}\macc@nucleus\rel@kern{0.2}}%
    \rel@kern{-0.2}%
  }%
  \macc@depth\@ne
  \let\math@bgroup\@empty \let\math@egroup\macc@set@skewchar
  \mathsurround\z@ \frozen@everymath{\mathgroup\macc@group\relax}%
  \macc@set@skewchar\relax
  \let\mathaccentV\macc@nested@a
  \macc@nested@a\relax111{#1}%
  \endgroup
}
\makeatother

\renewcommand*{\Re}{\operatorname{Re}}
\renewcommand*{\Im}{\operatorname{Im}}
\renewcommand*{\i}{\mathrm{i}}

\newcommand*{\kay}{k}

\newcommand{\M}{\mathcal{M}}
\newcommand{\Cc}{\mathcal{C}}

\newcommand{\Oc}{\Omega_c}
\newcommand{\GZ}{\Gamma_Z}
\newcommand{\GN}{\Gamma_N}

\newcommand{\Mcon}{\mathcal{M}(\Oc,\C^N)}
\newcommand{\Ccon}{\Cc(\Oc,\C^N)}

\newcommand{\Nn}{\mathcal{N}}

\newcommand{\J}{J}
\newcommand{\SO}{S}

\newcommand*{\Ogen}{D}
\newcommand*{\Mgen}{\mathcal{M}(\Ogen,H_1)}
\newcommand*{\Cgen}{\Cc_0(\Ogen,H_1)}

\newcommand*{\coeff}[1]{\boldsymbol{#1}}

\graphicspath{{./figures/}}

\theoremstyle{plain}
\newtheorem{theorem}{Theorem}
\numberwithin{theorem}{section}
\newtheorem{proposition}[theorem]{Proposition}
\newtheorem{lemma}[theorem]{Lemma}
\newtheorem{corollary}[theorem]{Corollary}
\theoremstyle{definition}
\newtheorem{definition}{Definition}

\theoremstyle{remark}
\newtheorem{remark}{Remark}

\numberwithin{equation}{section}

\begin{document}

\title[Inverse point source location with the Helmholtz equation]{%
  Inverse point source location with the Helmholtz equation on a bounded domain
}

\author{Konstantin Pieper}
\address{%
  Department of Scientific Computing,
  Florida State University,
  400 Dirac Science Library,
  Tallahassee, FL 32306, USA
}
\email{kpieper@fsu.edu}
\author{Bao Quoc Tang}
\address{%
  Institute of Mathematics and Scientific Computing, University of Graz,
  Heinrichstra{\ss}e 36, 8010 Graz, Austria
}
\email{quoc.tang@uni-graz.at}
\author{Philip Trautmann}
\address{
  Institute of Mathematics and Scientific Computing, University of Graz,
  Heinrichstra{\ss}e 36, 8010 Graz, Austria
}
\email{philip.trautmann@uni-graz.at}
\author{Daniel Walter}
\address{%
  Center for Mathematical Sciences, Chair M17,
  Technische Universit{\"a}t M{\"u}nchen,
  Boltzmannstr. 3,
  85748 Garching bei M{\"u}nchen, Germany
}
\email{walter@ma.tum.de}
\thanks{%
The authors gratefully acknowledge support through the International Research
Training Group IGDK 1754, funded by the German Science Foundation (DFG) and the
Austrian Science Fund (FWF).
K. Pieper acknowledges funding by the US
Department of Energy Office of Science grant DE-SC0016591 and by the US
Air Force Office of Scientific Research grant FA9550-15-1-0001.
D. Walter acknowledges support from the TopMath Graduate
Center of TUM Graduate School and
from the TopMath Program at the Elite Network of Bavaria.
}


\begin{abstract}
  The problem of recovering acoustic sources, more specifically
  mono\-poles, from point-wise measurements of the
  corresponding acoustic pressure at a limited number of frequencies is addressed.
  To this purpose, a family of sparse optimization problems in measure space in
  combination with the Helmholtz equation on a bounded domain is considered.
  A weighted norm with unbounded weight near the observation points
  is incorporated into the formulation.  Optimality conditions and
  conditions for recovery in the small noise case are discussed, which motivates concrete
  choices of the weight.  The numerical realization is based on an accelerated conditional
  gradient method in measure space and a finite element discretization.
\end{abstract} 
  


\maketitle

\section{Introduction}

In this paper, we consider the problem of recovering a sound source \(u\),
consisting of an unknown number time-harmonic monopoles, from pointwise measurements of
the acoustic pressure.
It is well known that under the assumption of a time-harmonic signal consisting of
\(N\) frequencies, the acoustic wave equation can be reduced to a family of
Helmholtz equations.
Concretely, let $\Omega\subset \mathbb \R^d$, $d\in\set{2,3}$ be a bounded, convex,
and polygonal (two dimensional) or polyhedral (three dimensional) domain.
The boundary $\partial\Omega$ is partitioned into
perfectly reflecting walls contained in \(\Gamma_N \subset \partial\Omega\),
and \(\GZ = \partial\Omega\setminus\GN\) modeling absorbing walls or artificial
boundaries arising from a truncation of an unbounded domain.
We model the acoustic pressure \(p_n \in L^2(\Omega)\) at the \(n\)-th frequency as the solution of
\begin{equation}
\label{eq:helm}
\begin{curlyeq}
- \Lap p_n - \kay_n^2 p_n &= u_n\rvert_{\Omega} &&\text{in } \Omega, \\
\partial_\nu p_n - \i\kappa_n p_n &= u_n\rvert_{\GZ} &&\text{on } \GZ, \\
\partial_\nu p_n &= u_n\rvert_{\GN} &&\text{on } \GN,
\end{curlyeq}
\end{equation}
where \(n = 1,2,\ldots,N\).
Here, \(\kay_n > 0\) is a sequence of wavenumbers, which are defined as
usual by \(\kay_n = \omega_n/c\), where \(c\) is the \emph{speed of sound} and
\(\set{\omega_n}_n\) a set of \emph{circular frequencies}. The numbers \(\kappa_n \in \C\) with
\(\Re\kappa_n \neq 0\) are related to the properties of walls that are modeled on the
boundary \(\GZ\); cf.\ \cite{BermudezGamalloRoriguez:2004}.
In the simplest case, we set
\(\kappa_n = \kay_n\), and obtain the well-known zeroth-order absorbing boundary
conditions~\cite{EngquistMajda:1977,Ihlenburg:1998}.
We model the source \(u_n\) by a superposition of \(N_d\) acoustic monopoles,
\begin{equation}
\label{eq:dirac_source}
u_n = \sum_{j=1}^{N_d} \coeff{u}_{j,n} \delta_{\hat x_j},
\end{equation}
where \(\coeff{u}_{j,n} \in \C\) and \(\hat x_j \in \Oc\), where
$\Oc \subset \widebar{\Omega}$ is a set containing all possible source locations.
We suppose that for a finite number of observation points
\(\Xi = \set{x_m \;|\; m=1,\dotsc,M}\) pressure values \(p^m_d \in \C^N\) of~\eqref{eq:helm} are
given (in the form of noisy recordings at \(M\) microphones, i.e.~$p^m_d = p(x_m) + z^m$,
$z^m\in \C^N$). Based on these observations the number of point sources $N_d$, the
positions $\hat{x}_j \in \Oc$ and coefficients
$\coeff{u}_j \in \C^N$ are to be reconstructed. Inverse problems of this kind are of
great importance in engineering applications such as
beamforming~\cite{VeenBuckley:1988,Suzuki:2011,XenakiGerstoftMosegaard:2014,XenakiGerstoft:2015}.
For instance, one is interested in locating a
source of noise pollution using processed data captured by a microphone array.

Due to the fact that we have only partial observations of the acoustic pressure, the problem is
under-determined, and therefore ill-posed. Thus we solve it based on a
regularized least-squares formulation. We follow the approach of
\cite{BrediesPikkarainen:2013} and consider the following convex problem:
\begin{equation}
\label{eq:problem_convex}
\begin{aligned}
\min_{u \in \mathcal{M}_w(\Oc,\C^N)} \quad&
\frac{1}{2} \sum_{m=1}^M \abs{p(x_m) - p^m_d}^2_{\C^N} + \alpha\norm{u}_{ \mathcal M_w(\Oc,\C^N)}, \\
\text{subject to } \quad&\eqref{eq:helm},
\end{aligned}
\end{equation}
where $p = (p_1,\ldots, p_N)$ and $u = (u_1, \ldots, u_N)$. In this problem, the solution
of \eqref{eq:problem_convex} is searched in the space of \(\C^N\)-valued Radon measures which satisfy
\[
\norm{u}_{\mathcal M_w(\Oc,\C^N)}= \int_{\Oc} \abs{w u'}_{\C^N}\de\abs{u}<\infty
\]
for a vector-valued weighting function $w\colon \Oc \to \C^N$. Here, the point-wise product
$w(x)u'(x) = (w_1(x)u'_1(x),\dotsc, w_N(x)u'_N(x))$ should be understood in the sense of the
Hadamard-product. The regularization functional promotes the
sparsity of the support of the solution in $\Oc$ independent of the frequency components
(also referred to as group or directional sparsity~\cite{HerzogStadlerWachsmuth:2012});
see~\cite{BrediesPikkarainen:2013,RouxBoufounosKangHershey:2013}. More concretely, it promotes solutions of the
structure~\eqref{eq:dirac_source}.

Note that, a more direct
reconstruction approach would be the solution of the problem
\begin{equation}
\label{eq:problem_nonconvex}
\begin{aligned}
\min_{x_j \in \Oc, \coeff{u}_j \in \C^N} \quad&
\frac{1}{2} \sum_{m=1}^M \abs{p(x_m) - p^m_d}^2_{\C^N}
 + \alpha\sum_{j=1}^{N_d}\abs{w(x_j)\coeff{u}_j}_{\C^{N}},
\\
\text{subject to } \quad&
\eqref{eq:helm}\quad\text{with } u = (u_n)_n \text{ as in } \eqref{eq:dirac_source},
\end{aligned}
\end{equation}
where the number of sources \(N_d\) is fixed, but can be regarded as an additional
discrete optimization variable.
Since the locations \(x_j\) are now considered optimization variables,
this is a non-convex finite-dimensional optimization problem with constraints \(x_j \in
\Oc\), which complicates the numerical solution. At first glance, the
problem formulation~\eqref{eq:problem_convex} seems to be more general than
\eqref{eq:problem_nonconvex} since we discard the structural assumption on the source $u$
by considering general Borel measures. However, the existence of minimizers to
\eqref{eq:problem_convex} of the form~\eqref{eq:dirac_source} can be
guaranteed for \(N_d \leq 2NM\).
Hence, if the number of sources \(N_d\) is left free, both problems are essentially
equivalent, i.e.~we can obtain a solution
to the nonconvex problem~\eqref{eq:problem_nonconvex} by solving the convex
version~\eqref{eq:problem_convex}.

The objective of this work is to provide a systematic theoretical development of the above
recovery approach, including analysis of the problem, conditions for recovery, and
algorithmic solution and numerical discretization strategies.
In the case $w\equiv1$ the analysis of the problem~\eqref{eq:problem_convex} relies on the
assumption that the observation points and the control set $\Oc$ are separated from each
other. However, by using weighting functions in the regularization functional with specific
properties this restriction can be overcome. Moreover, an optimal choice of the weight
function is shown to lead to improved theoretical and practical properties of the approach.

\subsection{Related works}

The analysis of the recovery approach is based on the analysis of the noise-free case,
which leads to the corresponding minimum norm problem
\begin{equation}
\label{eq:problem_convex_min_norm}
\begin{aligned}
\min_{u \in  \mathcal{M}_w(\Oc,\C^N)}\quad&\norm{u}_{\mathcal M_w(\Oc,\C^N)}\\
\text{subject to } \quad&p(x_m)=p^m_d,\quad m=1,\ldots,M,
\end{aligned}
\end{equation}
where $p$ is the solution of~\eqref{eq:helm};
see, e.g., \cite{BrediesPikkarainen:2013,CandesFernandezGranda:2013,DuvalPeyre:2015}. For
$w\equiv1$ it is shown in \cite{BrediesPikkarainen:2013} that the
solutions of~\eqref{eq:problem_convex} converge for $\alpha \to 0$ and
$\abs{z}_{\C^{MN}}^2/\alpha\to 0$ to a solution of \eqref{eq:problem_convex_min_norm} in
the weak-star sense; see also~\cite{BurgerOsher:2004,HofmannKaltenbacherPoeschlScherzer:2007}.
This can be carried over to the weighted case easily.  We also note that the inverse
problem under consideration
can be interpreted as a deconvolution problem for measures involving the Green's function
corresponding to \eqref{eq:helm} as convolution kernel. Problems of this form have been
studied recently
in~\cite{CandesFernandezGranda:2014,AzaisDeCastroGamboa:2015,DuvalPeyre:2015,CandesFernandezGranda:2013}. In
\cite{CandesFernandezGranda:2014} the recoverability of an exact source from convolutions with the F{\'e}jer
kernel is proven under the assumption that the
exact point sources are sufficiently well separated from each other. Concerning the use of a
non-constant weight $w \neq 1$ we refer to~\cite{SchiebingerRobevaRecht:2018}.
By an appropriate choice of
the weighting function the authors prove an exact recoverability result for a general 
deconvolution problem on a one-dimensional domain without requiring a minimum
separation distance between the exact source points. However, these results are not
directly applicable in our setting due to the more complicated structure of the
convolution kernel under consideration.

Robustness with respect to noise has been investigated
in~\cite{AzaisDeCastroGamboa:2015,CandesFernandezGranda:2013,DuvalPeyre:2015}.
In~\cite{DuvalPeyre:2015} it is shown that a strengthened source condition for small enough 
noise level $\delta$ and regularization parameter $\alpha$ the solution of
\eqref{eq:problem_convex} is unique and consists of the same number of point sources as
the exact solution. Convergence rates for coefficients and positions of the
reconstructed source to the exact coefficients and positions are derived.

Moreover, we mention that, after discretization on a finite grid,
the inverse problem under consideration
corresponds to an inverse problem involving an over-complete dictionary; see, e.g.,
\cite{Tropp:2004}. The dictionary is given by point-evaluations of the Green's functions
of~\eqref{eq:helm}. In the
noise-free case such problems are often solved by a problem formulation corresponding
to~\eqref{eq:problem_convex_min_norm} (Basis Pursuit), and
in the noisy case a problem corresponding to~\eqref{eq:problem_convex} is
solved (LASSO). In most of the literature concerning over-complete
dictionaries it is assumed that the entries of the dictionary have unit norm, in
order to prevent bias in the dictionary. In our problem this is not the case. However, a
particular form of the weight function $w(x)$ leads to reweighted versions of the problems
\eqref{eq:problem_convex} and \eqref{eq:problem_convex_min_norm} in the variable $v=wu$, which
have a dictionary with entries of unit norm.

Finally, concerning the discretization of the PDE-constrained optimization problem, a problem
similar to~\eqref{eq:problem_nonconvex} has been
proposed in~\cite{BermudezGamalloRoriguez:2004} for a fixed number \(N_d\)
and FE-discretizations have been analyzed (cf.\ also \cite{DokmanicVetterli:2012}).
Concerning the regularity and numerical analysis for sparse control problems with
measures, in combination with different PDEs, we also refer
to~\cite{CasasClasonKunisch:2012,CasasVexlerZuazua:2014,KunischPieperVexler:2014,KunischTrautmannVexler:2016}.

\subsection{Contribution}
Concerning the analysis of~\eqref{eq:problem_convex}, we first focus on the
case \(w \equiv 1\), which is complicated by the presence of point-wise sources (which
lead to unbounded solutions) with point-wise observations of the solution.
Nevertheless, based on regularity results for~\eqref{eq:helm}, we show
that~\eqref{eq:problem_convex} and~\eqref{eq:problem_convex_min_norm} are
well-posed if the sources are restricted to some compact set \(\Oc\) which does not
contain the set of observation points \(\Xi\).
Note that this implies \(\dist(\Oc,\Xi) > 0\).
While this may not seem like a severe restriction, it introduces
additional questions: On the one hand, a large distance restricts the
possible location from where sources can be recovered.
On the other hand, for a too small distance the
problem favors sources close to the observation points, which
introduces undesirable reconstruction artifacts. In fact, it can be proven that the problem
with $w\equiv1$  has no solutions if \(\Xi\cap \Oc \neq
\emptyset\); see Proposition~\ref{prop:counter_example_weight}.
By introduction of a weight function $w$ that is unbounded in the observation
points, well-posedness of \eqref{eq:problem_convex}
can be shown for arbitrary \(\Oc\); see Section~\ref{sec:weighted_norm}. Concerning the
structure of the solutions, we show both problems always admit solutions of the
form~\eqref{eq:dirac_source} with \(N_d \leq 2NM\).

Clearly, not all sources of the form~\eqref{eq:dirac_source} can be recovered
by~\eqref{eq:problem_convex}. However, we
show that all minimum norm solutions of~\eqref{eq:problem_convex_min_norm} fulfill a
\emph{source condition}, which allows
us to deduce convergence rates for the convergence of the solutions of
\eqref{eq:problem_convex} to solutions of \eqref{eq:problem_convex_min_norm} for vanishing
noise and appropriately chosen $\alpha$; see Section~\ref{sec:regularization}. Additionally,
we give numerical examples of recoverable and non-recoverable sources. Even in the simple case of
one unknown source, recoverability can fail unless an appropriate weight is employed. Moreover,
numerical experiments suggest that the use of specific weights increases the number of
recoverable sources. This is confirmed by statistical test involving randomly chosen
positions and coefficients of the exact sources. In the
case of a single point source we are able to prove that the exact source is the unique
solution of~\eqref{eq:problem_convex_min_norm} when using a specific weighting function
and under additional assumption on the forward operator; see
Proposition~\ref{prop:exact_reconstruction}.

Concerning the numerical solution of~\eqref{eq:problem_convex}, we adopt
the algorithmic strategy proposed in~\cite{BrediesPikkarainen:2013} (see also
\cite{BoydRechtSchiebinger:2017}),
which operates on the linear span of Dirac delta functions and combines point-insertion
and removal steps. Moreover, a function space convergence theory is available, which
bounds the number
of necessary steps to obtain a prescribed accuracy in the functional value. We augment
the procedure by an additional step which guarantees that the size
of the support of the iterations of the algorithm can not grow beyond \(2NM\). In
\cite{BrediesPikkarainen:2013} Dirac deltas are removed using one step of a
proximal gradient method applied to~\eqref{eq:problem_convex} for the magnitudes with
fixed positions. To further promote the sparsity of the iterates, this finite dimensional
non-smooth optimization problem is resolved in every iteration (cf.\ also
\cite{BoydRechtSchiebinger:2017}) by means of a globalized semi-smooth Newton method.
Additionally, we employ a discretization
of~\eqref{eq:problem_convex} with finite elements for
\(p\) and Dirac delta functions in the grid nodes. Although this
transforms~\eqref{eq:problem_convex} into a finite dimensional optimization problem
(amenable to a wide range of optimization algorithms), the function space
analysis of the presented algorithm ensures that the number of iterations stays
(uniformly) bounded for arbitrarily fine meshes.

This paper is organized in the following way. In Section~\ref{sec:existence_regularity} we
establish regularity properties of the Helmholtz equation needed for the analysis of the
optimization problem. Section~\ref{sec:opt_cont_prob} is devoted to the analysis of the
problem with $w\equiv 1$. Section~\ref{sec:weighted_norm} is concerned
with the weighted problem for a general weight. In Section~\ref{sec:regularization}, the
regularization properties of the reconstruction procedure are
investigated. Section~\ref{sec:optimization} describes the optimization algorithm we use
for the solution of the measure-valued optimal control problem. Finally, in
Section~\ref{sec:numerics} we conduct several numerical experiments.

\subsection{Notation and conventions}
Throughout the paper we adopt the following conventions: The complex numbers \(\C\) are
regarded as a \(\R\)-linear vector space endowed with the inner product
\((z,v)_\C = \Re (z\bar{v}) = \Re(z)\Re(v) + \Im(z)\Im(v)\).
Correspondingly, we denote the inner product on the Hilbert space \(L^2(\Omega) = L^2(\Omega,\C)\) by
\[
(v,\phi)_\Omega = \int_{\Omega} \Re(v\bar{\phi}) \de x.
\]
This convention extends to all other inner products or duality pairings defined on derived
spaces. We identify the space of \(\C^N\)-valued vector measures as
\[
\M(\Oc, \C)^N \cong
\M(\Oc, \C^N) \cong \Cc(\Oc, \C^N)^*,
\]
where the second isomorphism is isometric if \(\Cc(\Oc, \C^N)\), the space of continuous
functions with values in \(\C^N\), is endowed with the norm
\(\norm{\phi}_{\Cc(\Oc, \C^N)} = \sup_{x\in\Oc} \abs{\phi(x)}_{\C^N}\).
The duality pairing is defined by
\[
\pair{u, \phi} = \Re\left(\int_{\Oc} \bar{\phi} \de u\right) = \int_{\Oc} (u', \phi)_{\C^N} \de \abs{u}
= \sum_{n=1}^N \Re \left(\int_{\Oc} \bar{\phi}_n \de u_n \right),
\]
with the total variation measure \(\abs{u} \in \M^+(\Oc)\) (in the space of positive Borel
measures), the Radon-Nikodym derivative \(u' = \de u/\de\abs{u} \in L^1(\Oc,\C^N,\de\abs{u})\),
and \(u_n \in \M(\Oc)\) the signed real valued measures arising as the component measures
of \(u\). By \(C\) we denote a generic constant, which has different values at different
appearances.

\section{Analysis of the Helmholtz equation}
\label{sec:existence_regularity}

Let $\Omega\subset \mathbb R^d$, $d\in\set{2,3}$ be a bounded, convex, and polytopal domain.
Following~\cite{BermudezGamalloRoriguez:2004}, we assume that the boundary
is of the form $\Gamma = \partial\Omega = \Gamma_N \cup \Gamma_Z$ where
\(\GN = \cup_{m} \overline{\Gamma_m}\) can be written as the union of some subset of plane faces of
\(\Gamma\) and that \(\GZ = \partial\Omega\setminus\GN\).
We note that these assumption on the boundary could be relaxed considerably, at the
expense of making the following arguments more technical; see
Remark~\ref{rem:more_general_domain} below. For simplicity, we follow the setting
of~\cite{BermudezGamalloRoriguez:2004}. Moreover, we
assume that \(\GZ\) has positive measure, which is needed to ensure unique solvability for
all wave numbers. We denote the characteristic function of \(\GZ\) by \(\chi_{\GZ}\colon
\Gamma \to \set{0,1}\).

Denote by $\Oc \subset\widebar\Omega$ the control set, which
is required to be closed (and therefore compact).
The state equation problem reads as: find $p = (p_1, \ldots, p_N)$ for \(n \in
\set{1,2,\ldots,N}\) where $p_n\colon \Omega \rightarrow \C$ solves
\begin{equation}\label{eq:state_n}
  \begin{curlyeq}
    -\Lap p_n - \kay_n^2 p_n &= u_n\rvert_{\Omega}  && \text{in } \Omega,\\
    \partial_{\nu}p_n - \i\kappa_n \chi_{\GZ} p_n &= u_n\rvert_{\Gamma} && \text{on } \Gamma,
  \end{curlyeq}
\end{equation}
$\kay_n > 0$ are real numbers and $u = (u_1, u_2, \ldots, u_N) \in \M(\Oc,\C^N)$ is a vector
measure. Note that, in the interest of generality, we allow the measure to be supported on
the boundary. These contributions of the measure appear in the boundary conditions, but
are included in the weak formulation given below in a natural way.

In this section, we assume without restriction that \(N=1\) and suppress the dependency on
\(n\) of \(\kay\), \(\kappa\), \(u\) and \(p\). The general case of the results follows directly from the
(complex) scalar case.

\begin{definition}[Very weak solutions for~\eqref{eq:state_n}]
  \label{def:transposition}
  Let \(u \in \M(\bar\Omega)\) be a complex valued measure.
  A complex valued function $p \in L^2(\Omega)$ is said to be a solution \emph{by transposition}
  to~\eqref{eq:state_n} if it satisfies
  \begin{equation}\label{eq:duality}
    (p,q)_\Omega = \pair{u, r}  \quad\text{for all } q\in L^2(\Omega),
  \end{equation}
  where $r\in H^2(\Omega)$
  is the solution to the dual problem
  \begin{equation}\label{eq:dual_state}
    \begin{curlyeq}
      -\Lap r - \kay^2 r &= q &&\text{in }\Omega,\\
      \partial_{\nu}r + \i\bar{\kappa} \chi_{\GZ} r &= 0, &&\text{on }\Gamma.
    \end{curlyeq}
  \end{equation}
  Note, that the duality pairing \(\pair{u, r}\) is well defined due to the continuous
  embedding \(H^2(\Omega) \hookrightarrow \Cc(\overline{\Omega})\) for spatial dimension
  \(d \leq 3\).
  It can be shown that the solution by transposition also satisfies the following \emph{very weak} formulation:
  \begin{multline}\label{eq:very_weak}
    - (p, \Lap \phi+\kay^2\phi)_\Omega
    = \pair{u, \phi} \\
      \quad\text{for all } \phi \in H^2(\Omega) \;\;\text{with}\;
      \partial_{\nu}\phi + \i\bar{\kappa} \chi_{\GZ} \phi = 0 \;\text{on } \Gamma.
  \end{multline}
\end{definition}

\begin{theorem}\label{thm:L2sol}
  For any $u\in \M(\Oc)$, there exists a unique very weak solution $p\in L^2(\Omega)$ to
  \eqref{eq:state_n} and there holds
  \begin{equation*}
    \norm{p}_{L^2(\Omega)} \leq C\norm{u}_{\M(\Oc)}.
  \end{equation*}
\end{theorem}
\begin{proof}
  This result is proven by the method of transposition as in Definition~\ref{def:transposition}
  (cf.~\cite{LionsMagenes:1972}) using the $H^2(\Omega)$-regularity of the unique
  solution of the dual equation~\eqref{eq:dual_state};
  see~\cite[Theorem~3.3]{BermudezGamalloRoriguez:2004}. For the underlying regularity
  theory for the Neumann problem on convex polytopal domains we refer also
  to~\cite{Dauge:1988,Grisvard:1985}.
\QED
\end{proof}

\begin{lemma}\label{lem:Wones}
  The very weak solution \(p \in L^2(\Omega)\) from Theorem~\ref{thm:L2sol} has the improved
  regularity \(p \in W^{1,s}(\Omega)\) for any \(s < d/(d-1)\) and there holds
  \begin{equation*}
    \norm{p}_{W^{1,s}(\Omega)} \leq C\norm{u}_{\M(\Oc)}.
  \end{equation*}
\end{lemma}
\begin{proof}
This result can be proved by using a H{\"o}lder continuity result for the dual
equation~\eqref{eq:dual_state} with weak formulation
\[
(\nabla \phi, \nabla r)_\Omega - (\kay^2\phi,r)_\Omega
 - (\i\bar{\kappa}\chi_{\GZ} \phi, r)_{\Gamma} = \pair{q, \phi}_{W^{-1,s'}(\widebar\Omega),W^{1,s}(\Omega)}
\]
with data
\(q \in W^{-1,s'}(\widebar\Omega) = (W^{1,s}(\Omega))^*\), i.e.,
with \(1/s' + 1/s = 1\) and the corresponding a priori  estimate
\[
  \norm{r}_{\Cc(\widebar\Omega)} \leq C \norm{q}_{W^{-1,s'}(\widebar\Omega)}.
\]
Such a result can be found, e.g., in~\cite{GriepentrogRecke:2001} (cf.\ also~\cite{Droniou:2000a}).
To apply the result, which is derived for real systems of equations, we split the solution into
real and imaginary part, apply~\cite[Theorem~7.1~(i)]{GriepentrogRecke:2001}, and use the
embedding properties of Sobolev-Campanato spaces; see,
e.g.,~\cite[Theorem~2.1~(i)]{GriepentrogRecke:2001}.
\QED
\end{proof}

Based on the previous existence and regularity results, certain observations of
the state solution \(p\) (e.g., in \(L^2(\Omega)\) or \(L^s(\Gamma)\) for \(s \leq
d/(d-1)\)) are possible. To obtain the
continuity of point evaluations, we use the smoothness of the solution away from the
support of the source \(u\). First we analyze the fundamental solutions.

\begin{lemma}\label{lem:Green}
  Let $y \in \widebar\Omega$. Then the very weak solution $G^y$ to the
  equation
  \begin{equation}\label{eq:source_y}
    \begin{curlyeq}
      - \Lap G^y - \kay^2 G^y &= \delta_{y}\rvert_{\Omega},  &&\text{in } \Omega,\\
      \partial_{\nu}G^y - \i\kappa \chi_{\GZ} G^y &= \delta_{y}\rvert_{\Gamma}, &&\text{on } \Gamma,
    \end{curlyeq}
  \end{equation}
  satisfies for $\eps > 0$ the estimate
  \begin{equation}\label{estimate_Green}
    \norm{G^y}_{H^2(\Omega\setminus \widebar{B_{\eps}}(y))} \leq C(\eps),
  \end{equation}
  where $B_\eps(y)$ is the $\eps$-ball around $y$, and \(C\) depends continuously
  on \(\eps\).
\end{lemma}
\begin{proof}
  We follow standard arguments based on a smoothed indicator function. For completeness,
  we give a short sketch of the proof.
  Multiply \(G^y\) with a weight function \(\zeta_\eps \in
  \Cc^\infty_c(\widebar\Omega\setminus \widebar{{B}_{\eps/2}}(y))\), such that
  \(\zeta_\eps(x) = 1\) for \(x \in \widebar\Omega\setminus \widebar{{B}_{\eps}}(y)\). Now, by
  the chain rule and~\eqref{eq:source_y}, the product
  \(G_{\zeta}^y = \zeta_\eps G^y\) fulfills
  \begin{equation}\label{eq:xitilde}
    \begin{curlyeq}
      - \Lap G_\zeta^y - \kay^2G_\zeta^y
      &= -\Lap \zeta_\eps G^y - 2\nabla \zeta_\eps\nabla G^y, &&\text{in }\Omega,\\
      \partial_{\nu}G_{\zeta}^y - \i\kappa \chi_{\GZ} G_{\zeta}^y
      &=  G^y \partial_{\nu} \zeta_\eps - \i\kappa \chi_{\GZ} G^y \zeta_\eps, &&\text{on }\Gamma.
     \end{curlyeq}
   \end{equation}
  Now, we use the facts that \(G^y \in L^2(\Omega)\) with Theorem~\ref{thm:L2sol} and
  \(\nabla G^y \in L^s(\Omega)\) for \(s < d/(d-1)\) arbitrary with Lemma~\ref{lem:Wones}.
  With the trace theorem it additionally follows \(G^y\rvert_{\Gamma} \in
  L^{s}(\Gamma)\). By the Sobolev embedding in dimensions \(d \leq 3\), we obtain
  \(\nabla G^y \in H^{-1}(\Omega)\) (choose \(s > 2d/(d+2)\)) and
  \(G^y\rvert _{\Gamma} \in H^{-1/2}(\Gamma)\) (choose \(s > 2-2/d\)).
  Together with  \(\norm{\nabla^2 \zeta_\eps}_{L^{\infty}(\bar\Omega)} \leq C \eps^{-2}\) it follows
  now from a classical result for~\eqref{eq:xitilde} that \(G_{\zeta}^y \in H^1(\Omega)\) with
  \(\norm{G_{\zeta}^y}_{H^1(\Omega)} \leq C/\eps^2\). By the trace theorem, it follows that
  \(G_{\zeta}^y\rvert_{\Gamma} \in H^{1/2}(\Gamma)\). Now, we introduce \(G_{\zeta^2}^y = \zeta_\eps
  G^y_\zeta\) and repeat the argument to derive regularity of \(G_{\zeta^2}^y\) from the previous
  results for \(G_{\zeta}^y\).
  By a \(H^2\) regularity result (see, e.g.,
  \cite[Theorem~3.3]{BermudezGamalloRoriguez:2004}), we obtain \(G_{\zeta^2}^y \in H^2(\Omega)\),
  with norm bounded by \(C/\eps^{-4}\).
  Since by construction \(G^y_{\zeta^2}(x) = G^y(x)\) for all \(x \in \Omega\) with
  \(\abs{x-y} \geq \eps\), we obtain~\eqref{estimate_Green}.
\QED
\end{proof}

\begin{lemma}\label{lem:improved}
  Let \(\Nn_\eps(\Oc) = \set{x\in\Omega \;|\; \operatorname{dist}(x,\Oc) < \eps}\).
  The solution $p$ to~\eqref{eq:state_n} belongs to $\Cc(\widebar{\Omega}\setminus\Nn_\eps(\Oc))$ for
  all $\eps > 0$ together with
  \begin{equation*}
    \norm{p}_{\Cc(\widebar{\Omega}\setminus\Nn_\eps(\Oc))} \leq C(\eps)\norm{u}_{\M(\Oc)}.
  \end{equation*}
\end{lemma}
\begin{proof}
  We approximate $u$ by a sequence of finite sum of Dirac delta measures, i.e.,
  there exists a sequence $u_K \rightharpoonup^* u$ in $\M(\Oc)$ with \(\norm{u_K}_{\M(\Oc)} \leq
  \norm{u}_{\M(\Oc)}\) and
  \begin{equation*}
    u_K = \sum_{k=1}^{K}\coeff{u}_k\delta_{y_k}
  \end{equation*}
  with $\coeff{u}_k\in \C$ and $y_k \in \Oc$.
  By linearity, we have for the unique solution $p_K$ of~\eqref{eq:state_n} corresponding to
  \(u_K\) that \(p_K = \sum_{k=1}^K \coeff{u}_k G^{y_k}\) where $G^{y_k}$ is the solution of \eqref{eq:source_y} with $\delta_{y_k}$ in place of $\delta_y$.
  For every \(\eps > 0\) there exists a \(C = C(\eps)\) with
  \begin{equation*}
    \norm{p_K}_{L^2(\Omega)} + \norm{p_K}_{H^2(\Omega\setminus\widebar{{\Nn}_{\eps}}(\Oc))}
    \leq  C\sum_{k=1}^{K}\abs{\coeff{u}_k} = C\norm{u_K}_{\M(\Oc)} \leq C\norm{u}_{\M(\Omega_c)}
  \end{equation*}
  using Theorem~\ref{thm:L2sol} and Lemma~\ref{lem:Green}.
  Hence, there exists a function $p \in L^2(\Omega) \cap H^2(\Omega\setminus\widebar{{\Nn}_{\eps}}(\Oc))$ such that
  \begin{equation*}
    p_K \rightharpoonup p \quad \text{in} \quad  L^2(\Omega) \cap H^2(\Omega\setminus\widebar{{\Nn}_{\eps}}(\Oc))
  \end{equation*}
  up to a subsequence. Using this weak convergence and $u_K \rightharpoonup^* u$ in
  $\M(\Oc)$ we can pass to the limit $K \rightarrow \infty$ to obtain that
  $p$ is the very weak solution to the problem \eqref{eq:state_n} and the estimate
  \begin{equation*}
    \norm{p}_{L^2(\Omega)} + \norm{p}_{H^2(\Omega\setminus\widebar{{\Nn}_{\eps}}(\Oc))} \leq C(\eps)\norm{u}_{\M(\Omega)}.
  \end{equation*}
  holds for some $C(\eps) > 0$. Thus, the proof is complete when
  we use the embedding $H^2(\Omega\setminus\widebar{{\Nn}_{\eps}}(\Oc)) \hookrightarrow
  \Cc(\widebar{\Omega}\setminus\Nn_{\eps}(\Oc))$ for dimensions $d\leq 3$.
\QED
\end{proof}

Clearly, the same regularity results also hold for the dual equation,
\begin{equation}\label{eq:dual_source_y}
  \begin{curlyeq}
    - \Lap \bar{G}^y - \kay^2 \bar{G}^y &= \delta_{y},  &&\text{in } \Omega,\\
    \partial_{\nu}\bar{G}^y + \i\bar{\kappa}\chi_{\GZ} \bar{G}^y &= 0, &&\text{on } \Gamma.
  \end{curlyeq}
\end{equation}
Note that the only difference between~\eqref{eq:source_y}
and~\eqref{eq:dual_source_y} occurs in the boundary conditions on \(\GZ\). It is
therefore easy to see that the solutions to~\eqref{eq:source_y} are~\eqref{eq:dual_source_y} are
the same up to complex conjugation, which justifies the notation \(\bar{G}^y\).
In the case \(y \in \Omega\) (and not on \(\Gamma\)), we can give a more precise
description of the nature of the singularity. We will need this for the adjoint equation
in section~\ref{sec:weighted_norm}.

\begin{proposition}
\label{prop:splitting}
Let $y \in \Omega$. Then the very weak solution $\bar{G}^y$ to the dual
equation~\eqref{eq:dual_source_y} can be written as \(\bar{G}^y(x) = \bar{\Phi}^y(x) +
\bar{\xi}^y(x)\) for \(x\in\widebar{\Omega}\), where
\begin{equation}
\label{eq:fundamental_sol}
\Phi^y(x) =
\oldphi_\kay(\abs{x-y}) =
  \begin{cases}
    (i/4){H_0^{(1)}(\kay\abs{x-y})}
    & \text{for } d = 2, \\
    {\exp(\i\kay\abs{x-y})}/({4\pi\abs{x-y}}) & \text{for } d = 3,
  \end{cases}
\end{equation}
is a fundamental solution of the free space Helmholtz equation
\begin{equation}
\label{eq:free_space_helm}
  -\Lap \Phi^y - \kay^2 \Phi^y = \delta_y, \quad x\in \R^n,
\end{equation}
(fulfilling the Sommerfeld radiation condition),
and \(\xi^y \in H^2(\Omega)\) is the solution to \eqref{eq:xi}.
The special function \(H_0^{(1)}\) is the Hankel function of the first kind; see,
e.g.,~\cite[Section~3.4]{ColtonKress:2013}.
\end{proposition}
\begin{proof}
We follow~\cite{BermudezGamalloRoriguez:2004}.
First, we consider a fundamental solution
$\Phi^y$ to the Helmholtz equation in the whole domain~\eqref{eq:free_space_helm}.
In fact $\Phi^y$ can be written explicitly as in~\eqref{eq:fundamental_sol}; see,
e.g.,~\cite{ColtonKress:2013}. We will
use the facts that $\Phi^y \in C^{\infty}(\R^n\setminus \{y\})$ and
\(\norm{\Phi^y}_{\Cc^1(K)} \leq C(\operatorname{dist}(y,K)) \abs{K}\) for any
\(K \subset\subset \Omega\). Then $\bar{G}^y$
is a solution of~\eqref{eq:dual_source_y} if and only if $\bar{G}^y = \bar{\Phi}^y + \bar{\xi}^y$,
with $\xi^y$ satisfying
\begin{equation}\label{eq:xi}
  \begin{curlyeq}
    -\Lap \xi^y - \kay^2\xi^y &= 0, &&\text{in }\Omega,\\
    \partial_{\nu}\xi^y - \i\kappa\chi_{\GZ}\xi^y
    &= -\partial_{\nu}\Phi^y + \i\kappa\chi_{\GZ}\Phi^y, &&\text{on }\Gamma.
  \end{curlyeq}
\end{equation}
We have the following estimate for $\xi^y$ (see, e.g., \cite[Theorem 3.3]{BermudezGamalloRoriguez:2004}):
\begin{equation*}
  \norm{\xi^y}_{H^2(\Omega)}
  \leq C\left(\norm{\partial_{\nu}\Phi^y - \i\kappa\Phi^y}_{H^{1/2}(\GZ)}
    + \norm{\partial_{\nu}\Phi^y}_{H^{1/2}(\GN)}\right).
\end{equation*}
Thus, it follows directly
\(\norm{\xi^y}_{H^2(\Omega)} \leq C(\operatorname{dist}(y,\Gamma)).\)
\QED
\end{proof}

\begin{remark}
  \label{rem:more_general_domain}
  The $H^2$ regularity of $G^y$ in Lemma~\ref{lem:Green} (and of \(\xi^y\) in
  Proposition~\ref{prop:splitting}) uses the structural assumption on
  the polygonal domain, namely that the boundary conditions can only change on different
  plane faces of the boundary (based on the results in \cite{BermudezGamalloRoriguez:2004}).
  It is possible to relax this assumption, and consider more general domains $\Omega$ in two
  or three dimensions. We will comment on two possible options, which we however do not pursue
  here for the sake of brevity.

  \emph{H{\"o}lder-regularity}: By using the regularity results from, e.g.,
  \cite{GriepentrogRecke:2001,Droniou:2000a} (as in Lemma~\ref{lem:Wones}), which are
  valid for much more general configurations of the boundary, we can get continuous
  solutions without \(H^2\) regularity. The solution by transposition can be based on
  these regularity results directly;
  cf.~\cite{Stampacchia:1965,Troianiello:1987}. Additionally, Lemma~\ref{lem:Green} can be
  modified to show local H{\"o}lder-continuity, which again leads to the result of Lemma~\ref{lem:improved}.
  A similar comment applies to Proposition~\ref{prop:splitting}.

  \emph{Interior regularity}:
  If we introduce a \(\Omega'\subset\subset \Omega\), we can show alternative to
  Lemma~\ref{lem:Green} the result
  \(G^y \in H^2(\Omega' \setminus \overline{B_{\eps}}(y))\)
  without using any assumptions on the boundary beyond Lipschitz-continuity.
  The proof can be done as in Lemma~\ref{lem:Green}, by suitably modifying the smoothed
  indicator function. For interior regularity results of elliptic equations cf.\ also
  \cite[Theorem~47.1]{QS07} \cite[Theorems~9.11 and 9.13]{GilbargTrudinger:2015}.
  However, interior results do not allow to include point sources or pointwise
  observations on the boundary of the domain.
\end{remark}

\section{Analysis of the optimization problem}
\label{sec:opt_cont_prob}

We suppose that for some points $\{x_m\}_{m=1,2,\ldots,M} \subset \widebar\Omega \setminus
\Oc$ the acoustic pressure values $p_d^m \in \C^N$ are given.
We consider the following optimization problem:
\begin{align}
\label{eq:pre_opt_prob}
  & \min_{u\in \Mcon}\J(p, u) = \frac{1}{2} \sum_{m=1}^M\abs{p(x_m) - p_d^m}_{\C^N}^2 +
    \alpha\norm{u}_{\Mcon},\\
\label{eq:state}
  & \text{ subject to}\quad
    \begin{curlyeq}
      -\Lap p_n - \kay^2_n p_n &= u_n\rvert_{\Omega}, &&\text{in }\Omega,\\
      \partial_{\nu}p_n - \i\kappa_n\chi_{\GZ} p_n &= u_n\rvert_{\Gamma}, &&\text{on }\Gamma,
    \end{curlyeq} \quad n = 1,2,\ldots, N.
\end{align}
Since $x_m \notin \Oc$, there
exists $\eps_0>0$ such that $x_m \notin \Nn_{\eps_0}(\Oc)$ for all $m=1,2,\ldots,M$. Due to
Lemma~\ref{lem:improved} we can evaluate $p_n$ at $x_k$ and thus define the
control-to-observation operator
\[
\SO: \Mcon \rightarrow (\C^N)^M \quad \text{as} \quad \SO u = (p(x_1), p(x_2), \ldots, p(x_M)).
\]
We introduce the reduced optimal control problem
\begin{equation}\label{eq:reduced_cost}\tag{$P_\alpha$}
  \begin{aligned}
    &\min_{u\in \Mcon} j(u)= \frac{1}{2}\sum_{m=1}^{M}\abs{(\SO u)_m - p_d^m}_{\C^N}^2 + \alpha\norm{u}_{\Mcon},
  \end{aligned}
\end{equation}
which is clearly equivalent to~\eqref{eq:pre_opt_prob}--\eqref{eq:state}.

We will see that \(\SO\) can alternatively be defined as the dual of a linear bounded
operator \(\SO^*\), to be introduced below.
\begin{lemma}\label{lem:weak_to_strong}
  $u_n \rightharpoonup^* u$ in $\Mcon$ implies \(\SO u_n \rightarrow \SO u\) in \(\C^{NM}\).
\end{lemma}
By established arguments, we obtain the following basic existence result.
\begin{proposition}\label{prop:exist_opt}
  The problem~\eqref{eq:reduced_cost} has an optimal solution $\widehat{u}$.
\end{proposition}

To derive optimality conditions, we consider the adjoint equation,
\begin{equation}\label{eq:dual_prob}
  \begin{curlyeq}
    -\Lap \xi_n - \kay^2_n\xi_n
    &= \textstyle\sum_{\set{m \;|\; x_m \in \Omega}} q_{n,m}\delta_{x_m}, &&\text{in }\Omega,\\
    \partial_{\nu}\xi_n + \i\bar{\kappa}_n\chi_{\GZ}\xi_n
    &= \textstyle\sum_{\set{m \;|\; x_m\in\Gamma}} q_{n,m}\delta_{x_m}, &&\text{on }\Gamma,
  \end{curlyeq}
\quad n = 1,2,\ldots, N,
\end{equation}
for given $q\in \C^{NM}$. We denote the by \(\SO^*\) the operator that maps a given \(q\)
to the restriction \(\xi\rvert_{\Oc}\), where $\xi = (\xi_1, \ldots, \xi_N)$ is the corresponding solution
to~\eqref{eq:dual_prob}.
\begin{proposition}\label{prop:SO}
  The linear operator $\SO^*\colon \C^{NM} \to \Cc$ is bounded.
\end{proposition}
\begin{proof}
  First we note that the equation~\eqref{eq:dual_prob} has a measure right-hand
  side. However, since $x_m \in \Omega\setminus \Nn_{\eps_0}(\Oc)$ for all $m=1,2,\ldots, M$, we
  have $\xi_n\in \Cc(\Nn_{\eps_0}(\Oc))\subset \Cc(\Oc)$ thanks to Lemma \ref{lem:Green}. Thus the operator $\SO^*$ is well
  defined. The linearity of $\SO^*$ is trivial.  The boundedness of $\SO^*$ follows with
  linearity from Lemma~\ref{lem:Green}.
\QED
\end{proof}
\begin{proposition}\label{prop:SO_dual}
  The operator $\SO$ is the dual of the operator $\SO^*$, that is
  \begin{equation}\label{eq:dual}
    \inner{\SO u, q} = \pair{u, \SO^* q}
    = \sum_{m=1}^M \sum_{n=1}^N \pair{u_n, \bar{G}_n^{x_m}q_{n,m}}
  \end{equation}
  for all $q\in \C^{NM}$ and all $u \in \Mcon$, where \(\bar{G}_n^{x_m}\) is defined
  in~\eqref{eq:dual_source_y} with \(\kay = \kay_n\).
\end{proposition}
\begin{proof}
  Similar to Lemma~\ref{lem:improved}, we approximate $u$ by a sequence $u_K$ of the form
  \(u_K = \sum_{k=1,\ldots,K}\coeff{u}_k\delta_{y_k}\).
  From \cite[Theorem 7.2]{BermudezGamalloRoriguez:2004}, with a slight modification, we have for all $K$ that
  \begin{equation*}
    (\SO u_K, q) = \pair{u_K, \SO^* q}.
  \end{equation*}
  Passing to the limit as $K\rightarrow \infty$ and using Lemma~\ref{lem:weak_to_strong} and
  $u_K \rightharpoonup^* u$ we get the desired result. The last equality in~\eqref{eq:dual}
  follows by linearity of \(\SO^*\).
\QED
\end{proof}

As in \cite{BrediesPikkarainen:2013}, the following optimality conditions system can be derived.
\begin{proposition}
  \label{prop:optimality_conditions}
  A measure $\widehat{u} \in \Mcon$ is a solution to~\eqref{eq:reduced_cost} if and only
  if $\widehat{\xi} = - \SO^*(\SO\widehat{u} - p_d)$ satisfies
  $\norm{\widehat{\xi}}_{\Ccon} \leq \alpha$
  and the polar decomposition $\de \widehat{u} = \widehat{u}'\de\abs{\widehat{u}}$, with
  $\widehat{u}'\in L^1(\Oc,\abs{\widehat{u}}, \C^N)$, satisfies
  \begin{equation*}
    \alpha \widehat{u}' = \widehat{\xi} \qquad \abs{\widehat{u}}\text{-almost everywhere}.
  \end{equation*}
  Thereby,
  $\supp\abs{\widehat{u}} \subset \set{x\in\Oc \;|\; \abs{\widehat{\xi}(x)}_{\C^N} = \alpha}$
  for each solution $\widehat{u}$.
\end{proposition}
\begin{proof}
  The proof follows the one of~\cite[Proposition~3.6]{BrediesPikkarainen:2013} with minor
  modification concerning the complex valued measure and the compact control domain.
\QED
\end{proof}

Since the operator \(S\) maps into a finite dimensional space, the solution set
of~\eqref{eq:reduced_cost} always contains linear combinations of Dirac delta
function. This can be seen by interpreting the corresponding dual problem as a
\emph{semi-infinite} optimization problem; see, e.g.,
\cite[Section~5.4]{BonnansShapiro:2000}. For the convenience of the reader, we provide an
independent exposition in Appendix~\ref{app:extremal_solutions}.
\begin{corollary}\label{cor:dirac_solutions}
There exists an optimal solution \(\widehat{u}\) to~\eqref{eq:reduced_cost} which consists
of \(N_d \leq 2NM\) point sources,
\[
\widehat{u} = \sum_{j=1}^{N_d} \widehat{\coeff{u}}_j \delta_{\widehat{x}_j}
\quad\text{where } \widehat{\coeff{u}}_j \in \C^N,\; \widehat{x}_j \in \Oc.
\]
\end{corollary}
\begin{proof}
This follows by combining Proposition~\ref{prop:krein_milman} with
Theorem~\ref{thm:extremal_dirac}. Note that it holds \(\dim \Ran S \leq \dim \C^{NM} = 2NM\), since
\(\C\) is regarded as a real vector space.
\QED
\end{proof}
\begin{corollary}
  \label{cor:optimality_conditions_dirac}
Any solution \(\widehat{u} = \sum_{j=1}^{N_d} \widehat{\coeff{u}}_j \delta_{\widehat{x}_j}\)
from Corollary~\ref{cor:dirac_solutions} is uniquely characterized by the optimality conditions
\begin{align*}
\norm{\widehat{\xi}}_{\Cc(\Oc,\C^N)} \leq \alpha, \qquad
\alpha\,\widehat{\coeff{u}}_j = \abs{\widehat{\coeff{u}}_j}_{\C^N}\widehat{\xi}(\widehat{x}_j),
\quad j \in \set{1,2,\ldots,N_d},
\end{align*}
where \(\widehat{\xi} = - \SO^* (\SO\widehat{u} - p_d)\) is the associated adjoint state.
\end{corollary}

\section{Weighted norm approach}
\label{sec:weighted_norm}

In practical computations, the recovery based on~\eqref{eq:reduced_cost}
succeeds only in some cases. In particular, there exist single point-sources which can not be
recovered even in the noise-free case. These cases occur when the boundary of the
set \(\Oc\) is close to the observation points (in which case several spurious sources
tend to be placed in these spots), or if the exact source is located in a spot with ``bad''
acoustical properties; see section~\ref{sec:numerics}.
Consider for a moment the case \(N = 1\), and assume that the exact source is given by
\(u^\star = \coeff{u}^\star \delta_{x^\star}\). The magnitude of the observed signal is given by
\[
\abs{S u^\star}_{\C^M} = \abs{\coeff{u}^\star}\sqrt{\sum_{m=1}^M \abs{G^{x_{m}}(x^\star)}^2}
  = \abs{\coeff{u}^\star} \hat{w}(x^\star)
\]
Thus, the magnitude of the observation for a unit source
originating from \(x \in \Omega\) is
described by the function \(\hat{w}\colon \Omega \to \R_+\cup\set{+\infty}\). Empirically,
the cases of non-identifiability coincide with the cases where \(\hat{w}(x^\star)\) is
small, compared to a global value such as, e.g., \(\max_{x\in\Oc}\hat{w}(x)\) or the mean
of \(\hat{w}\). However, if the magnitude of each source is computed in the weighted norm,
\[
\norm{u^\star}_{\M_{\hat{w}}(\Oc,C^N)} = \int_{\Oc} \hat{w}\de \abs{u^\star} =
\abs{\coeff{u}^\star} \hat{w}(x^\star),
\]
a source of unit size leads to an observation of unit size.

Motivated by this, we introduce for each frequency \(n\) a weight \(w^n\) and consider a
weighted problem:
\begin{equation}
\label{eq:weighted_problem}
\begin{aligned}
  & \min_{u\in \M_w(\Oc,\C^N)} \J_w(p, u) = \frac{1}{2} \sum_{m=1}^M \abs{p(x_m) - p_d^m}_{\C^N}^2
  + \alpha\norm{u}_{\M_w(\Oc,\C^N)}, \\
  & \text{subject to}\quad \eqref{eq:state}.
\end{aligned}
\end{equation}
In the interest of generality, we consider a
formulation with a general class of weights.
We will define the weighted norm $\norm{\cdot}_{\M_w(\Omega, \mathbb C^N)}$ for admissible
choices of the weight \(w\) below.

In a weighted problem formulation, the technical
condition on the observation points \(x_m \notin \Oc\) can be avoided. Therefore, in the
following, we only assume that \(\Oc \subset \Omega\) is closed in \(\Omega\).
Let $\Xi = \set{x_m \;|\; m=1,2,\ldots,M } \subset \Omega$ be the observation points (pairwise
distinct). For simplicity, we do not consider boundary observation in this section.
Note that the original problem~\eqref{eq:pre_opt_prob}--\eqref{eq:state} is not
necessarily well-posed in such cases.
\begin{proposition}
\label{prop:counter_example_weight}
Suppose that \(\Oc\) does not contain isolated points and that \(\dist(\Xi,\Oc) = 0\). Then,
without restriction, \(x_m\in\Oc\) for \(1\leq m \leq M_1 \leq M\) and \(x_m\notin\Oc\)
for \(m > M_1\). If \((p_d^m)_{m=1,\ldots,M_1} \in \C^{N M_1}\) is sufficiently
large (or \(M_1 = M\)),~\eqref{eq:pre_opt_prob}--\eqref{eq:state} does
not admit a solution.
\end{proposition}
\begin{proof}
For simplicity of notation, we assume without
restriction that \(N = 1\). Denote the optimization
problem~\eqref{eq:pre_opt_prob}--\eqref{eq:state} by \((P_{\textrm{orig}})\).
Consider first a modified
optimization problem, where we minimize
\[
J_{\mathrm{aux}}(p,u) = \frac{1}{2}\sum_{m=M_1+1}^M \abs{p(x_m) - p_d^m}^2_{\C}
+ \alpha\norm{u}_{\M(\Oc, \C)},
\]
subject to~\eqref{eq:state}. We denote the corresponding optimization problem by
\((P_{\mathrm{aux}})\). By similar arguments as in section~\ref{sec:opt_cont_prob},
there exists an optimal solution \(u_0 \in \M(\Oc, \C)\) to the modified problem
\((P_{\mathrm{aux}})\). By optimality, we obtain that
\[
  \norm{u_0}_{\M(\Oc, \C)} \leq \frac{\norm{(p_d^m)_{m=M_1+1,\ldots,M}}^2_{\C^{M-M_1}}}{2\alpha}.
\]
By continuity, it holds \(\norm{S u_0}_{\C^{M}} \leq
C\norm{(p_d^m)_{m=M_1+1,\ldots,M}}^2_{\C^{M-M_1}}\) for a generic \(C > 0\) and any
solution of \((P_{\mathrm{aux}})\).
Clearly, \(\min (P_{\mathrm{aux}}) \leq \inf(P_{\textrm{orig}})\).
In fact, equality holds: We show that for
\begin{equation}\label{eq:add}
 u^n = u_0 + \sum_{m=1}^{M_1} \coeff{u}^n_m \delta_{{x}^n_m},
\end{equation}
with appropriate \(\coeff{u}^n_m \to 0 \in \C\), \({x}^n_m \to x_m\) it holds
\(J(\SO{u}^n,{u}^n)  \to \min(P_{\mathrm{aux}})\).
To this purpose, we first fix \({x}^n_m \in \Oc\) with \(\abs{x_m-{x}^n_m} = r^n_m\),
for \(r^n_m > 0\) with \(r^n_m \to 0\) as \(n\to\infty\). Then, we consider the matrix
\(M^n \in \C^{M_1\times M_1}\), which results from the restriction of \(S\) to the span of
\(\delta_{{x}^n_m}\) in the domain space and to the first \(M_1\) observations in the
image space, that is
\[
M^n_{m,k} = G^{{x}^n_k}(x_m)\quad\text{for }m,k = 1,\ldots,M_1.
\]
Moreover, recalling the definition of $\oldphi_{\kay}$, see \eqref{eq:fundamental_sol}, we introduce the diagonal matrix
\[
D^n=\mathrm{diag}\left(\frac{1}{\abs{\oldphi_{\kay_1}(r^n_1)}}, \ldots, \frac{1}{\abs{\oldphi_{\kay_1}(r^n_{M_1})}}\right).
\]
By Proposition~\ref{prop:splitting} and the properties of the Green's functions, we derive that
\[
D^n M^n \to \Id_{\C^{M_1}} \quad\text{for } n \to \infty.
\]
Thus we have $\abs{\det(D^n)}\abs{\det(M^n)}=\abs{\det(D^nM^n)}>1/2$ for $n$ large
enough. Consequently, for \(n\) large enough the matrix $M^n$ is invertible. We can
therefore choose \(\coeff{u}^n = (\coeff{u}^n_1, \ldots, \coeff{u}^n_{M_1})\) to be the
solution of the system of equations \((M^n \coeff{u}^n)_m = p_d^m - (Su_0)_m\) for
\(m=1,\ldots,M_1\).
Therefore we have \((S(u^n))_m = p_d^m\) for \(m=1,\ldots,M_1\), thanks to~\eqref{eq:add}, and since
\(\abs{\oldphi_{\kay_1}(r_m^n)} \to \infty\) for \(n\to\infty\), it follows
additionally that \(\coeff u^n_m \to 0\) for $m=1,\ldots,M_1$. This shows that \(u^n
\to u_0\) strongly in \(\Mcon\) and \(\inf (P_{\textrm{orig}})\leq J(\SO u^n,u^n)
\to \min(P_{\mathrm{aux}})\leq \inf (P_{\textrm{orig}})\) for \(n\to\infty\).
Assume now that \((P_{\textrm{orig}})\) admits a solution \(\widehat{u}\). With
\(J(\SO\widehat{u},\widehat{u}) = \inf(P_{\textrm{orig}}) = \min(P_{\mathrm{aux}})\) we immediately deduce that
\((S\widehat{u})_m = p_d^m\) for \(m=1,\ldots,M_1\), and \(\widehat{u}\) also
solves~\eqref{eq:reduced_cost}.
However, choosing \(\norm{(p_d^m)_{m=1,\ldots,M_1}}_{\C^{M_1}}\) large enough contradicts
the bound \(\norm{S\widehat{u}}_{\C^M} \leq C\norm{(p_d^m)_{m=M_1+1,\ldots,M}}^2_{\C^{M-M_1}}\)
which follows from the optimality of
\(\widehat{u}\) for \((P_{\mathrm{aux}})\).
\QED
\end{proof}

Now, we introduce the class of admissible weight functions.
\begin{definition}[Admissible weights]
\label{def:admissible}
We call a family of weight functions
\(w^n\colon \Omega_c \rightarrow \R \cup \set{+\infty}\),
\(n \in \set{1,2,\ldots,N}\) admissible, if they fulfill the following properties:
\begin{enumerate}[label=\roman*)]
\item\label{it:w_pos}
  \(\inf_{x\in \widebar\Omega} w^n(x) > 0\),
\item\label{it:w_cont}
  \(w^n\) is upper semi-continuous and
  \(w^n\) restricted to \(\Oc \setminus \Xi\) is continuous.
\item\label{it:wG_cont}
  The function $G^{x_m}_n/w^n$ can be continuously extended from \(\Oc\setminus\Xi\) to
  $\Oc$.
\end{enumerate}
For admissible weights, we denote \([G^{x_m}_n/w^n](x_m) = \lim_{x\to x_m}
G^{x_m}_n(x)/w^n(x)\). The case \([G^{x_m}_n/w^n](x_m) = 0\) for all \(m\) is of special
interest.
\end{definition}
Due to the fact that \(\abs{G^{x_m}_n(x)} \to \infty\) for \(x \to x_m\), the
upper semi-continuity of \(w^n\) and Property~\ref{it:wG_cont} imply that \(w^n(x_m) = +\infty\).
Now, we construct functions \(w^n\) such that the above conditions hold. With regard
to the representation formula from Lemma~\ref{lem:Green}, we can take for instance the functions
\begin{align}
  \label{eq:example_w_free}
  w^n_{\mathrm{free}} = \sum_{m=1}^M \abs{\Phi_n^{x_m}}
\end{align}
In the following, we will again suppress the dependency on \(n\), for convenience of
notation.
\begin{proposition}
\label{prop:weight_free_admissible}
The weights given in~\eqref{eq:example_w_free} are admissible.
\end{proposition}
\begin{proof}
Property~\ref{it:w_pos} holds by the properties of the Green's functions. In both the two-
and three-dimensional case, the functions \(\abs{\Phi^{0}(x)}\) are radially symmetric and
monotonously decreasing towards zero for \(\abs{x} \to \infty\). Therefore,
\(\abs{\Phi^{x_m}(x)} = \abs{\Phi^{0}(x-x_m)}\) is uniformly bounded from below on
\(\Omega\) for all \(m\).
By a similar argument, property~\ref{it:w_cont} follows. It remains to verify~\ref{it:wG_cont}.
With Lemma~\ref{lem:Green}, we notice that
\[
\frac{G^{x_m}(x)}{w_{\mathrm{free}}(x)}
= \frac{\xi^{x_m}(x) + \Phi^{x_m}(x)}{w_{\mathrm{free}}(x)}
= \frac{\xi^{x_m}(x)}{w_{\mathrm{free}}(x)} + \frac{\Phi^{x_m}(x)}{w_{\mathrm{free}}(x)}
\]
with \(\xi^{x_m} \in H^2(\Omega)\).
Since \(\inf_{x\in \widebar\Omega} w_{\mathrm{free}}(x) > 0\) and for all
points \(\hat{x}\) where \(w_{\mathrm{free}}\) is discontinuous it holds
\(\lim_{x\to\hat{x}} w_{\mathrm{free}}(x) = +\infty\), the first term is continuous and we have
\[
\lim_{x\rightarrow x_m}\frac{\xi^{x_m}(x)}{w_{\mathrm{free}}(x)}=0.
\]
Furthermore, \(w_{\mathrm{free}}\) has the form
\(w_{\mathrm{free}}(x) = f_m(x) + \abs{\Phi^{x_m}(x)}\)
for an \(f_m\colon \Omega \to \R_+\cup\set{+\infty}\), which is finite and continuous in a
neighborhood of \(x_m\). Thus we have
\[
\lim_{x\rightarrow x_m}\frac{\Phi^{x_m}(x)}{w_{\mathrm{free}}(x)}
= \lim_{x\rightarrow x_m}\frac{\Phi^{x_m}(x)}{\abs{\Phi^{x_m}(x)}} = 1.
\]
In fact, for this, we use the concrete formulas for \(\Phi^{x_m}\); see
Lemma~\ref{lem:Green}. In the case \(d = 3\), it holds that \(\Phi^{x_m}(x) =
\exp(\i\kay\abs{x-x_m})/4\pi\abs{x-x_m}\), and the equality follows directly.
In the case \(d = 2\), we use that for \(t = \kay\abs{x-x_m}\) we have
\[
\Phi^{x_m}(x) =
\frac{i}{4} H_0^{(1)}(t)
= - \frac{1}{4} Y_0(t) + \frac{i}{4} J_0(t),
\]
where \(J_0\colon \R_+ \to \R\) and \(Y_0\colon \R_+ \to \R\) are the Bessel functions of
the first and second kind. It is known that \(J_0\) is continuous at \(t = 0\) and \(Y_0\) is
diverging towards \(+\infty\) at \(t = 0\);
see, e.g.,~\cite[Section~3.4]{ColtonKress:2013}.
\QED
\end{proof}
\begin{remark}
We verify that \(w^n_{\mathrm{free}}\) is independent of the wave number
\(\kay_n\) in three dimensions, since \(\abs{\Phi_n^{x_m}(x)} = 1/(4\pi\abs{x-x_m})\).
In two dimensions, the singularity of \(\abs{\Phi_n^{x_m}}\) is of same type as the
singularity of the Green's function of the Laplacian,
\(g(x) = - 1/(2\pi) \ln\abs{x-x_m}\), and \(\kay_n\) enters only in an additive
constant; see, e.g.,~\cite[Section~3.4]{ColtonKress:2013}. Therefore, we could
alternatively take the same weight for all \(n\).
\end{remark}
Other families of weight functions can be based on the Green's function on the domain. For
instance, they are given by
\begin{align}
  w^n_{\Omega,1} = \sum_{m=1}^M\abs{G_n^{x_m}}, \qquad
  w^n_{\Omega,2} = \sqrt{\sum_{m=1}^M\abs{G_n^{x_m}}^2} \label{eq:example_w_omega}.
\end{align}
Note that these weights depend on
the shape of \(\Omega\) and the wave number \(\kay_n\).
As for~\eqref{eq:example_w_free}, we obtain the admissibility of~\eqref{eq:example_w_omega}.
\begin{proposition}
\label{prop:weight_omega_admissible}
Suppose that for any \(n\) there exists no \(x \in \bar\Omega\), such that
\(G^{x_m}_n(x) = 0\) for all \(m\).
Then, the weights given in~\eqref{eq:example_w_omega} are admissible.
\end{proposition}
\begin{proof}
With Lemma~\ref{lem:Green}, the verification
of~\ref{it:w_cont} and~\ref{it:wG_cont} follows by straightforward computations, since the
local behavior of \(w_{\mathrm{free}}\) and \(w^n_{\Omega,1}\), \(w^n_{\Omega,2}\) at the
observation points are the same. For the uniform boundedness from below it
suffices to observe that \(w^n(x) > 0\) for all \(x \in \Oc\setminus\Xi\), the \(w^n\) are
continuous on the same set, and \(w^n(x) \to \infty\) for \(x \to x_m\).
\QED
\end{proof}
\begin{remark}
Certainly, there are many more possibilities to define admissible weights. For instance,
we can use a different discrete norm for the absolute values of the Green's functions
associated with the \(x_m\) or employ a weighed sum. Moreover, the
weight for each \(m\) could be used as a separate regularization parameter, to obtain a
more flexible regularization strategy.
\end{remark}

For any vectors \(v,w \in \C^N\), we define by \(v w \in \C^N\) the coordinate-wise, or
Hadamard product. Define now the weighted norm
\[
\norm{u}_{ M_w(\Oc,\C^N)}
= \int_{\Oc} \abs{w u'}\de\abs{u}
= \int_{\Oc} \sqrt{\sum_{n=1}^N (w^n(x) \abs{u_n'(x)})^2}\de\abs{u}(x)
\]
Since \(u' \in L^\infty(\Oc,\abs{u},\C^N)\) and \(w\) is upper semi-continuous, the
function under the integral is positive and Borel-measurable, and the integral
is well-defined for any \(u \in \Mcon\) (but not necessarily finite).
Note that if \(w^n = w\) for all \(n\), we obtain the more intuitive form
\[
\norm{u}_{\M_w(\Oc,\C^N)} = \int_{\Oc} w\de\abs{u}.
\]
We define the corresponding subspace of \(\Mcon\) as
\[
\M_w(\Oc,\C^N) = \left\{u \in \Mcon \;\Big|\; \int_{\Oc} \abs{w u'}\de\abs{u} < \infty\right\}.
\]
Next, we introduce the mapping $W\colon \M(\Oc, \C^N)\rightarrow \M_w(\Oc,\C^N)$ defined by
\[
\de W(v)=\frac{v'}{w} \de\abs{v}.
\]
Again, the division \(v/w\) for \(v,w\in\C^N\) is understood in a coordinate-wise fashion.
We adopt the convention \(z/(+\infty) = 0\) for any \(z \in \C\).
\begin{proposition}
Let \(w\) fulfill property~\ref{it:w_pos} and~\ref{it:w_cont} and
\(w(\Xi\cap\Oc) \equiv +\infty\).
The mapping $W$ is well-defined and surjective. Moreover, the restriction
\[
W\rvert_{\M(\Oc\setminus\Xi, \C^N)} \colon \M(\Oc\setminus\Xi, \C^N)
\to \M_w(\Oc, \C^N)
\]
is an isometric isomorphism.
\end{proposition}
\begin{proof}
The function $x\mapsto 1/w(x)$ is continuous on $\Oc$ according to the
assumptions. Thus $W(v)$ is an element of $\M(\Oc, \C^N)$. Trivially, there holds $W(v)\in
\M_w(\Oc,\C^N)$ for any \(v \in \Mcon\). Additionally, for any
$u\in \M_w(\Oc,\C^N)$, the product $u w$ defined by
\[
\de(u w)=wu'\de \abs{u}
\]
gives an element in $\Mcon$ since $w$ is
upper semi-continuous. Clearly, we have $W(u w)=u$ and thus $W$ is surjective.
However, \(W\) is not injective, and the kernel of \(W\) can be characterized as
\[
\ker W = \M(\Xi,\C^N)
= \left\{\sum_{m=1}^M u_m\delta_{x_m} \;\Big|\; u_m\in \C^N\right\}
\]
In fact, let $v$ be an element of $\ker W$. Thus there holds
\(\norm{W(v)}_{\Mcon} = \int_{\Oc} 1 / w \de\abs{v} = 0\),
which is equivalent to
\[
\supp{v} = \supp\abs{v} \subseteq \set{x\in \Oc \;|\; 1/w(x)=0} = \Xi\cap\Oc.
\]
As a direct consequence of the isomorphism theorem, we obtain that
\[
W\colon \M(\Oc, \C^N)/\ker W \to \M_w(\Oc, \C^N)
\]
is an isomorphism.
It can be directly verified that the quotient space
is isomorphic to $\M(\Oc\setminus\Xi, \C^N)$; see, e.g., \cite[Theorem~4.9~a)]{Rudin:1991}.
\QED
\end{proof}
Based on these observations, we transform the weighted problem to one with weight one,
which enables us to reuse the general results. We introduce a new optimization variable
$v = u w \in \Mcon$ and employ a reduced formulation in terms of
\(v\). The corresponding observation operator and its adjoint are
defined as
\begin{equation}
\label{eq:so_u}
(\SO^w v, q) = \pair{v, (\SO^w)^*q}
 = \sum_{n=1}^N\sum_{m=1}^M \pair{v, (\bar{G}_n^{x_m} / w_n)\, q_{n,m}},
\end{equation}
for any \(v\in \Mcon, q\in \C^{NM}\).
For any admissible weight, due to property~\ref{it:wG_cont}, this yields a well defined
operator.
\begin{proposition}
\label{prop:so_u}
For any admissible \(w\), the operators \(\SO^w \colon \Mcon \to \C^{NM}\)
and \((\SO^w)^*\colon \C^{NM} \to \Ccon\) are well-defined and continuous with respect
to the weak-\(*\) topology and bounded, respectively.
\end{proposition}
Now, we consider the reduced optimization problem
\begin{equation}\label{eq:reduced_cost_w}\tag{$P_{\alpha, w}$}
  \begin{aligned}
    &\min_{v\in \mathcal M(\Oc,\C^N)} j_w (v)
    = \frac{1}{2}\sum_{m=1}^{M}\abs{(\SO^w v)_m - p_d^m}_{\C^N}^2 + \alpha\norm{v}_{\M(\Oc,\C^N)}.
  \end{aligned}
\end{equation}
Since the reweighed problem~\eqref{eq:reduced_cost_w} has exactly the same structural
properties as the reduced problem~\eqref{eq:reduced_cost}, all results from
sections~\ref{sec:opt_cont_prob} and~\ref{sec:regularization} can be transferred without
modification.
In particular, for any admissible weight the problem~\eqref{eq:reduced_cost_w} admits
optimal solutions \(\widehat{v} \in \Mcon\) consisting of at most \(2NM\) Dirac delta functions.

Given a solution \(\widehat{v}\) of~\eqref{eq:reduced_cost_w} which does not contain any
Dirac delta functions in the observation points (i.e., \(\widehat{v} \in
\M(\Oc\setminus\Xi,\C^N)\)), we can apply \(W\) to obtain a solution of the original
problem. First, we need some result to connect the algebraically defined operator
\(\SO^w\) to the point evaluations of the solutions of~\eqref{eq:state}.
\begin{lemma}
\label{lem:solution_operator_w}
For \(\eps > 0\) define \(\Xi^\eps = \cup_{m=1,\ldots,M} B_\eps(x_m)\).
Let the observation operator \(\SO_\eps \colon \M(\Oc\setminus \Xi^\eps,\C^N) \to \C^{NM}\)
be defined as \(S_\eps(u) = (p(x_m))_{m=1,\ldots,M}\), where \(p\) is the solution
to~\eqref{eq:state} (defined with Lemma~\ref{lem:improved}).

If \(w\) is admissible, the operator \(\SO^w\colon \M(\Oc\setminus \Xi,\C^N) \to \C^{NM}\) is the
unique weak-\(*\) continuous extension of the family of operators
\(\SO^w_\eps = \SO_\eps \circ W\).
\end{lemma}
\begin{proof}
By a simple computation, \(\SO^w\rvert_{\M(\Oc\setminus\Xi,\C^N)}\) extends all
\(\SO^w_\eps\) and by Proposition~\ref{prop:so_u} it is continuous.
Clearly, the spaces \(\M(\Oc\setminus\Xi^\eps,\C^N)\) are
weak-\(*\) dense in \(\M(\Oc\setminus\Xi,\C^N)\), which proves the uniqueness of the
extension.
\QED
\end{proof}

\begin{lemma}
\label{lem:well_posed_weight}
Let \(w\) be admissible.
\begin{enumerate}[label=\roman*)]
\item
  Suppose that~\eqref{eq:reduced_cost_w} possesses a solution
  \(\widehat{v}\) with \(\supp\widehat{v} \subset \Oc\setminus\Xi\). Then, \(\widehat{u}
  = W\widehat{v}\) is a solution of~\eqref{eq:weighted_problem}.
\item
  Conversely, suppose that any solution of~\eqref{eq:reduced_cost_w} fulfills
  \(\abs{\widehat{v}}(\Xi) > 0\) and that \(\Oc\) contains no isolated
  points. Then~\eqref{eq:weighted_problem} possesses no solution.
\end{enumerate}
\end{lemma}
\begin{proof}
Based on Proposition~\ref{prop:so_u} and Lemma~\ref{lem:solution_operator_w}, the point
evaluations of the solutions to~\eqref{eq:state} with sources in \(\M_w(\Oc,\C^N)\) are
well-defined. Moreover, using the isometric isomorphism property of \(W\)
from Proposition~\ref{prop:so_u}, the infimum
of~\eqref{eq:weighted_problem} is equal to
\begin{equation}
\label{eq:inf_help_jw}
\widehat\jmath = \inf_{v \in \M(\Oc\setminus\Xi,\C^N)} j_w(v).
\end{equation}
Clearly, the minimum of~\eqref{eq:reduced_cost_w} fulfills
\(\min_{v\in\Mcon} j_w(v) \leq \widehat\jmath\).

Now, if~\eqref{eq:reduced_cost_w} admits a solution \(\widehat{v} \in
\M(\Oc\setminus\Xi,\C^N)\), it follows that \(j_w(\widehat{v}) =
\widehat\jmath\) and the infimum of~\eqref{eq:weighted_problem}
is assumed by \(\widehat{u} = W\widehat{v}\).

Conversely, if any solution to~\eqref{eq:reduced_cost_w} is not in
\(\M(\Oc\setminus\Xi,\C^N)\), the infimum in~\eqref{eq:inf_help_jw} is not assumed.
To see this, we first show that it in fact holds that \(\min_{v\in\Mcon} j_w(v) =
\widehat\jmath\). Take any sparse solution \(\widehat{v}\) of~\eqref{eq:reduced_cost_w}.
By the assumption, it contains Dirac delta functions supported on \(\Xi\). Since the support
points which coincide with observation points are not isolated in \(\Oc\), we can slightly
perturb them, such that \(x_m \neq \tilde{x}_m^\ell \to x_m\) for \(\ell \to
\infty\). Denote the perturbed measure by \(\tilde{v}_\ell\). It holds \(\norm{\tilde{v}_\ell}_{\Mcon} =
\norm{\widehat{v}}_{\Mcon}\) and \(\tilde{v}_\ell \in \M(\Oc\setminus\Xi,\C^N)\) for \(\ell\)
big enough and with the weak-\(*\) continuity of \(\SO^w\) we obtain \(\widehat\jmath \leq
\lim_{n\to\infty} j_w(\tilde{v}_\ell) = j_w(\widehat{v})\). Therefore, \(j_w\) can not assume
its minimum on \(\M(\Oc\setminus\Xi,\C^N)\), which
directly implies that~\eqref{eq:weighted_problem} has no minimum, using
again Proposition~\ref{prop:so_u} and Lemma~\ref{lem:solution_operator_w}.
\QED
\end{proof}

To obtain well-posedness of the weighted
problem~\eqref{eq:weighted_problem} without any assumptions on the
structure of the solutions of the auxiliary problem~\eqref{eq:reduced_cost_w}, we can
impose the additional condition \([G^{x_m}/w](x_m) = 0\) for all \(m\).
For instance, for any admissible weight \(w\) (such as given in~\eqref{eq:example_w_free}
or~\eqref{eq:example_w_omega}) and some monotonously increasing function \(\psi\colon \R
\to \R_+\) with \(\psi(0) = 0\), \(\psi(t) > 0\) for \(t > 0\), and \(\psi(t)/t \to
\infty\) for \(t\to\infty\), the weight \(\tilde{w} = \psi\circ w\) has this property.
\begin{proposition}
Suppose that \(w\) is admissible with \([G^{x_m}/w](x_m) = 0\) for all \(m\). Then, the
operator \(S^w\) is weak-\(*\) continuous on the space \(\M(\Oc\setminus\Xi,\C^N)\).
\end{proposition}
\begin{proof}
This follows directly from the observation that
\[
(S^w)^* \colon \C^N \to \Cc_0(\Oc\setminus\Xi,\C^N)
= \set{v\in \Cc(\Oc,\C^N) \;|\; v(x_m)=0,x_m \in \Xi\cap\Oc}.
\]
and the identification
\(\Cc_0(\Oc\setminus\Xi,\C^N)^* = \M(\Oc\setminus\Xi,\C^N)\).
\QED
\end{proof}
In this case, the solutions of~\eqref{eq:reduced_cost_w} are always supported on
\(\Oc\setminus \Xi\), which follows from the optimality conditions and the fact that
\(\widehat\xi = - (S^w)^*(S^w(\widehat{v}) - p_d)\) fulfills \(\widehat{\xi}(\Xi\cap\Oc) = 0\).
We summarize all results in the following theorem.
\begin{theorem}
  \label{thm:well_posed_weighted}
  Let \(w\) be admissible and suppose that~\eqref{eq:reduced_cost_w} admits solutions in
  the space \(\M(\Oc\setminus\Xi,\C^N)\) or that \([G^{x_m}/w](x_m) = 0\) for all \(m\).
  Then, the problem~\eqref{eq:weighted_problem} has a minimum
  \(\widehat{u} \in \M_w(\Oc,\C^{N})\) which consists of finitely many Dirac delta
  functions, \(\widehat{u} = \sum_{j=1}^{N_d} \widehat{\coeff{u}}_{j}
  \delta_{\widehat{x}_j}\). Together with the associated
  \[
    \widehat{\xi}= - \SO^* (\SO\widehat{u} - p_d),
  \]
  it is uniquely characterized by the optimality conditions
\begin{align*}
  \norm{\widehat{\xi}/w}_{\Cc(\Oc,\C^N)} \leq \alpha, \qquad
  \alpha\,w(\widehat{x}_j)\widehat{\coeff{u}}_j
  = \abs{w(\widehat{x}_j)\widehat{\coeff{u}}_j}_{\C^N} \, \widehat{\xi}(\widehat{x}_j)/w(\widehat{x}_j),
\end{align*}
\(j \in \set{1,2,\ldots,N_d}\). Moreover,
$\supp\abs{\widehat{u}} \subset \set{x\in\Oc \;|\; \abs{\widehat{\xi}(x)/w(x)}_{\C^N} = \alpha}$
for each solution $\widehat{u}$.
\end{theorem}

\section{Regularization properties}
\label{sec:regularization}
In this section, we study (loosely speaking) if the minimization problem
delivers an appropriate solution for the inverse problem:
solve \(S u = p\) for \(u\). We mainly rely on general results for nonsmooth Tikhonov
regularization~\cite{BurgerOsher:2004,HofmannKaltenbacherPoeschlScherzer:2007} and sparse
spike deconvolution~\cite{BrediesPikkarainen:2013,DuvalPeyre:2015}.
To that purpose, we assume that we are given the exact source $u^\star$ of the form
\begin{equation}
\label{eq:exact_source}
u^\star=\sum_{j=1}^{N^\star}\coeff{u}^\star_j\delta_{x^\star_j},
\quad\text{where }\coeff{u}^\star_j\in \C^N\setminus\set{0},~x^\star_j\in \Oc\setminus\Xi
\end{equation}
and noisy observations \(p_d = \SO u^\star + f = p^\star + f\) with small noise
\(\norm{f}_{\C^{NM}} \leq \delta\). In the following we state conditions on $u^\star$ and
a parameter choice rule for \(\alpha\) in dependence of $\delta$ which are sufficient for
the convergence of the solutions \(\widehat{u}_\alpha\) of~\eqref{eq:pre_opt_prob}--\eqref{eq:state} (or the
weighted problem~\eqref{eq:weighted_problem})
towards the exact solution $u^\star$ for vanishing noise
\(\delta \to 0\) and for \(\alpha(\delta) \to 0\).
Moreover, convergence rates are given.

Without loss of generality, we only study the reduced weighted
problem~\eqref{eq:reduced_cost_w} for a general admissible weight \(w\). The case
of \(w\equiv 1\) with \(\Oc\cap\Xi = \emptyset\) from section~\ref{sec:opt_cont_prob} is
then included as a simple special case. In the case of solutions \(\widehat{v}_\alpha\) of
formulation~\eqref{eq:reduced_cost_w}, we are interested in the convergence of
\(W\widehat{v}_\alpha\) towards \(u^\star\). We define
\begin{equation}
\label{eq:exact_source_w}
v^\star=\sum_{j=1}^{N^\star}\coeff{v}^\star_j\delta_{x^\star_j},
\quad\text{where }\coeff{v}^\star_j = w(x^\star_j)\coeff{u}^\star_j,
\end{equation}
In the following, we study the convergence of solutions \(\widehat{v}_{\alpha(\delta)}\) towards \(v^\star\).
Clearly, since \(1/w\) is a continuous function on \(\Oc\), this
implies convergence of \(W\widehat{v}_\alpha\) towards \(Wv^\star = u^\star\).
We first analyse the following minimum norm problem, (cf., e.g.,
\cite{HofmannKaltenbacherPoeschlScherzer:2007,BrediesPikkarainen:2013,DuvalPeyre:2015}):
\begin{align}
\label{eq:reduced_cost_zero_w}\tag{$P_{0,w}$}
  \min_{v\in \Mcon} \norm{v}_{\Mcon} \quad\text{subject to } \SO^w v = p^\star.
\end{align}
By assumption, the admissible set of~\eqref{eq:reduced_cost_zero_w} is not empty, since
\(p^\star = S^w v^\star\).
Therefore, with Lemma~\ref{lem:weak_to_strong}, we can derive the
following basic result; see Appendix~\ref{app:extremal_solutions}.
\begin{proposition}
\label{prop:existence_zero}
There exists a solution \(v^\dagger \in \Mcon\) to~\eqref{eq:reduced_cost_zero_w}, which
consists of \(N_d \leq 2NM\) point sources,
\[
v^\dagger = \sum_{j=1}^{N_d} \coeff{v}^\dagger_j \delta_{x^\dagger_j}
\quad\text{where } \coeff{v}^\dagger_j \in \C^N,\; x^\dagger_j \in \Oc.
\]
\end{proposition}
We now turn to the limiting behavior of~\eqref{eq:reduced_cost_w} for small \(\alpha\)
and \(\delta\).
From~\cite{HofmannKaltenbacherPoeschlScherzer:2007} (cf.\
\cite[Section~4]{BrediesPikkarainen:2013}), we have the following result.
\begin{theorem}
\label{thm:variational_reg}
For any monotonously increasing parameter choice rule \(\alpha(\delta)\) for which
$\delta^2/\alpha(\delta) \to 0$ and \(\alpha(\delta) \to 0\) for \(\delta \to 0\),
any sequence \(\widehat{v}_{\alpha(\delta)}\) of solutions to~\eqref{eq:reduced_cost_w}
contain a subsequence which converges towards a solution $v^\dagger$
of~\eqref{eq:reduced_cost_zero_w} (weakly-\(*\) in \(\M(\Oc,\C^N)\)).
If additionally \(v^\dagger\) is unique, the whole sequence converges towards \(v^\dagger\).
\end{theorem}
Under a \emph{source condition}
convergence rates can be derived in a generalized Bregman distance (see,
e.g.,~\cite{BurgerOsher:2004}). It has the following form:
\begin{equation}
\label{eq:source_condition}
\text{There exists a}~y^\dagger \in \C^{NM},
\;\text{such that}~(\SO^w)^* y^\dagger \in \partial \norm{v^\star}_{\M(\Oc,\C^N)}.
\end{equation}
A concrete form of this condition can be given by using the characterization of the
subdifferential.
\begin{proposition}
\label{prop:source_condition_write_out}
The source condition~\eqref{eq:source_condition} can be equivalently expressed as: There
exists a \(y^\dagger \in \C^{NM}\), such that the associated adjoint state \(\xi^\dagger = \SO^*
y^\dagger\) fulfills
\begin{align*}
\norm{\xi^\dagger/w}_{\Cc(\Oc,\C^N)} \leq 1, \qquad
\coeff{v}^\star_j = \abs{\coeff{v}^\star_j}_{\C^N}\,\xi^\dagger(x^\star_j)/w(x^\star_j),
\quad j \in \set{1,2,\ldots,N^\star}.
\end{align*}
The last condition can also be given by \(w(x^\star_j)\coeff{u}^\star_j =
\abs{w(x^\star_j)\coeff{u}^\star_j}_{\C^N}\,\xi^\dagger(x^\star_j)/w(x^\star_j)\).
\end{proposition}
In our situation, the source condition is satisfied if $v^\star$ is a
minimum norm solution, since~\eqref{eq:source_condition} is a necessary and
sufficient optimality condition of the minimum norm problem
problem~\eqref{eq:reduced_cost_zero_w}.
\begin{proposition}
\label{prop:source_condition_fulfilled}
For any solution \(v^\dagger\) of \eqref{eq:reduced_cost_zero_w} there exists a corresponding \(y^\dagger\) such that \((\SO^w)^* y^\dagger \in \partial \norm{v^\dagger}_{\M(\Oc,\C^N)}\). Conversely, for any
pair \(y^\dagger\) and \(v^\star\) fulfilling~\eqref{eq:source_condition}, \(v^\star\) is
a solution of~\eqref{eq:reduced_cost_zero_w}.
\end{proposition}
\begin{proof}
This result follows by an application of Fenchel duality (see
Propositions~\ref{prop:strong_duality} and~\ref{prop:existence_dual_zero} in the Appendix).
\QED
\end{proof}
\begin{corollary}
The element $v^\star$ satisfies the source condition~\eqref{eq:source_condition} if and
only if $v^\star$ is a solution of the minimum norm problem~\eqref{eq:reduced_cost_zero_w}.
\end{corollary}
\begin{remark}
The equivalence between the minimum norm problem and the
source condition is due the semi-infinite character of the
dual problem of~\eqref{eq:reduced_cost}; see Appendix~\ref{app:minimization_finite_rank}.
In an general setting (with infinite dimensional observation) this equivalence is not
always given; cf.~\cite{HofmannKaltenbacherPoeschlScherzer:2007}.
\end{remark}
The convergence rates for the regularized solutions will now be given in terms of a
generalized, set-valued Bregman distance $D\colon \Mcon\times \Mcon\rightarrow \R$ defined by
\[
D(v_1,v_2)=\set{\norm{v_1}_{\Mcon}-\norm{v_2}_{\Mcon}-\pair{\xi,v_1-v_2} \;|\; \xi\in \partial\norm{v_2}_{\Mcon}}
\]
for any $v_1,v_2\in \Mcon$.
In \cite[Theorem~2]{BurgerOsher:2004} (cf.\ also~\cite[Section~4]{BrediesPikkarainen:2013})
the following convergence result is proven.
\begin{theorem}\label{thm:conv_breg}
Let the source condition \eqref{eq:source_condition} be satisfied and let
\(\delta/c_1 \leq \alpha(\delta) \leq c_1\delta\) for some fixed \(c_1 \geq 1\). Then for
each minimizer $\widehat{u}_\alpha$ of \eqref{eq:pre_opt_prob} there exists a
$d\in D(\widehat u_\alpha,u^\star)$  such that $d \leq C\delta$ holds (for some generic
constant \(C\)).
\end{theorem}
Based on Theorem~\ref{thm:variational_reg} and
Proposition~\ref{prop:source_condition_fulfilled} we see that the only missing part for the
convergence of $\widehat{v}_{\alpha(\delta)}$ to $v^\star$ is the uniqueness of the solution of
the minimum norm problem. Due to
Proposition~\ref{prop:existence_zero} unique solutions must necessarily consist of
finitely many Dirac delta functions.
Additionally, criteria for uniqueness based on the source condition can be derived.
We give without proof the following popular one; cf.\ \cite[Lemma~1.1]{DeCastroGamboa:2012}
or \cite[Proposition~5]{DuvalPeyre:2015}:
\begin{proposition}
\label{prop:criterion_uniqueness}
Let \(v^\star \in \Mcon\)  with \(v^\star = \sum_{j=1,\ldots,N_d} \coeff{v}^\star_j \delta_{x^\star_j}\),
where \(\coeff{v}_j \in \C^N\setminus\set{0}\), \(x_j^\star\in \Oc\setminus\Xi\) pairwise different.
Suppose further that the source condition~\eqref{eq:source_condition} holds with
\(\xi^\dagger/w = (S^w)^* y^\dagger\), the vectors \(z_j = \SO^w (\xi_w^\dagger(x^\star_j)
\delta_{x^\star_j}) \in \C^{MN}\) form a \(\R\)-linearly independent set, and for every
\(x \in \Oc \setminus \set{x^\star_j \;|\; j = 1,2,\ldots,N^\star}\)
there holds \(\norm{\xi^\dagger(x)/w(x)}_{\C^{NM}} < 1\).
Then \(v^\star\) is the unique solution of~\eqref{eq:reduced_cost_zero_w}.
\end{proposition}
Finally, we sum up the findings of this section.
\begin{corollary}
\label{cor:inverse_problem_convergence}
Let $u^\star = W v^\star$, where \(v^\star\) is a solution of~\eqref{eq:reduced_cost_zero_w}
(or equivalently let $v^\star$ satisfy the source condition~\eqref{eq:source_condition})
and let the conditions from Proposition~\ref{prop:criterion_uniqueness} be
satisfied. Furthermore, let $\delta/c_1 \leq \alpha(\delta) \leq
c_1\delta$ for some \(c_1 > 0\) as \(\delta \to 0\). Then for any sequence of minimizers
$\widehat{u}_{\alpha(\delta)}$ of~\eqref{eq:reduced_cost_w} it holds
\[
W \widehat{v}_{\alpha(\delta)} \rightharpoonup^* u^\star,
\]
and there exists $d\in D(\widehat{v}_{\alpha(\delta)},v^\star)$ such that $d \leq C\delta$
(for some generic \(C\)).
\end{corollary}

Due to the complex geometrical setup of~\eqref{eq:helm} (in the general case,
analytical solutions are not known), we know of no way to further characterize
the set of sources for which the assumptions of
Corollary~\ref{cor:inverse_problem_convergence} hold. However, we refer
to~\cite{DuvalPeyre:2015,AzaisDeCastroGamboa:2015,CandesFernandezGranda:2013}, where for
certain classes of analytically given convolution operators similar results to
Corollary~\ref{cor:inverse_problem_convergence} can be guaranteed under simple structural
assumptions on the source, such as, e.g., a minimum separation distance between the
support points of~\eqref{eq:exact_source}.
In our situation, we will investigate the assumptions of
Corollary~\ref{cor:inverse_problem_convergence} numerically in
section~\ref{sec:numerics}. The numerical results suggest that, even in the case of
an arbitrary number of measurements, the source condition holds
only in some cases. However, for a special choice of the weight, reconstruction of a
single point source can be guaranteed.

\subsection{Exact reconstruction of a single source}

In this section we prove that, using the weight \(w_{\Omega,2}\)
as defined in~\eqref{eq:example_w_omega}, a source consisting of a
single Dirac-delta function can always be reconstructed using the weighted
problem. We first consider the noise free case:
\begin{proposition}
  \label{prop:exact_reconstruction}
  Suppose that \(w^n_{\Omega,2}(x) > 0\) for all \(x\in\Oc\), \(n = 1,\ldots,N\).
  Let \(u^\star = \coeff{u}^\star \delta_{x^\star}\) with \(x^\star \in
  \Oc\setminus\Xi\), \(\coeff{u}^\star \in \C^N\) and consider noise-free observations
  \(p_d = \SO u^\star\). Then, for any \(\alpha > 0\) and \(w = w_{\Omega,2}\) the function
  \[
    \widehat{u}
    = \widehat{\coeff{u}} \delta_{x^\star}
    \quad\text{with}\quad
    \widehat{\coeff{u}}
    = \max \set{0,\; 1 - \alpha /\abs{\coeff{u}^\star w(x^{\star})}_{\C^N}}\coeff{u}^\star
  \]
  is a solution of~\eqref{eq:weighted_problem}. Furthermore,
  \(u^\dagger = u^\star\) solves the corresponding minimum norm problem defined as in
  section~\ref{sec:regularization}.
\end{proposition}
\begin{proof}
  We verify that the first order conditions from
  Theorem~\ref{thm:well_posed_weighted} are fulfilled.
  First, we compute \(\widehat{\xi} = - \SO^*(\SO\widehat{u} - p_d)\) at
  every point and frequency. We directly obtain that
  \[
    \widehat{\xi}_n(x) = \sum_{m=1}^M \bar{G}_n^{x_m}(x) G_n^{x_m}(x^\star)
    (\coeff{u}^\star_n - \widehat{\coeff{u}}_n),
    \quad x \in\Oc,\, n = 1,\ldots,N.
  \]
  We compute that
  \((\coeff{u}^\star_n - \widehat{\coeff{u}}_n)
  = \min\set{1,\; \alpha/ \abs{w(x^\star)\coeff{u}^\star}_{\C^N}} \coeff{u}^\star_n\).
  Introducing the rescaled Green's functions \(h^n_{m} = G_n^{x_m}/w^n\), we obtain
  \begin{equation}
    \label{eq:grad_expanded}
    \widehat{\xi}_n(x)/w^n(x)
    = \sum_{m=1}^M \bar{h}^n_{m}(x) h^n_{m}(x^\star)
    \min\set{1,\; \alpha / \abs{w(x^\star)\coeff{u}^\star}_{\C^N}} w^n(x^\star)\coeff{u}^\star_n.
  \end{equation}
  By the definition of \(w = w_{\Omega,2}\), we compute that \(\abs{h^n(x)}_{\C^M} =
  \sqrt{\sum_{m} \abs{h^n_m(x)}^2} = \abs{G^n(x)}_{\C^M}/w_{\Omega,2}(x) = 1\) for all
  \(x\in\Oc\setminus\Xi\). Therefore, we can apply the Cauchy-Schwarz inequality to the
  term \(\sum_{m} \bar{h}^n_{m}(x) h^n_{m}(x^\star)\) in~\eqref{eq:grad_expanded} and obtain
  \[
    \abs{\widehat{\xi}_n(x)}/w^n(x)
    \leq \min\set{1,\;\alpha/\abs{w(x^\star)\coeff{u}^\star}_{\C^N}} w^n(x^\star)\abs{\coeff{u}^\star_n}.
  \]
  Summing the squares of both sides and taking the square root, we derive that
  \[
    \abs{\widehat{\xi}(x)/w(x)}_{\C^N}
    \leq \min\set{\abs{w(x^\star)\coeff{u}^\star}_{\C^N},\; \alpha} \leq \alpha,
    \quad x \in\Oc.
  \]
  In the case that \(\alpha < \abs{w(x^\star)\coeff{u}^\star}_{\C^N}\), it remains to
  verify the optimality condition for \(\widehat{\coeff{u}}\):
  Taking \(x = x^\star\), we have
  \(\sum_{m} \bar{h}^n_{m}(x^\star) h^n_{m}(x^\star) = 1\) in~\eqref{eq:grad_expanded},
  and it follows that
  \[
    \widehat{\xi}_n(x^\star)/w^n(x^\star)
    = \alpha w^n(x^\star)\coeff{u}^\star_n/\abs{w(x^\star)\coeff{u}^\star}_{\C^N},
    \quad n = 1,\ldots,N,
  \]
  which implies the desired condition,
  since \(\widehat{\coeff{u}}\) and \(\coeff{u}^\star\) are scalar multiples of each other.
  Thus, \(\widehat{u}_\alpha\) solves the weighted problem by
  Theorem~\ref{thm:well_posed_weighted}.

  In the case \(\alpha = 0\), we show that the solution of the dual problem is given by
  \(y^\dagger = p_d/\abs{w(x^\star)\coeff{u}^\star}_{\C^N}\).
  In light of Proposition~\ref{prop:source_condition_fulfilled}, we have to
  verify that \(\xi^\dagger = \SO^*y^\dagger\) fulfills the source condition, i.e.,
  \(\xi^\dagger/w \in \partial\norm{w u^\star}_{\Mcon}\). We have
  \begin{multline*}
    \xi_n^\dagger(x)/w^n(x) = \sum_{m=1}^M \bar{h}^n_{m}(x) p^m_d/\abs{w(x^\star)\coeff{u}^\star}_{\C^N}
    \\
    = \sum_{m=1}^M \bar{h}^n_{m}(x) h_m^n(x^\star)  w^n(x^\star)\coeff{u}^\star_n/\abs{w(x^\star)\coeff{u}^\star}_{\C^N}.
  \end{multline*}
  Similarly, it follows \(\norm{\xi^\dagger/w}_{\Ccon} \leq 1\)
  and \(\xi^\dagger(x^\star)/w(x^\star) =
  w(x^\star)\coeff{u}^\star/\abs{w(x^\star)\coeff{u}^\star}_{\C^N}\), which implies the
  result by Proposition~\ref{prop:source_condition_write_out}.
\QED
\end{proof}
Note that~\eqref{prop:exact_reconstruction} also applies in the case of only one
measurement, i.e.~\(M = 1\). In this case, for any \(\xi = S^* y\) with \(y \in \C^N\),
the expression \(\norm{\xi/w}_{\C^N}\) is constant in the domain \(\Omega\), and any
source \(\coeff{u} \delta_x\) for arbitrary \(x\in\Oc\setminus\Xi\) and appropriate
\(\coeff{u} \in \C^N\) solves the minimum norm problem. A criterion for
\(u^\star\) to be the unique solution, which can be derived by straightforward
extension of the previous result, is given next.
\begin{proposition}
In addition to the requirements of Proposition~\ref{prop:exact_reconstruction}, assume
that the observations for different source locations are complex linearly independent
(i.e., there exist no \(x,x' \in \Oc\), such that \(S\delta_{x} = z
S\delta_{x'}\) for \(z \in \C\)).

Then the functions given in Proposition~\ref{prop:exact_reconstruction} are the unique
solutions of the respective problems.
\end{proposition}

\section{Optimization algorithm}
\label{sec:optimization}

We base the numerical optimization of~\eqref{eq:problem_convex} upon the successive peak
insertion and thresholding algorithm proposed in~\cite{BrediesPikkarainen:2013}.
It is based on iterates of the form \(u^k = \sum_{j=1,\ldots,N_d^k}
\coeff{u}^k_j\delta_{x^k_j}\) (with distinct \(x^k_j\) and \(\coeff{u}^k_j \neq 0\)) and
performs alternating steps, combining insertion of Dirac delta functions at new locations
with removal steps.

For the convenience of the reader, we give a general description of the resulting
procedure in Algorithm~\ref{alg:SPINAT}.
Note, that the point insertion is performed at the maximum of the norm of the current adjoint state.
For more details we refer to~\cite[Section~5]{BrediesPikkarainen:2013}.
\begin{algorithm}
\begin{algorithmic}
 \WHILE {``duality-gap large''}
 \STATE 1.
 Compute $\xi^k = \SO^*(\SO u^k - p_d)$.
 Determine $\hat{x}^k \in \argmax_{x\in\Oc}\abs{\xi^k(x)}_{\C^N}$.
 \STATE 2.
 Set $\theta^k = \begin{cases} 0, & \norm{\xi^k}_{\Ccon} \leq \alpha, \\
   -\left[\alpha^{-2}\norm{p_d}^2/2\right] \xi^k(\hat{x}^k), & \text{else}. \end{cases}$
 \STATE 3.
 Select stepsize \(s^k \in (0,1]\) and set
 \(u^{k+1/2} = (1-s^k)u^k + s^k \theta^k \delta_{\hat{x}_k}\).
 \STATE 4.
 Set \(\mathcal{A} = \supp(u^{k+1/2})\) and
 find \(\coeff{u}^{k+1} \in \C^{N\#{\mathcal{A}}}\) such that \(u^{k+1} =
 U_{\mathcal{A}}(\coeff{u}^{k+1})\)  with \(j(u^{k+1}) \leq j(u^{k+1/2})\).
 \ENDWHILE
\end{algorithmic}
\caption{Successive peak insertion framework~\cite{BrediesPikkarainen:2013}}\label{alg:SPINAT}
\end{algorithm}
The following convergence result is obtained there:
\begin{theorem}[{\cite[Theorem~5.8]{BrediesPikkarainen:2013}}]
\label{thm:sublinear_convergence}
  Let the sequence $u^k$ be generated by Algorithm~\ref{alg:SPINAT}. Then
  every subsequence of $u^k$ has a weak-\(*\) convergent
  subsequence that converges to a minimizer $\widehat{u}$. Furthermore:
  \begin{align*}
    j(u^k)-j(\widehat{u})
    \leq \frac{C}{k}.
  \end{align*}
\end{theorem}
To discuss different possible implementations of step~4 in Algorithm~\ref{alg:SPINAT}, we
define for a ordered set of distinct points
\(\mathcal{A} = \set{x_j \in \Oc \;|\; j=1,\ldots,\#\mathcal{A}}\)
the operator \(U_{\mathcal{A}} \colon \C^{N \#\mathcal{A}} \to \Mcon\)  by
\[
U_{\mathcal{A}}(\coeff{u}) = \sum_{j=1}^{\#\mathcal{A}} \coeff{u}_j\delta_{x_j}.
\]
The removal steps are based on the consideration of the finite-dimensional problem
\begin{equation}\label{eq:reduced_cost_active}
  \begin{aligned}
    \min_{\coeff{u}\in\C^{NN_d}} j(U_{\mathcal{A}}(\coeff{u}))
    &= \frac{1}{2}\norm{\SO(U_{\mathcal{A}}(\coeff{u})) - p_d}_{\C^{NK}}^2
    + \alpha \norm{U_{\mathcal{A}}(\coeff{u})}_{\Mcon} \\
    &= \frac{1}{2}\norm{\coeff{S}_{\mathcal{A}}\coeff{u} - p_d}_{\C^{NK}}^2
    + \alpha \sum_{j=1}^{\#\mathcal{A}}\abs{\coeff{u}_j}_{\C^N},
  \end{aligned}
\end{equation}
for \(\mathcal{A}\) determined by an intermediate iterate and
\((\coeff{S}_{\mathcal{A}})_{j,n} = \SO\delta_{x_j}e_n\).
Different concrete choices of step~4 are
discussed in~\cite[Section~5]{BrediesPikkarainen:2013}: it is suggested to
perform one step of the well-know proximal gradient/iterative tresholding algorithm
for the finite dimensional problem~\eqref{eq:reduced_cost_active}.
In this way, step~5 is easy to implement, has a small cost (depending linearly on the
current size of the support), and has the potential to set some coefficients
to zero (by virtue of the soft shrinkage operator). Additional steps of the proximal
gradient method could be performed, to possibly increase this ``sparsifying'' effect.
Note that if we omit step~4, the size of the support will grow monotonically throughout
the iterations due to the particular form of step~3 (except for the unlikely case that
\(s_k = 1\)).

In our setting, we additionally know that solutions consisting of
at most \(2NM\) Dirac delta functions exist; see
Corollary~\ref{cor:dirac_solutions}. Since the proof of the underlying result is
constructive, it directly suggests an algorithm to remove excess point sources; see
Proposition~\ref{prop:remove_diracs}.
\begin{corollary}
  \label{cor:remove_diracs}
  For given \(u^{k+1/2}\) with \(\#{\supp(u^{k+1/2})} > 2NM\), the
  algorithm from the proof of Proposition~\ref{prop:remove_diracs} constructs a new iterate
  \(u^{k+1} = U_{\mathcal{A}}(\coeff{u})\), such that \(\coeff{u}_{\hat\jmath} = 0\) for one
  \(\hat\jmath\) and \(j(u^{k+1}) \leq j(u^{k+1/2})\).
\end{corollary}
\begin{proposition}
  Suppose that step~5 of Algorithm~\ref{alg:SPINAT} includes the procedure from
  Corollary~\ref{cor:remove_diracs} and that \(u^0\) consists of at most \(2NM\) Dirac
  delta functions. Then the iterates \(u^k\) and each
  weak-\(*\) accumulation point \(\widehat{u}\) of \(u^k\) consists of at most \(2NM\) Dirac
  delta functions (in addition to the properties from Theorem~\ref{thm:sublinear_convergence}).
\end{proposition}
\begin{proof}
The bound on the support size for \(u^k\) is a direct consequence of
Corollary~\ref{cor:remove_diracs}. The bound for the limit follows from a general result
on the weak-\(*\) convergence of measures consisting of a uniformly bounded number of
Dirac delta  functions; see Appendix \ref{app:weak_closedness}.
\QED
\end{proof}

Additionally, \cite{BrediesPikkarainen:2013} suggests acceleration strategies  based on
point moving and merging. Since they cannot be easily realized in our numerical setup
using \(\Cc^0\) finite elements (see section~\ref{sec:numerics}), we do not
discuss them here. Alternatively, we suggest to solve the
subproblem~\eqref{eq:reduced_cost_active} exactly (up to machine precision) to accelerate
the convergence. The resulting procedure is given in Algorithm~\ref{alg:PDAP}.
\begin{algorithm}
\begin{algorithmic}
 \WHILE {``duality-gap large''}
 \STATE 1. Calculate $\xi^k = \SO^*(\SO u^k - p_d)$.
 Determine $\hat{x}^k \in \argmax_{x\in\Oc}\abs{\xi^k(x)}_{\C^N}$.
 \STATE 2. Set $\mathcal{A} = \supp(u^k) \cup \set{\hat{x}^k}$,
  compute a solution \(\widehat{\coeff{u}} \in \C^{N\#{\mathcal{A}}}\) of
  \eqref{eq:reduced_cost_active} with \(\#\supp(U_{\mathcal{A}}(\widehat{\coeff{u}})) \leq 2NM\), and
  set \(u^{k+1} = U_{\mathcal{A}}(\widehat{\coeff{u}})\).
 \ENDWHILE
\end{algorithmic}
\caption{Primal-Dual-Active-Point strategy}\label{alg:PDAP}
\end{algorithm}
Since the point insertion is the same in both algorithms,
Algorithm~\ref{alg:PDAP} is a special case of Algorithm~\ref{alg:SPINAT}.
\begin{proposition}
  \label{prop:PDAPasSPINAT}
  The iterates of Algorithm~\ref{alg:PDAP} coincide with the iterates of
  Algorithm~\ref{alg:SPINAT}, if in step~4, \(\coeff{u}^{k+1}\) is chosen as a
  solution \(\widehat{\coeff{u}} \in \C^{N\#{\mathcal{A}}}\)
  of~\eqref{eq:reduced_cost_active}.
\end{proposition}
\begin{proof}
This is a direct consequence of the fact that step~1 and the choice of \(\mathcal{A}\)
coincide for both algorithms, and that \(j(U_{\mathcal{A}}(\widehat{\coeff{u}}))
\leq j(u^{k+1/2}) \leq j(u^k)\); see~\cite[Proposition~5.6]{BrediesPikkarainen:2013}.
\QED
\end{proof}
\begin{remark}
\label{rem:stopping_criterion}
Another possible stopping criterion for Algorithm~\ref{alg:PDAP} would be the condition
that the active set \(\mathcal{A}\) coincides in two subsequent iterations \(k\) and
\(k+1\) i.e., that \(\hat{x}^{k+1} \in \mathcal{A}(u^{k+1})\) in step \(k+1\). Clearly, if
this holds true, we have \(u^{k+1} = u^{k+2} = \widehat{u}\). In fact, the optimality of
\(u^{k+1}\) can be obtained in this situation by formulating the
optimality conditions of~\eqref{eq:reduced_cost_active} from step \(k\) for \(u^{k+1} =
U_{\mathcal{A}}(\widehat{\coeff{u}})\), concluding that \(\hat{x}^{k+1}
\in \mathcal{A}(u^{k+1})\) implies that \(\norm{\xi^{k+1}}_{\Ccon} \leq \alpha\) and
verifying the first order conditions from Corollary~\ref{cor:optimality_conditions_dirac},
which are sufficient for optimality.
\end{remark}

It remains to address the cost associated with the numerical solution of
subproblem~\eqref{eq:reduced_cost_active}. It is well-known that this problem can be
reformulated as a second order cone constrained linear optimization problem, by
introducing \(\#{\mathcal{A}} + 1\) additional variables. Such problems can be solved
efficiently by interior point methods. Since we can bound the number of active points
\(\#{\mathcal{A}}\) a~priori by \(2NM+1\), the cost for the approximate numerical solution
of~\eqref{eq:reduced_cost_active} (up to machine precision) can be regarded as a constant;
see, e.g., \cite{BoydVandenberghe:2004}.
In practice, we choose to implement a semi-smooth Newton method; see, e.g.,
\cite{MilzarekUlbrich:2014}. While there are no complexity bounds for this class of methods,
the local superlinear convergence properties (which, in contrast to interior point
methods, allows for warm starts) makes this alternative seem appealing, since we have a
potentially good initial guess for \(\widehat{\coeff{u}}\) from the previous iteration.


\section{Numerical Results}
\label{sec:numerics}

In this section we briefly describe the discretization methods used for the solution of
the Helmholtz equation in a bounded domain and for the sources from $\Mcon$.
Let $p=p_n^1+\i p_n^2$, $n=1,\ldots,N$ the solution of \eqref{eq:state_n} for the control
$u_n=u_n^1+\i u_n^2$, $u\in \Mcon$.
For the numerical computations we rewrite the state equation \eqref{eq:state_n} in following equivalent real-valued form
\begin{equation}\label{eq:state_n_real}
\begin{curlyeq}
\left(
  \begin{array}{cc}
    -\Lap-\kay_n^2I & 0 \\
    0 & -\Lap-\kay_n^2I \\
  \end{array}
\right)
\left(
  \begin{array}{c}
    p_n^1 \\
    p_n^2 \\
  \end{array}
\right)
&=\frac{1}{w^n}\left(
   \begin{array}{c}
     u_n^1\rvert_{\Omega} \\
     u_n^2\rvert_{\Omega}  \\
   \end{array}
 \right)
&& \text{in } \Omega,\\
\left(
  \begin{array}{cc}
    \partial_\nu & \kappa_n\chi_{\GZ} \\
    -\kappa_n\chi_{\GZ} & \partial_\nu \\
  \end{array}
\right)
\left(
  \begin{array}{c}
    p_n^1 \\
    p_n^2 \\
  \end{array}
\right)
&=\frac{1}{w^n}\left(
   \begin{array}{c}
     u_n^1\rvert_{\Gamma} \\
     u_n^2\rvert_{\Gamma} \\
   \end{array}
 \right)
&& \text{on } \Gamma,
\end{curlyeq}
\end{equation}
where $w$ is one of the weight functions introduced in Section
\ref{sec:weighted_norm}. Based on this formulation of the state equation we employ linear
finite elements on a triangulation of $\Omega$ for the approximation of the state
variables $p_n^1$ and $p_n^2$; cf.~\cite{Ihlenburg:1998,BermudezGamalloRoriguez:2004,DokmanicVetterli:2012}.
We only mention that the discretized state equation has unique
and stable solutions $(p_h^1,p_h^2)$ for a small enough grid size $h$; see, e.g.,
\cite[Theorem~4.4]{BermudezGamalloRoriguez:2004}.
We denote the set of grid nodes in the triangulation with $\mathcal{N}$. Moreover we denote
the number of grid points with $N_h$ and denote number of grid nodes in $\Oc$ with
$N_c$.
Corresponding to the discretization of the state space by finite elements, we discretize
the control space by Dirac-delta functions in the gird nodes (see~\cite{CasasClasonKunisch:2012}):
\begin{equation}\label{eq:discrete_control_space}
\M_h = \left\{u\in \Mcon \;\Big|\;
    u=\sum_{i=1}^{N_c}\coeff{u}_i\delta_{x_i},\;\coeff{u}_i\in \C^N, x_i\in \Oc\cap \mathcal{N}\right\}.
\end{equation}
Since the measure is discretized in the grid nodes, we only need to compute the values of
the weight \(w\) in the grid nodes to obtain a fully discrete problem.
For instance, for the weight function
$w^n_{\Omega,2}=\sqrt{\sum_{m=1}^M\abs{G_n^{x_m}}^2}$, the functions $G_n^{x_m}$ are
approximated again by linear finite elements. Based on the pointwise values of the  finite
element approximations we obtain a discrete approximation of the given weight in the grid
nodes.

We introduce the discrete reweighed observation mapping
$\SO_h^w \colon \M_h\rightarrow \C^{NM}$ defined by
\[
  \SO^w_h\colon u\mapsto \{p_{n,h}^1(x_m)+\i p_{n,h}^2(x_m)\}_{n,m=1}^{N,M}.
\]
Based on the operator $S^w_h$ we formulate the reweighed discrete control problem
\begin{equation}\label{eq:reduced_cost_discrete}
  \begin{aligned}
    &\min_{u\in \M_h} j_h(u)= \frac{1}{2}\sum_{m=1}^{M}\abs{(\SO_h^w u)_m - p_d^m}_{\C^N}^2 + \alpha\norm{u}_{\Mcon}.
  \end{aligned}
\end{equation}
For an $u\in\M_h$ the regularization functional has the form
\[
\alpha\norm{u}_{\Mcon}=\alpha\sum_{j=1}^{N_c}\abs{\coeff u_j}_{\C^N}.
\]
Thus, problem~\eqref{eq:reduced_cost_discrete} is a finite dimensional non-smooth and convex
optimization problem. There are several algorithms which can be used for its
solution. For example, the CVX toolbox~\cite{CVX:2014} reformulates the problem as a cone
constrained problem and solves the resulting problem using an interior point method. While
highly efficient for medium sized problems, the
performance of such a method suffers dramatically from the high dimension $2NN_c$ of the
optimization variable in problem~\eqref{eq:reduced_cost_discrete} (in the case of a fine
discretization).

Finally, we implement the algorithms from section~\ref{sec:optimization} on the discrete
level. To adapt Algorithm~\ref{alg:SPINAT} and
Algorithm~\ref{alg:PDAP} to the discrete level, it suffices to note that the maximization
of the adjoint variable \(\xi^k\) needs to be performed only over the grid points, which
is done by a direct search. The other steps can be implemented directly.
Since the dimension of the observation $2NM$ is
low in comparison to $\dim \mathcal{M}_h = 2NN_c$, we build up the matrix representation
$(\coeff S^w)^* \in \C^{NM\times NN_c}$ of $(\SO_h^w)^\ast$ in a preprocessing step.
This step involves $M$-times the solution of the discrete adjoint state equation. By transposition we get the matrix
representation $\coeff S^w$ of $\SO_h^w$. Note that this matrix is often referred to as
the mixing matrix of a microphone array in Beamforming applications; see
\cite{RouxBoufounosKangHershey:2013}. Thus, the evaluation of the solution operator and the adjoint equation
needed for the application of Algorithm~\ref{alg:SPINAT} resp.~\ref{alg:PDAP} reduces to a
matrix vector multiplication. Due to the convergence analysis on the continuous
level, we can expect the algorithms to behave independently of the number of grid points,
where the cost of each iteration scales linearly in \(N_c\).

\subsection{Interpretation of discrete solutions}
It is known that a discretization of a measure on a finite grid introduces artifacts:
Roughly speaking, a source present in the continuous problem at a off-grid location tends
to appear spread out over the adjacent grid cells, which artificially increases the
number of support points in the discrete solution, and makes the direct interpretation
of the numerical solutions difficult. For a theoretical analysis of this effect
we refer to~\cite{DuvalPeyre:2015}. For practical purposes, we employ the following
post-processing strategy: First, we build the connectivity graph of the sparsity pattern of the finite element
discretization, and interpret all point sources less than two nodes away from each other
as part of a cluster. Then, for each cluster we replace the sources of hat cluster by a source located at
the center of gravity of the cluster with a coefficient given by the sum of the
coefficients. Mathematically, this can be regarded as an interpolation operation on
the space of measures, which introduces an additional error proportional to \(h\)
under reasonable assumptions.

\subsection{Numerical experiments}
In this section we conduct several numerical experiments based on an acoustic inverse source
problem involving the Helmholtz equation. In all considered scenarios we are given a
computational domain $\Omega$ with reflecting as well as absorbing boundary conditions.
We give examples to demonstrate the applicability of
the general approach, and investigate the influence of the choice of
the weight \(w\) and the performance of the presented algorithms.
In all examples, we use the following setting:
\begin{itemize}
\item The computational domain is given by a square of four by four meters, i.e.,
  $\Omega = [0,4]^{2}$.
\item The computational grid $\mathcal T_h$ is given by an uniform triangular
  discretization of $\Omega$ with $h = {\sqrt{2}}/{2^{l}}$ with grid level \(l \in
  \set{6,\ldots,9}\).
\item Two reflecting walls $\Gamma_N$ are located on the left and top and two absorbing walls
  $\Gamma_{Z}$ (with \(\kappa_n = \kay_n\)) on the bottom and right.
\item The speed of sound is set to $c = 345\,$[m/s].
\end{itemize}

\subsubsection{Deterministic comparison of weights}
The results of Proposition~\ref{prop:exact_reconstruction} show that one point source can
be exactly recovered in the noise free case for the weighted
approach~\eqref{eq:weighted_problem}. However, we can construct a
simple example, which numerically demonstrates that the reconstruction
based on the non-weighted approach~\eqref{eq:pre_opt_prob} does not necessarily yield the exact
positions and intensities in this scenario.
To this purpose, we choose an exact source located close to the reflecting boundaries
of $\Omega$ and compute a minimum norm solution for different problem formulations. More
precisely, we set $u^\star = e^{i\pi/4}\delta_{x^\star}$ with $x^\star=(0.5; 3.75)$. Furthermore, for
simplicity, we consider the case with only one frequency $\omega=2\pi\,261.6$, which
corresponds to the tone C4, and
three microphones located in $(3.75,1),~(3.75,2),~(3.75,3)$ as depicted in Figure~\ref{fig:domain}.

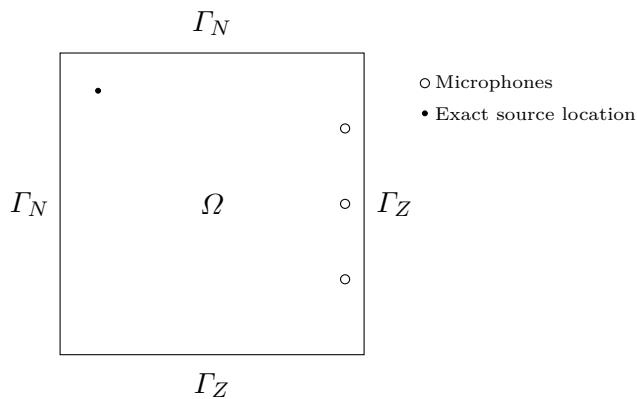
\begin{figure}[htp]
\begin{center}
\begin{tikzpicture}[scale=4]
\draw (0,0) rectangle (1,1);
\node at (.5,.5) {$\Omega$};
\node at (-0.1,.5) {$\Gamma_{N}$};
\node at (1.1,.5) {$\Gamma_{Z}$};
\node at (0.5,1.1) {$\Gamma_{N}$};
\node at (0.5,-0.1) {$\Gamma_{Z}$};
\draw[fill] (0.5/4,3.5/4) circle [radius=0.008];
\foreach \x in {1,...,3}
{\draw (3.75/4,\x/4) circle [radius=.015];}
\draw (1.2,.9) circle [radius=.015];
\node[right] at (1.2,.9) {\tiny Microphones};
\draw[fill] (1.2,.8) circle [radius=0.008];
\node[right] at (1.2,.8) {\tiny Exact source location};
\end{tikzpicture}
\caption{The computational domain $\Omega$, the array of microphones and the exact source position.}
\label{fig:domain}
\end{center}
\end{figure}

Since we compare different problems settings under ideal conditions, we consider noise-free
observations which are generated on the same grid as the subsequent
computations. Therefore, we set $p_d=p_h(u^\star)$ generated by solving the discrete Helmholtz
equation~\eqref{eq:state} with the exact source $u^\star$.
\begin{figure}[htb]
\begin{subfigure}[t]{.48\linewidth}
\centering
\includegraphics[trim = .5cm .5cm .7cm .5cm, clip,height=4.2cm]{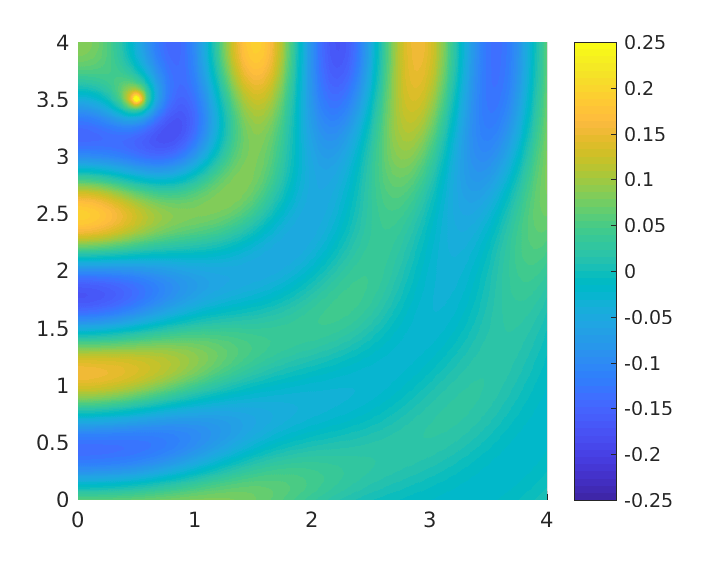}\\
\caption{Real part of the acoustic pressure $p_h(u^\star)$.}\label{fig:state}
\end{subfigure}
\begin{subfigure}[t]{.45\linewidth}
\centering
\includegraphics[trim = .5cm .5cm .75cm .5cm, clip,height=4.2cm]{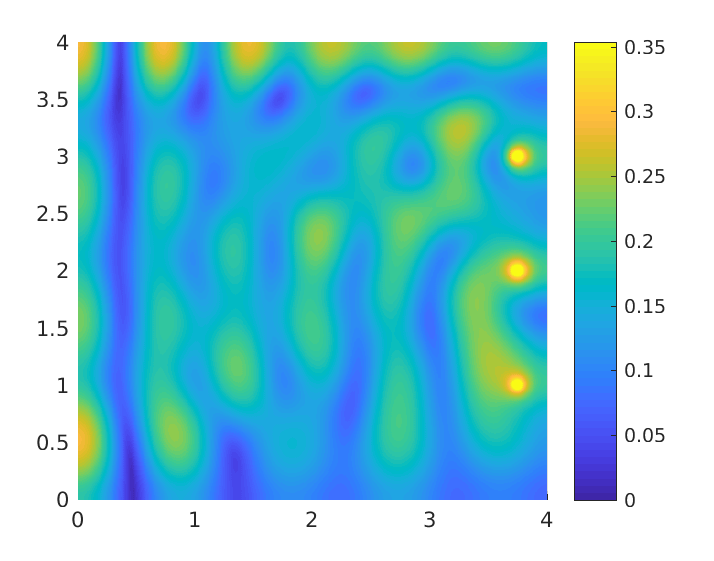}\\
\caption{Discrete approximation of $w_{\Omega,2}$.}\label{fig:weight}
\end{subfigure}
\caption{Exact pressure and weight \(w = w_{\Omega,2}\).}
\end{figure}
In Figure~\ref{fig:state} the real part of the acoustic pressure $p(u^\star)$ is
displayed. Circular waves are generated from the point source and intensified by the
reflections on $\Gamma_{N}$. Figure~\ref{fig:weight} shows the
weight $w = w_{\Omega,2} = \sqrt{\sum_{m=1}^M\abs{G^{x_m}}^2}$.
As mentioned before, the value of the weight at point in the domain corresponds to the
magnitude of the signal that will be received at the microphones.
We clearly see that $w$ has a relatively low value in a
neighborhood of the exact source position. This
behavior of $w$ is caused by negative interference of the generated and
reflected waves. Furthermore, we clearly observe the large values of the weight close to
the microphones.

In the following, we numerically approximate the minimum norm solutions \(u^\dagger\) for
different weights.
To this purpose, we solve the respective discrete problems for a decreasing sequence of cost
parameters ($\alpha=10^{-0},\ldots,10^{-10}$)
up to machine precision (using Algorithm~\ref{alg:PDAP}). Then, we take the solution
\(\widehat{u}_\alpha\) for the smallest \(\alpha\) as an approximation of \(u^\dagger\)
(which is justified by Corollary~\ref{cor:inverse_problem_convergence}).
Furthermore, an approximation of the element \(\xi^\dagger\) from the source
condition~\eqref{eq:source_condition} is given by
\(\widehat{\xi}_\alpha = -\SO^*(\SO\widehat{u}_\alpha - p_d)/\alpha\).

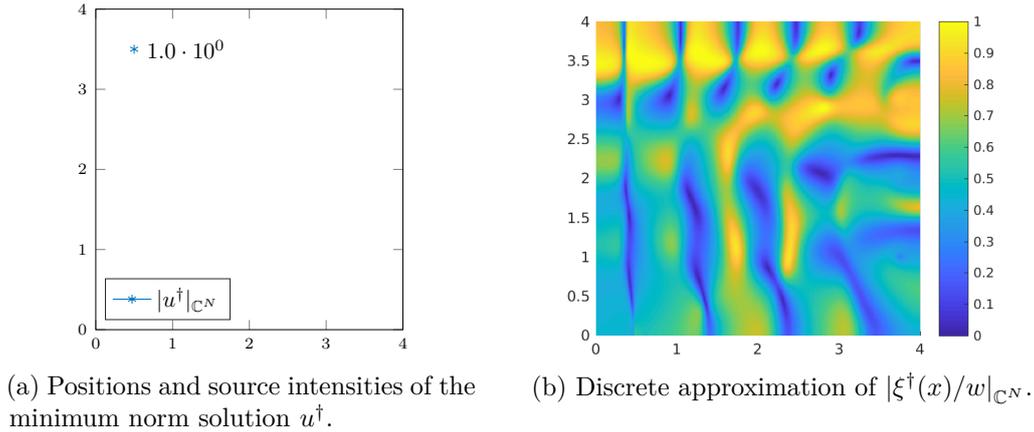
\begin{figure}[htb]
\begin{subfigure}[t]{.45\linewidth}
\centering
\scalebox{.85}{
%
%
\definecolor{mycolor1}{rgb}{0.00000,0.44700,0.74100}%
\begin{tikzpicture}

\begin{axis}[%
width=4.755cm,
height=5cm,
at={(0cm,0cm)},
scale only axis,
xmin=0,
xmax=4,
ymin=0,
ymax=4,
axis background/.style={fill=white},
legend style={at={(0.03,0.03)}, anchor=south west, legend cell align=left, align=left, draw=white!15!black}
]
\addplot [color=mycolor1, draw=none, mark=asterisk, mark options={solid, mycolor1}]
  table[row sep=crcr]{%
0.5	3.5\\
};
\addlegendentry{\small $|u^\dagger|_{\mathbb{C}^N}$}

\node[left, align=right]
at (axis cs:1.8,3.5) {\small $1.0\cdot 10^{0}$};
\end{axis}
\end{tikzpicture}%
}
\caption{Positions and source intensities of the minimum norm solution \(u^\dagger\).}
\label{fig:distributiondirac_weight_no_cd}
\end{subfigure}\quad\quad
\begin{subfigure}[t]{.48\linewidth}
\centering
\includegraphics[trim = .5cm .5cm .75cm .6cm, clip,height=4.5cm]{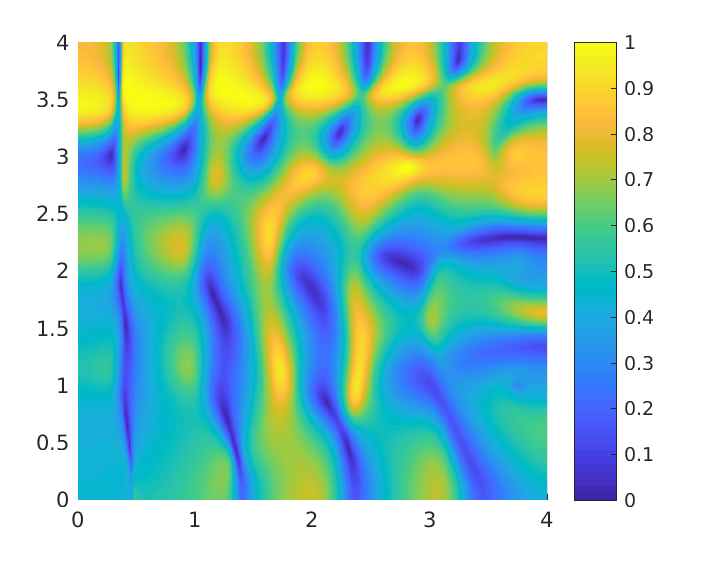}
\caption{Discrete approximation of $\abs{\xi^\dagger(x)/w}_{\C^N}$.}\label{fig:adjoint_weight_cd}
\end{subfigure}
\caption{Noise free reconstructions for $\Oc=\Omega$, weight \(w = w_{\Omega,2}\).}
\end{figure}

We give the results for \(w = w_{\Omega,2}\) in
Figure~\ref{fig:distributiondirac_weight_no_cd}.
Here, for the reconstruction we admit all possible sources and set \(\Oc = \Omega\).
In agreement with Proposition~\ref{prop:exact_reconstruction} we observe that the support of the solution is
recovered exactly, and that the coefficient coincides to the exact one up to the seventh
digit. Moreover, a close inspection of the variable $\abs{\xi^\dagger(x)/w}_{\C^N}$ shows
that its maximum value one is uniquely attained at the exact source position; the next
biggest local minimum has a value of \(\sim 0.995\). This demonstrates uniqueness of the
discrete minimum norm solution in this case (cf.\ Proposition~\ref{prop:criterion_uniqueness}).

\begin{figure}[htb]
\begin{subfigure}[t]{.45\linewidth}
\centering
\scalebox{.85}{
%
%
\definecolor{mycolor1}{rgb}{0.00000,0.44700,0.74100}%
\begin{tikzpicture}

\begin{axis}[%
width=4.755cm,
height=5cm,
at={(0cm,0cm)},
scale only axis,
xmin=0,
xmax=4,
ymin=0,
ymax=4,
axis background/.style={fill=white},
legend style={at={(0.03,0.03)}, anchor=south west, legend cell align=left, align=left, draw=white!15!black}
]
\addplot [color=mycolor1, draw=none, mark=asterisk, mark options={solid, mycolor1}]
  table[row sep=crcr]{%
3.75	2\\
3.75	1\\
3.75	3\\
};
\addlegendentry{\small $|u^\dagger|_{\mathbb{C}^N}$}

\node[left, align=right]
at (axis cs:3.75,2) {\small $5.7\cdot 10^{-2}$};
\node[left, align=right]
at (axis cs:3.75,1) {\small $4.1\cdot 10^{-2}$};
\node[left, align=right]
at (axis cs:3.75,3) {\small $1.3\cdot 10^{-1}$};
\end{axis}
\end{tikzpicture}%
}
\caption{Positions and source intensities of the minimum norm solution \(u^\dagger\).}
\label{fig:distribution_diracs_no_weight_no_cd}
\end{subfigure}\quad\quad
\begin{subfigure}[t]{.48\linewidth}
\centering
\includegraphics[trim = .5cm .5cm .75cm .6cm, clip,height=4.5cm]{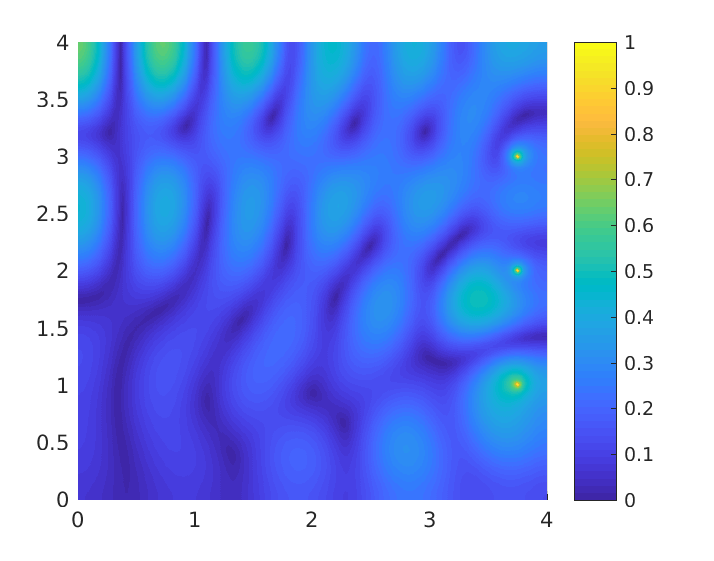}
\caption{Discrete approximation of $\abs{\xi^\dagger(x)}_{\C^N}$.}\label{fig:adjoint_no_weight_no_cd}
\end{subfigure}
\caption{Noise free reconstructions for $\Oc=\Omega$, no weight (\(w \equiv 1\)).}
\end{figure}
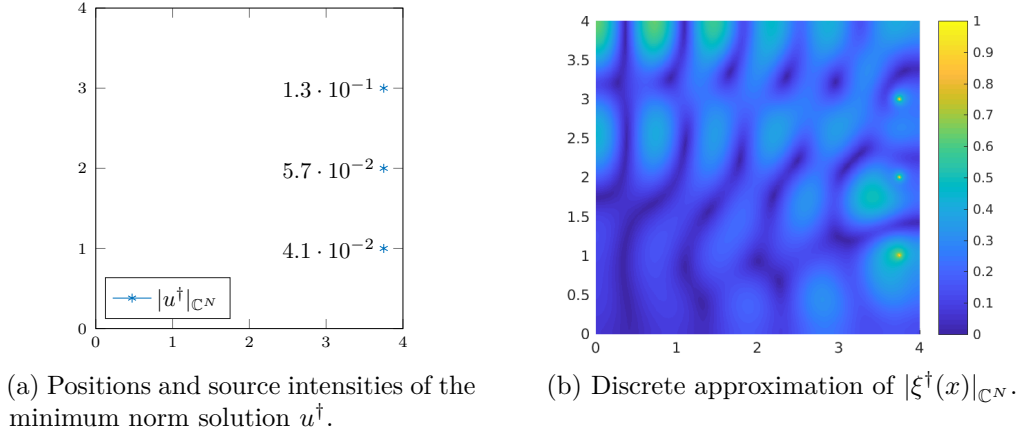
Next, we consider the case without weight.
According to Proposition~\ref{prop:counter_example_weight} the corresponding problem with
\(\Oc = \Omega\setminus\Xi\) has no solution since there exists vanishing sequences of point sources
which generate the exact measurements and converge to the positions of the
microphones. However, in the discrete
setting the problem always has a solution, since the discrete Green's functions are bounded by a
mesh-dependent constant. We give the numerical results in
Figure~\ref{fig:distribution_diracs_no_weight_no_cd}. Here, the minimum norm solution $u^\dagger$
consists of three point sources located in the microphone positions.
The maximum of the absolute value of the adjoint state is assumed only there; see
Figure~\ref{fig:adjoint_no_weight_no_cd}. Note that this numerical solution is highly
sensitive to the grid resolution. In fact, for \(h \to 0\) the minimum norm solution
and dual variable converge to zero.

\begin{figure}[htb]
\begin{subfigure}[t]{.45\linewidth}
\centering
\scalebox{.85}{
%
%
\definecolor{mycolor1}{rgb}{0.00000,0.44700,0.74100}%
\begin{tikzpicture}

\begin{axis}[%
width=4.755cm,
height=5cm,
at={(0cm,0cm)},
scale only axis,
xmin=0,
xmax=4,
ymin=0,
ymax=4,
axis background/.style={fill=white},
legend style={at={(0.03,0.03)}, anchor=south west, legend cell align=left, align=left, draw=white!15!black}
]
\addplot [color=mycolor1, draw=none, mark=asterisk, mark options={solid, mycolor1}]
  table[row sep=crcr]{%
2.75	4\\
1.359375	4\\
2.046875	4\\
2.703125	2.53125\\
2.046875	2.296875\\
};
\addlegendentry{\small $|u^\dagger|_{\mathbb{C}^N}$}

\node[left, align=right]
at (axis cs:4,3.7) {\small $1\cdot 10^{-1}$};
\node[left, align=right]
at (axis cs:1.4,3.7) {\small $2.9\cdot 10^{-1}$};
\node[left, align=right]
at (axis cs:2.8,3.6) {\small $3.5\cdot 10^{-2}$};
\node[left, align=right]
at (axis cs:2.703,2.6) {\small $8.5\cdot 10^{-2}$};
\node[left, align=right]
at (axis cs:2.047,2.2) {\small $5.5\cdot 10^{-2}$};
\end{axis}
\end{tikzpicture}%
}
\caption{Positions and source intensities of the minimum norm solution \(u^\dagger\).}
\label{fig:distribution_no weight_cd}
\end{subfigure}\quad\quad
\begin{subfigure}[t]{.48\linewidth}
\centering
\includegraphics[trim = .5cm .5cm .75cm .6cm, clip,height=4.5cm]{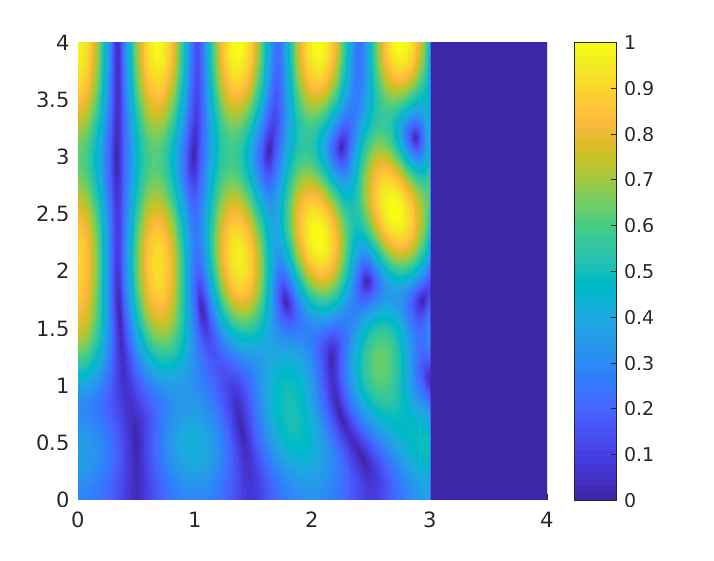}
\caption{Discrete approximation of $\abs{\xi^\dagger(x)}_{\C^N}$.}\label{fig:adjoint_no_weight_cd}
\end{subfigure}
\caption{Results for $\Oc=[0,3]\times [0,4]$, no weight (\(w \equiv 1\)).}
\end{figure}
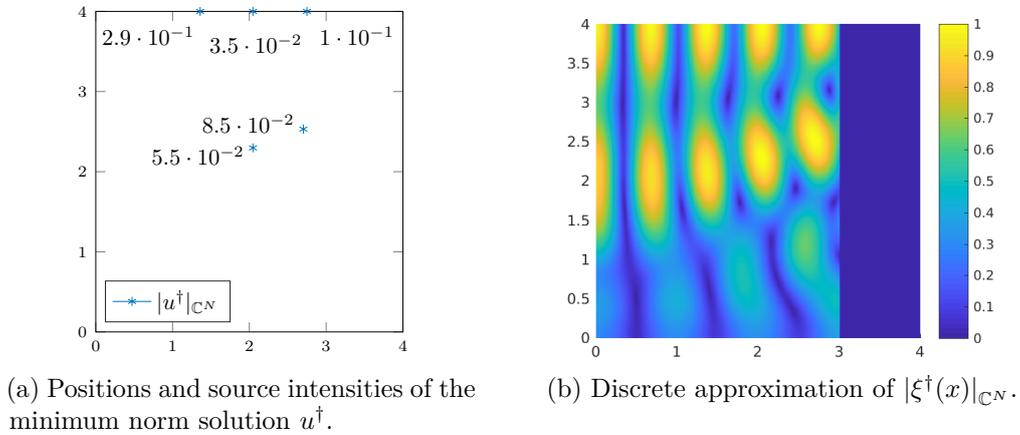
To obtain a well-posed optimization problem without weight we choose the control
domain $\Oc=[0,3]\times [0,4]$, which excludes the observation positions. The results are
given in Figure~\ref{fig:distribution_no weight_cd}, where we observe that the
optimal solution consists of five point sources: three are located on the reflecting
boundary $\Gamma_{N}$ and three are located in the interior of the domain.
The corresponding function $\abs{\xi^\dagger(x)}_{\C^N}$
attains its global maximum on the support points of \(u^\dagger\). However, the region
close to the exact source position assumes a visibly lower function value, and no source
is placed there. This can be connected to the negative interference at this point; cf.\
Figure~\ref{fig:weight}.

These examples show that even in simple settings the reconstruction results of the
non-weighted approach~\eqref{eq:pre_opt_prob} is affected by negative interference
caused by the reflecting boundaries, as well as the fact that the adjoint state takes
arbitrarily large values close to the microphone positions.

\subsubsection{Statistical comparison of weights}

Now, we consider a more involved problem setup to evaluate
the reconstruction quality for different weights. We consider the
same model as before, but consider the
frequencies \(\omega = 2 \pi (349.2, 523.3, 659.3)\) (corresponding to F4, C5, and E5).
The number of microphones is increased to $30$, and
the control domain is chosen as \(\Oc = [0,3]\times[0,4]\), which does not contain the
microphone locations at \((x_1,x_2)\) with \(x_1 = 3.25\) and \(x_1 = 3.75\)
and \(x_2\) regularly spaced from \(0\) to \(4\); see Figure~\ref{fig:setup_compare}.
All computation are performed on grid level \(l=8\).

To evaluate to reconstruction quality of different weights, we follow a statistical
approach: for each number of point sources \(N_d^\star \in \set{1,2,\ldots,5}\), we
generate a random source by selecting \(N_d^\star\) random indices from the
mesh nodes on the control domain and generating corresponding random coefficients by
drawing from a multivariate complex
Gaussian distribution with unit variance. Then, we compute a minimum norm
solution~\eqref{eq:reduced_cost_zero_w} from the corresponding exact observations for the
given weight, which is either \(w \equiv 1\) or \(w = w_{\Omega,2}\).
Here, we again approximate the minimum norm solution by the solution for a value of
\(\alpha = 10^{-9}\), which we compute by a continuation strategy in
the regularization parameter using Algorithm~\ref{alg:PDAP}.

Finally, we evaluate the average reconstruction error for each weight. Since the
generalized Bregman distance is multivalued, we focus on two simple citeria. The first is
simply the relative difference of the norms with respect to the employed weight,
\begin{equation}
\label{eq:rel_norm_diff}
e_1 = \left[\norm{u^\star}_{\M_w(\Oc,\C^N)} - \norm{u^\dagger}_{\M_w(\Oc,\C^N)}\right]/\norm{u^\star}_{\M_w(\Oc,\C^N)}
\end{equation}
Note that it can be easily verified that
\(\norm{u^\star}_{\M_w(\Oc,\C^N)} - \norm{u^\dagger}_{\M_w(\Oc,\C^N)} \in
D(u^\star,u^\dagger)\) (for the specific choice \(\xi = \SO^* y^\dagger\)), which relates
this criterion to the Bregman distance; cf. Theorem~\ref{thm:conv_breg}.
The results are given in Figure~\ref{fig:rel_norm_diff}. We observe that the difference
is smaller for the weight \(w_{\Omega,2}\), and that it is zero for the case of one
source, as predicted by theory.
However, we can expect the norm difference to severely underestimate the reconstruction
error. Moreover, the results for different weights are not directly comparable, due to the
fact that the error criterion itself depends on the weight.
Therefore, we also consider a second error criterion, which is based on convolution. We
introduce the componentwise convolution operator \(S^\sigma_{\mathrm{heat}} \colon
\M(\Omega,\C^N) \to L^1(\Omega,\C^N)\), which computes the solution at time
\(T = \sigma^2/2\) of the heat equation (endowed with homogeneous Neumann boundary conditions on
the domain \(\Omega\)) with the given initial data at time zero.
Then we define the second error criterion by
\begin{equation}
\label{eq:rel_kern_diff}
e_2 =
\norm{S^\sigma_{\mathrm{heat}}(u^\star - u^\dagger)}_{L^1(\Omega,\C^N)}/\norm{u^\star}_{\M(\Omega,\C^N)}.
\end{equation}
Here, we compare the reconstruction error in the canonical norm after convolution with a
regular kernel with approximate width \(\sigma\). Roughly speaking, we can expect small
errors in the source location to lead to small error terms (which is not the case if we apply
the total variation norm directly), whereas location errors larger than \(\sigma\) lead to
big error contributions. Mathematically, the backwards uniqueness property of the heat
equation guarantees that \(e_2 = 0\) can only occur for \(u^\dagger = u^\star\).  We
implement \(S^\sigma_{\mathrm{heat}}\) by a finite element
approximation on the given grid and an implicit Euler time discretization (with five steps).
The results for \(\sigma = 0.2\) and \(\sigma = 0.05\) are given in
Figures~\ref{fig:rel_kern_diff} and~\ref{fig:rel_kern_diff2}, respectively.
We observe that, although the errors increase for more strict error criteria, the average
errors are consistently smaller when the weight \(w_{\Omega,2}\) is employed.
\begin{figure}[htb]
\begin{subfigure}[t]{.3\linewidth}
\centering
\begin{tabular}{ccc}
  \toprule
  \(N^*_d\) & \(w\equiv1\) & \(w=w_{\Omega,2}\) \\
  \midrule
  1 & 0.0087  &  0.0000  \\
  2 & 0.0233  &  0.0030  \\
  3 & 0.0599  &  0.0174  \\
  4 & 0.0867  &  0.0404  \\
  5 & 0.1443  &  0.0754  \\
\bottomrule
\end{tabular}
\caption{Average relative norm error~\eqref{eq:rel_norm_diff}.}
\label{fig:rel_norm_diff}
\end{subfigure}
\quad
\begin{subfigure}[t]{.3\linewidth}
\centering
\begin{tabular}{ccc}
  \toprule
  \(N^*_d\) & \(w\equiv1\) & \(w=w_{\Omega,2}\) \\
  \midrule
  1 & 0.0875  &  0.0000 \\
  2 & 0.1894  &  0.0387 \\
  3 & 0.4364  &  0.2042 \\
  4 & 0.6416  &  0.4394 \\
  5 & 0.8326  &  0.6691 \\
  \bottomrule
\end{tabular}
\caption{Convolution error~\eqref{eq:rel_kern_diff} with \(\sigma = 0.2\).}
\label{fig:rel_kern_diff}
\end{subfigure}
\quad
\begin{subfigure}[t]{.3\linewidth}
\centering
\begin{tabular}{ccc}
  \toprule
  \(N^*_d\) & \(w\equiv1\) & \(w=w_{\Omega,2}\) \\
  \midrule
  1 & 0.1453  &  0.0000 \\
  2 & 0.2660  &  0.0625 \\
  3 & 0.6129  &  0.2971 \\
  4 & 0.8689  &  0.6271 \\
  5 & 1.1181  &  0.9556 \\
  \bottomrule
\end{tabular}
\caption{Convolution error~\eqref{eq:rel_kern_diff} with \(\sigma = 0.05\).}
\label{fig:rel_kern_diff2}
\end{subfigure}
\caption{Average reconstruction error for 200 randomly generated sources with
  different numbers of point sources \(N_d^\star\).}
\end{figure}

\subsubsection{Comparison of algorithms}

Now, we evaluate the practical performance of the algorithms from
section~\ref{sec:optimization}. We consider the same setting as in the previous section
(frequencies \(\omega = 2 \pi (349.2, 523.3, 659.3)\) and $30$ microphones).
We recover a source consisting of three point sources as depicted in
Figure~\ref{fig:setup_compare} with random coefficients (drawn from
a multivariate complex Gaussian distribution with unit variance). The control domain is
chosen as \(\Oc = [0,3]\times[0,4]\) and the weight \(w_{\Omega,2}\) is employed in all experiments.

We want to study the algorithms for a setting with noise and useful values of the
parameter \(\alpha\). Therefore, we compute synthetic measurements on the finest
grid level \(l = 9\) and perturb them by additive Gaussian noise, such that
\(\norm{Su^\star - p_d}/\norm{Su^\star} = 5\%\).
We then solve the problem on a coarser grid level \(l = 8\), to also take into account a
possible discretization error.
To determine a useful range of regularization parameters, we numerically compute an
L-curve: we solve the problem~\eqref{eq:reduced_cost_w} for a sequence of regularization
parameters \(\alpha_j = 10^{-j/4}\), \(j = 0,1,\ldots,20\) and plot the norm of the
solution \(\widehat{u}_\alpha\) over the data misfit term
\(\norm{S\widehat{u}_\alpha - p_d}_{\C^{NM}}\); see Figure~\ref{fig:lcurve}.
We observe that the data misfit term is reduced below the noise level at
\(\alpha_7 \approx 1.8 \cdot 10^{-2}\) (corresponding to the popular Morozov-criterion for the selection of
a regularization parameter), and at \(\alpha_9 \approx 5.6 \cdot 10^{-3}\) the norm of
the reconstruction starts to exceed the norm of the exact solution \(u^\star\).
We conclude that practically relevant values of \(\alpha\) are around
\(10^{-2}\) in this particular instance.

\begin{figure}[htb]
\begin{subfigure}[t]{.45\linewidth}
\centering
\scalebox{.85}{
%
%
\definecolor{mycolor1}{rgb}{0.00000,0.44700,0.74100}%
\begin{tikzpicture}

\begin{axis}[%
width=5.706cm,
height=6cm,
at={(0cm,0cm)},
scale only axis,
xmin=0,
xmax=4,
ymin=0,
ymax=4,
axis background/.style={fill=white},
legend style={at={(0.03,0.03)}, anchor=south west, legend cell align=left, align=left, draw=white!15!black}
]
\addplot [color=mycolor1, draw=none, mark=asterisk, mark options={solid, mycolor1}]
  table[row sep=crcr]{%
2	2\\
1.5	0.5\\
1	3.5\\
};

\addplot [color=red, draw=none, mark=o, mark options={solid, red}]
  table[row sep=crcr]{%
3.25	0.25\\
3.25	0.5\\
3.25	0.75\\
3.25	1\\
3.25	1.25\\
3.25	1.5\\
3.25	1.75\\
3.25	2\\
3.25	2.25\\
3.25	2.5\\
3.25	2.75\\
3.25	3\\
3.25	3.25\\
3.25	3.5\\
3.25	3.75\\
3.75	0.25\\
3.75	0.5\\
3.75	0.75\\
3.75	1\\
3.75	1.25\\
3.75	1.5\\
3.75	1.75\\
3.75	2\\
3.75	2.25\\
3.75	2.5\\
3.75	2.75\\
3.75	3\\
3.75	3.25\\
3.75	3.5\\
3.75	3.75\\
};

\addplot [color=black]
  table[row sep=crcr]{%
3	0\\
3	4\\
};

\node[left, align=right]
at (axis cs:2,2) {\small $2.5\cdot 10^{0}$};
\node[left, align=right]
at (axis cs:1.5,0.5) {\small $1.2\cdot 10^{0}$};
\node[left, align=right]
at (axis cs:1,3.5) {\small $1.7\cdot 10^{0}$};
\end{axis}
\end{tikzpicture}%
}
\caption{Exact source locations and source intensities on the left, microphone
  locations on the right.}
\label{fig:setup_compare}
\end{subfigure}
\quad\quad
\begin{subfigure}[t]{.45\linewidth}
\centering
\scalebox{.85}{
%
%
\definecolor{mycolor1}{rgb}{0.00000,0.44700,0.74100}%
\definecolor{mycolor2}{rgb}{0.85000,0.32500,0.09800}%
\definecolor{mycolor3}{rgb}{0.92900,0.69400,0.12500}%
\begin{tikzpicture}

\begin{axis}[%
width=5.706cm,
height=6cm,
at={(0cm,0cm)},
scale only axis,
xmode=log,
xmin=0.0227001398238849,
xmax=1.48441433828878,
xminorticks=true,
ymode=log,
ymin=0.478024512942234,
ymax=11.669667988265,
yminorticks=true,
axis background/.style={fill=white},
legend style={legend cell align=left, align=left, draw=white!15!black}
]
\addplot [color=mycolor1, dashed, mark=*, mark size=1.2pt, mark options={solid, mycolor1}]
  table[row sep=crcr]{%
1.23701194857398	0\\
0.867908378150196	0.531138347713593\\
0.564379775671043	1.02803076188303\\
0.32566101680593	1.45896923299507\\
0.190992571731227	1.70969816729845\\
0.118283611471946	1.85376728016957\\
0.081429815308487	1.93776506660221\\
0.0639914086563173	1.98947616185614\\
0.0543715611380531	2.03133937571817\\
0.0495539418679335	2.06372511156217\\
0.0474778395274618	2.08690190545796\\
0.0464972699437735	2.10601033118157\\
0.045187467046939	2.15388378086605\\
0.0418040993589147	2.35976676398775\\
0.0379660265100902	2.72161034835362\\
0.0352420608241753	3.13994852668695\\
0.0335872578561933	3.5585842671804\\
0.0325363572564096	4.02356432995591\\
0.0314779052853676	4.84160986512214\\
0.0301532319920666	6.61092308129921\\
0.0283751747798562	10.6087890802409\\
};
\addlegendentry{\tiny $\|\widehat u_\alpha\|_{\mathcal{M}_w(\Omega,\mathbb{C}^N)}$}

\addplot [color=mycolor2]
  table[row sep=crcr]{%
0.00283751747798562	2.05573751111398\\
12.3701194857398	2.05573751111398\\
};
\addlegendentry{\tiny $\|u^\star\|_{\mathcal{M}_w(\Omega,\mathbb{C}^N)}$}

\addplot [color=mycolor3]
  table[row sep=crcr]{%
0.0685415033685305	0.0531138347713593\\
0.0685415033685305	106.087890802409\\
};
\addlegendentry{\tiny $\|S u^\star - p_d\|_{\mathbb{C}^{NM}}$}

\node[right, align=left]
at (axis cs:0.191,1.8) {\small $\alpha_{4} = 10^{-1}$};
\node[right, align=left]
at (axis cs:0.054,2.3) {\small $\alpha_{8} = 10^{-2}$};
\node[right, align=left]
at (axis cs:0.04,2.8) {\small $\alpha_{12} = 10^{-3}$};
\node[right, align=left]
at (axis cs:0.034,3.7) {\small $\alpha_{16} = 10^{-4}$};
\node[right, align=left]
at (axis cs:0.028,10.) {\small $\alpha_{20} = 10^{-5}$};

\addplot [only marks, mark=*, mark size=1.8pt, mark options={solid, black}]
  table[row sep=crcr]{%
0.190992571731227	1.70969816729845\\
0.0543715611380531	2.03133937571817\\
0.045187467046939	2.15388378086605\\
0.0335872578561933	3.5585842671804\\
0.0283751747798562	10.6087890802409\\
};
\end{axis}
\end{tikzpicture}%
}
\caption{
  Norms of the solutions \(\widehat{u}_\alpha\) over the data misfit
  \(\norm{S\widehat{u}_\alpha - p_d}_{\C^{NM}}\) for different \(\alpha\) for noisy
  observations \(p_d\) (5\% noise).
}
\label{fig:lcurve}
\end{subfigure}
\caption{Problem setup and L-curve at grid level \(l=8\).}
\end{figure}
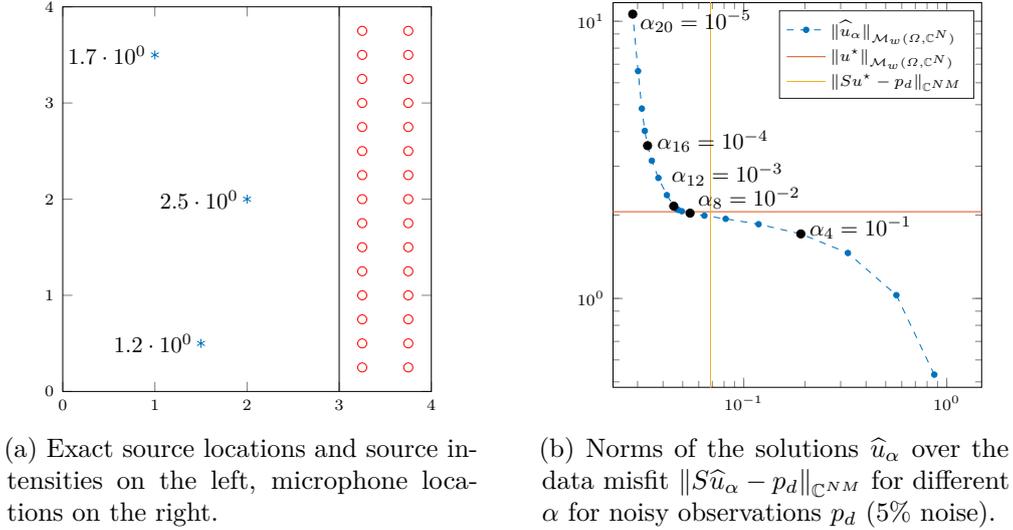

In a first test, we compute reconstructions (on grid level \(l = 8\)) starting from an
initial guess of \(u^0 = 0\) for \(\alpha = 10^{-1},10^{-2},10^{-3}\) with different
algorithms.
A visualization of the corresponding numerical solutions (computed with
Algorithm~\ref{alg:PDAP} up to machine precision) is given in Figure~\ref{fig:uopt}.
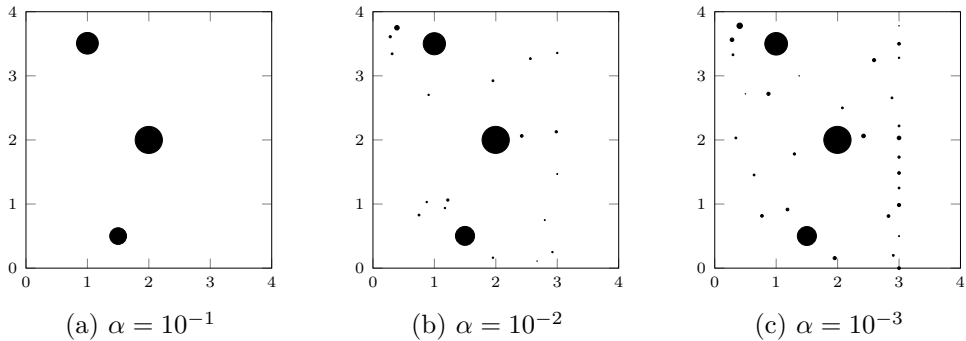
\begin{figure}[htb]
\begin{subfigure}[t]{.33\textwidth}
\centering
\scalebox{.85}{
%
%
\definecolor{mycolor1}{rgb}{0.00000,0.44700,0.74100}%
\begin{tikzpicture}

\begin{axis}[%
width=3.804cm,
height=4cm,
at={(0cm,0cm)},
scale only axis,
xmin=0,
xmax=4,
ymin=0,
ymax=4,
axis background/.style={fill=white}
]
\addplot [color=mycolor1, forget plot]
  table[row sep=crcr]{%
0	0\\
};
\addplot [color=black, draw=none, mark size=6.0pt, mark=*, mark options={solid, fill=black, black}, forget plot]
  table[row sep=crcr]{%
1.99911560817436	2\\
};
\addplot [color=black, draw=none, mark size=4.8pt, mark=*, mark options={solid, fill=black, black}, forget plot]
  table[row sep=crcr]{%
1	3.50783628466933\\
};
\addplot [color=black, draw=none, mark size=3.7pt, mark=*, mark options={solid, fill=black, black}, forget plot]
  table[row sep=crcr]{%
1.5	0.5\\
};
\end{axis}
\end{tikzpicture}%
}
\caption{\(\alpha = 10^{-1}\)}
\label{fig:uopt_1e-1}
\end{subfigure}
\begin{subfigure}[t]{.32\textwidth}
\centering
\scalebox{.85}{
%
%
\definecolor{mycolor1}{rgb}{0.00000,0.44700,0.74100}%
\begin{tikzpicture}

\begin{axis}[%
width=3.804cm,
height=4cm,
at={(0cm,0cm)},
scale only axis,
xmin=0,
xmax=4,
ymin=0,
ymax=4,
axis background/.style={fill=white}
]
\addplot [color=mycolor1, forget plot]
  table[row sep=crcr]{%
0	0\\
};
\addplot [color=black, draw=none, mark size=6.0pt, mark=*, mark options={solid, fill=black, black}, forget plot]
  table[row sep=crcr]{%
1.99779227427805	2.00086197664529\\
};
\addplot [color=black, draw=none, mark size=4.2pt, mark=*, mark options={solid, fill=black, black}, forget plot]
  table[row sep=crcr]{%
1.49848614122952	0.503133033414747\\
};
\addplot [color=black, draw=none, mark size=4.9pt, mark=*, mark options={solid, fill=black, black}, forget plot]
  table[row sep=crcr]{%
0.9994158448772	3.50052666537767\\
};
\addplot [color=black, draw=none, mark size=0.3pt, mark=*, mark options={solid, fill=black, black}, forget plot]
  table[row sep=crcr]{%
3	3.35678905797733\\
};
\addplot [color=black, draw=none, mark size=0.4pt, mark=*, mark options={solid, fill=black, black}, forget plot]
  table[row sep=crcr]{%
2.5625	3.26780416011901\\
};
\addplot [color=black, draw=none, mark size=0.5pt, mark=*, mark options={solid, fill=black, black}, forget plot]
  table[row sep=crcr]{%
2.984375	2.12725713774837\\
};
\addplot [color=black, draw=none, mark size=0.3pt, mark=*, mark options={solid, fill=black, black}, forget plot]
  table[row sep=crcr]{%
1.953125	0.162943485094488\\
};
\addplot [color=black, draw=none, mark size=0.4pt, mark=*, mark options={solid, fill=black, black}, forget plot]
  table[row sep=crcr]{%
1.953125	2.9233292422379\\
};
\addplot [color=black, draw=none, mark size=0.2pt, mark=*, mark options={solid, fill=black, black}, forget plot]
  table[row sep=crcr]{%
3	1.46875\\
};
\addplot [color=black, draw=none, mark size=0.3pt, mark=*, mark options={solid, fill=black, black}, forget plot]
  table[row sep=crcr]{%
0.875	1.03125\\
};
\addplot [color=black, draw=none, mark size=0.5pt, mark=*, mark options={solid, fill=black, black}, forget plot]
  table[row sep=crcr]{%
1.21875	1.0625\\
};
\addplot [color=black, draw=none, mark size=0.4pt, mark=*, mark options={solid, fill=black, black}, forget plot]
  table[row sep=crcr]{%
0.3125	3.34375\\
};
\addplot [color=black, draw=none, mark size=1.0pt, mark=*, mark options={solid, fill=black, black}, forget plot]
  table[row sep=crcr]{%
0.390625	3.75\\
};
\addplot [color=black, draw=none, mark size=0.6pt, mark=*, mark options={solid, fill=black, black}, forget plot]
  table[row sep=crcr]{%
2.421875	2.0625\\
};
\addplot [color=black, draw=none, mark size=0.3pt, mark=*, mark options={solid, fill=black, black}, forget plot]
  table[row sep=crcr]{%
2.921875	0.25\\
};
\addplot [color=black, draw=none, mark size=0.4pt, mark=*, mark options={solid, fill=black, black}, forget plot]
  table[row sep=crcr]{%
0.75	0.828125\\
};
\addplot [color=black, draw=none, mark size=0.2pt, mark=*, mark options={solid, fill=black, black}, forget plot]
  table[row sep=crcr]{%
2.796875	0.75\\
};
\addplot [color=black, draw=none, mark size=0.3pt, mark=*, mark options={solid, fill=black, black}, forget plot]
  table[row sep=crcr]{%
1.171875	0.9375\\
};
\addplot [color=black, draw=none, mark size=0.5pt, mark=*, mark options={solid, fill=black, black}, forget plot]
  table[row sep=crcr]{%
0.28125	3.609375\\
};
\addplot [color=black, draw=none, mark size=0.3pt, mark=*, mark options={solid, fill=black, black}, forget plot]
  table[row sep=crcr]{%
0.90625	2.703125\\
};
\addplot [color=black, draw=none, mark size=0.1pt, mark=*, mark options={solid, fill=black, black}, forget plot]
  table[row sep=crcr]{%
2.671875	0.109375\\
};
\end{axis}
\end{tikzpicture}%
}
\caption{\(\alpha = 10^{-2}\)}
\label{fig:uopt_1e-2}
\end{subfigure}
\begin{subfigure}[t]{.32\textwidth}
\centering
\scalebox{.85}{
%
%
\definecolor{mycolor1}{rgb}{0.00000,0.44700,0.74100}%
\begin{tikzpicture}

\begin{axis}[%
width=3.804cm,
height=4cm,
at={(0cm,0cm)},
scale only axis,
xmin=0,
xmax=4,
ymin=0,
ymax=4,
axis background/.style={fill=white}
]
\addplot [color=mycolor1, forget plot]
  table[row sep=crcr]{%
0	0\\
};
\addplot [color=black, draw=none, mark size=6.0pt, mark=*, mark options={solid, fill=black, black}, forget plot]
  table[row sep=crcr]{%
1.99728731211866	2.00069732467548\\
};
\addplot [color=black, draw=none, mark size=5.0pt, mark=*, mark options={solid, fill=black, black}, forget plot]
  table[row sep=crcr]{%
0.999897035070147	3.49972650691618\\
};
\addplot [color=black, draw=none, mark size=4.2pt, mark=*, mark options={solid, fill=black, black}, forget plot]
  table[row sep=crcr]{%
1.5	0.501189863331968\\
};
\addplot [color=black, draw=none, mark size=0.5pt, mark=*, mark options={solid, fill=black, black}, forget plot]
  table[row sep=crcr]{%
3	1.7319133156928\\
};
\addplot [color=black, draw=none, mark size=0.4pt, mark=*, mark options={solid, fill=black, black}, forget plot]
  table[row sep=crcr]{%
2.88491268213012	2.65625\\
};
\addplot [color=black, draw=none, mark size=0.7pt, mark=*, mark options={solid, fill=black, black}, forget plot]
  table[row sep=crcr]{%
0.874437981202791	2.71931201879721\\
};
\addplot [color=black, draw=none, mark size=0.7pt, mark=*, mark options={solid, fill=black, black}, forget plot]
  table[row sep=crcr]{%
2.59375	3.24464749718241\\
};
\addplot [color=black, draw=none, mark size=1.2pt, mark=*, mark options={solid, fill=black, black}, forget plot]
  table[row sep=crcr]{%
0.40625	3.78074148136388\\
};
\addplot [color=black, draw=none, mark size=0.4pt, mark=*, mark options={solid, fill=black, black}, forget plot]
  table[row sep=crcr]{%
2.90625	0.199027621453051\\
};
\addplot [color=black, draw=none, mark size=0.6pt, mark=*, mark options={solid, fill=black, black}, forget plot]
  table[row sep=crcr]{%
0.768515454431676	0.815390454431676\\
};
\addplot [color=black, draw=none, mark size=0.6pt, mark=*, mark options={solid, fill=black, black}, forget plot]
  table[row sep=crcr]{%
1.18399301936074	0.914861038721482\\
};
\addplot [color=black, draw=none, mark size=0.6pt, mark=*, mark options={solid, fill=black, black}, forget plot]
  table[row sep=crcr]{%
3	0\\
};
\addplot [color=black, draw=none, mark size=0.2pt, mark=*, mark options={solid, fill=black, black}, forget plot]
  table[row sep=crcr]{%
3	0.5\\
};
\addplot [color=black, draw=none, mark size=0.6pt, mark=*, mark options={solid, fill=black, black}, forget plot]
  table[row sep=crcr]{%
3	3.5\\
};
\addplot [color=black, draw=none, mark size=0.4pt, mark=*, mark options={solid, fill=black, black}, forget plot]
  table[row sep=crcr]{%
3	1.25\\
};
\addplot [color=black, draw=none, mark size=0.1pt, mark=*, mark options={solid, fill=black, black}, forget plot]
  table[row sep=crcr]{%
1.375	3\\
};
\addplot [color=black, draw=none, mark size=0.8pt, mark=*, mark options={solid, fill=black, black}, forget plot]
  table[row sep=crcr]{%
3	2.03125\\
};
\addplot [color=black, draw=none, mark size=0.1pt, mark=*, mark options={solid, fill=black, black}, forget plot]
  table[row sep=crcr]{%
3	3.78125\\
};
\addplot [color=black, draw=none, mark size=0.4pt, mark=*, mark options={solid, fill=black, black}, forget plot]
  table[row sep=crcr]{%
3	2.21875\\
};
\addplot [color=black, draw=none, mark size=0.3pt, mark=*, mark options={solid, fill=black, black}, forget plot]
  table[row sep=crcr]{%
3	3.28125\\
};
\addplot [color=black, draw=none, mark size=0.1pt, mark=*, mark options={solid, fill=black, black}, forget plot]
  table[row sep=crcr]{%
0.5	2.71875\\
};
\addplot [color=black, draw=none, mark size=0.4pt, mark=*, mark options={solid, fill=black, black}, forget plot]
  table[row sep=crcr]{%
0.34375	2.03125\\
};
\addplot [color=black, draw=none, mark size=0.8pt, mark=*, mark options={solid, fill=black, black}, forget plot]
  table[row sep=crcr]{%
0.28125	3.5625\\
};
\addplot [color=black, draw=none, mark size=0.7pt, mark=*, mark options={solid, fill=black, black}, forget plot]
  table[row sep=crcr]{%
3	0.984375\\
};
\addplot [color=black, draw=none, mark size=0.6pt, mark=*, mark options={solid, fill=black, black}, forget plot]
  table[row sep=crcr]{%
3	1.484375\\
};
\addplot [color=black, draw=none, mark size=0.4pt, mark=*, mark options={solid, fill=black, black}, forget plot]
  table[row sep=crcr]{%
2.078125	2.5\\
};
\addplot [color=black, draw=none, mark size=0.8pt, mark=*, mark options={solid, fill=black, black}, forget plot]
  table[row sep=crcr]{%
2.421875	2.0625\\
};
\addplot [color=black, draw=none, mark size=0.6pt, mark=*, mark options={solid, fill=black, black}, forget plot]
  table[row sep=crcr]{%
2.828125	0.8125\\
};
\addplot [color=black, draw=none, mark size=0.4pt, mark=*, mark options={solid, fill=black, black}, forget plot]
  table[row sep=crcr]{%
0.296875	3.328125\\
};
\addplot [color=black, draw=none, mark size=0.7pt, mark=*, mark options={solid, fill=black, black}, forget plot]
  table[row sep=crcr]{%
1.953125	0.15625\\
};
\addplot [color=black, draw=none, mark size=0.4pt, mark=*, mark options={solid, fill=black, black}, forget plot]
  table[row sep=crcr]{%
0.640625	1.453125\\
};
\addplot [color=black, draw=none, mark size=0.5pt, mark=*, mark options={solid, fill=black, black}, forget plot]
  table[row sep=crcr]{%
1.296875	1.78125\\
};
\end{axis}
\end{tikzpicture}%
}
\caption{\(\alpha = 10^{-3}\)}
\label{fig:uopt_1e-3}
\end{subfigure}
\caption{Visualization of the numerical reconstructions; each dot is one support point and dot
  area is proportional to source magnitude.}
\label{fig:uopt}
\end{figure}

In the following, we consider Algorithm~\ref{alg:PDAP} (denoted by PDAP), and
different versions of the accelerated conditional gradient method~\ref{alg:SPINAT} without
exact resolution of the subproblems. The unaccelerated version is denoted by GCG, and the
version performing one iterative tresholding step for the subproblem in each iteration is
denoted by SPINAT (cf.~\cite{BrediesPikkarainen:2013}). An suffix +PP denotes an
additional application of the sparsifying post-processing step from
Corollary~\ref{cor:remove_diracs}.
The numerical results are given in Figure~\ref{fig:res}, where we plot the evolution of
the residual over the computation time (in seconds).
We opt for computation times over the step counter \(k\)
to account for the fact that one step of an accelerated method may be more costly. We note
that all algorithms are implemented in MATLAB (version R2017a) and the computations are
performed on a compute node with a Intel\textregistered{} Xeon\textregistered{} CPU E5-2670
with eight cores at 2.60GHz.
\begin{figure}[htb]
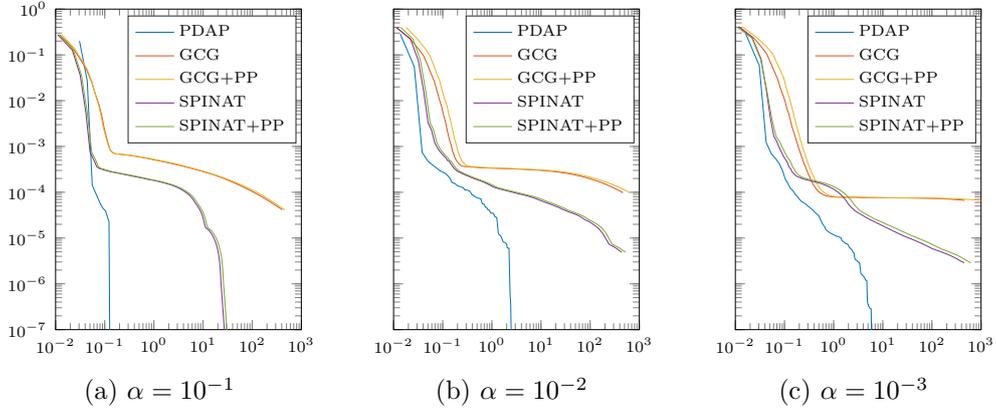

\begin{subfigure}[t]{.34\textwidth}
\centering
\scalebox{.85}{
\input{figures/compare_algs_res_tics_1e-1}
}
\caption{\(\alpha = 10^{-1}\)}
\label{fig:res_1e-1}
\end{subfigure}
\begin{subfigure}[t]{.32\textwidth}
\centering
\scalebox{.85}{
\input{figures/compare_algs_res_tics_1e-2}
}
\caption{\(\alpha = 10^{-2}\)}
\label{fig:res_1e-2}
\end{subfigure}
\begin{subfigure}[t]{.32\textwidth}
\centering
\scalebox{.85}{
\input{figures/compare_algs_res_tics_1e-3}
}
\caption{\(\alpha = 10^{-3}\)}
\label{fig:res_1e-3}
\end{subfigure}
\caption{Residuals \(j(u^k) - j(\widehat{u}_\alpha)\) over computation time in s.\ for
  different \(\alpha\).}
\label{fig:res}
\end{figure}
We observe that PDAP outperforms the other versions in almost all
situations. With the exception of \(\alpha = 10^{-1}\) it is the only implementation that
is able to solve the problem up the tolerance within the computational budget of \(50000\)
iterations (in fact it performs \(10\), \(96\), and \(129\) iterations, respectively).
We also see that SPINAT improves upon GCG, but not by as much as PDAP.

Additionally, we also give the current support size in Figure~\ref{fig:supp}.
\begin{figure}[htb]
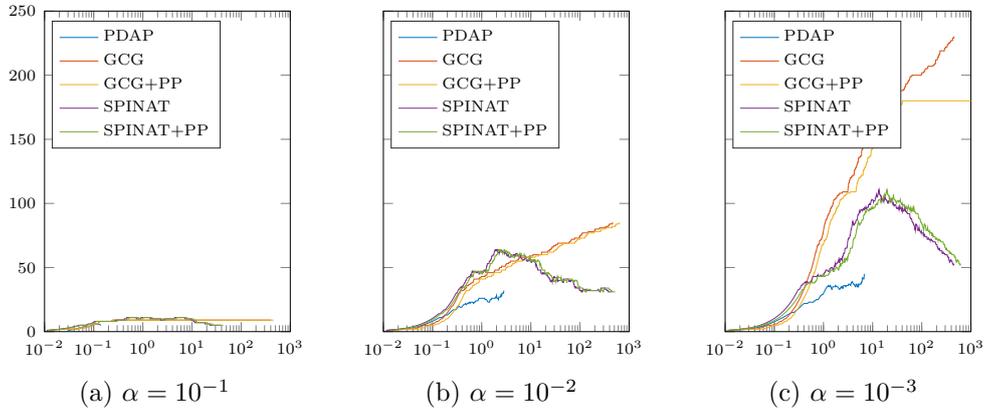

\centering
\begin{subfigure}[t]{.33\textwidth}
\centering
\scalebox{.85}{
\input{figures/compare_algs_supp_tics_1e-1}
}
\caption{\(\alpha = 10^{-1}\)}
\label{fig:supp_1e-1}
\end{subfigure}
\begin{subfigure}[t]{.32\textwidth}
\centering
\scalebox{.85}{
\input{figures/compare_algs_supp_tics_1e-2}
}
\caption{\(\alpha = 10^{-2}\)}
\label{fig:supp_1e-2}
\end{subfigure}
\begin{subfigure}[t]{.32\textwidth}
\centering
\scalebox{.85}{
\input{figures/compare_algs_supp_tics_1e-3}
}
\caption{\(\alpha = 10^{-3}\)}
\label{fig:supp_1e-3}
\end{subfigure}
\caption{Support size \(\# \supp u^k\) over computation time in s.\ for
  different \(\alpha\).}
\label{fig:supp}
\end{figure}
In the case of \(\alpha = 10^{-1}\), which is under-fitting the data, all algorithms
quickly identify a set of grid points which contain the support of the discrete numerical
solution and thus effectively stop to insert new points. However, note that this is only
the possible due to the finite grid, which limits the number of support point a~priori.
Note also that PDAP terminates once all support
points have been identified; cf.\ Remark~\ref{rem:stopping_criterion}.
In the other cases, the size of the support of the iterates is
negatively impacted by the spurious point sources introduced from over-fitting the
data. We note that for PDAP the support size of the iterates stays bounded by the
numerical support of the optimal solution (see Figure~\ref{fig:supp}), which keeps the
cost of resolution of the subproblems small. The theoretical upper bound on the support
size is \(2NM = 180\), which is very pessimistic for this example, and only provides an
advantage for GCG in the third setting.

Finally, we comment on the computation of the L-curve: Due to the fact that the
solution for a big \(\alpha\) can be used as an initial guess for a smaller \(\alpha\), the
computation of the L-curve up to \(\alpha_{12} = 10^{-3}\) with PDAP up to machine
precision is not much more expensive than computing just the solution for the last \(\alpha\)
starting from zero. For instance, in this case the number of iterations for each \(\alpha\) are
\((1,3,3,6,2,3,7,20,27,40,34,33,49)\), which results in a combined \(\sim 24\) seconds of
computation time versus \(129\) iterations in \(\sim 7\) seconds for just the last
value.

\subsubsection{Mesh independence}

Additionally, we investigate the behavior of the algorithms with respect to the mesh
width. Here, we only focus on PDAP, since we want to investigate if the improved
convergence observed before depends on the finite discretization.
Here, we compare iteration numbers, since the computation times are dominated by the
assembly of the gradients \(p^k_h\), which scales linearly in \(N_h\).
We give the results for the previous example on mesh levels \(l=7,8,9\) in
Figure~\ref{fig:res_mesh_ind}. We observe that although the number of iterations to reach
machine precision increases on finer meshes, the functional residual follows a similar
trajectory in the initial iterations. In the later iterations, the finite termination of
the method is reached earlier on coarse grids.
\begin{figure}[htb]
\begin{subfigure}[t]{.33\textwidth}
\centering
\scalebox{.85}{
%
%
\definecolor{mycolor1}{rgb}{0.00000,0.44700,0.74100}%
\definecolor{mycolor2}{rgb}{0.85000,0.32500,0.09800}%
\definecolor{mycolor3}{rgb}{0.92900,0.69400,0.12500}%
\definecolor{mycolor4}{rgb}{0.49400,0.18400,0.55600}%
\begin{tikzpicture}

\begin{axis}[%
width=3.804cm,
height=5cm,
at={(0cm,0cm)},
scale only axis,
xmin=0,
xmax=30,
ymode=log,
ymin=1e-08,
ymax=1,
yminorticks=true,
axis background/.style={fill=white},
legend style={legend cell align=left, align=left, draw=white!15!black}
]
\addplot [color=mycolor1]
  table[row sep=crcr]{%
1	0.57584150233727\\
2	0.201195551793215\\
3	0.0285384699786963\\
4	0.00167728801968131\\
5	0.00151972279385207\\
6	0.00118841407479989\\
7	0.00110338016525566\\
8	0.000977322738723779\\
9	0.00073905860147111\\
10	0.00068799552763521\\
11	0.000646452856884278\\
12	0.000594380797104999\\
13	0.000449122310647559\\
14	0.000331502726795607\\
15	0.000282843732202692\\
16	0.000233387435274263\\
17	0.000180246994292727\\
18	8.34164565014228e-05\\
19	4.84453717697342e-08\\
20	1.00000008274037e-10\\
};
\addlegendentry{\tiny $l = 9$}

\addplot [color=mycolor2]
  table[row sep=crcr]{%
1	0.575890382599302\\
2	0.201019200934506\\
3	0.0271013106019785\\
4	0.000141916818615445\\
5	8.87952786227386e-05\\
6	5.71211998832977e-05\\
7	4.46583497456798e-05\\
8	3.92402578966722e-05\\
9	2.24577768378476e-05\\
10	1.00000008274037e-10\\
};
\addlegendentry{\tiny $l = 8$}

\addplot [color=mycolor3]
  table[row sep=crcr]{%
1	0.575464098084636\\
2	0.200479510884596\\
3	0.02759598748372\\
4	0.0017961558768505\\
5	0.00120488270422187\\
6	0.00086610648593044\\
7	0.000692608672065137\\
8	0.000664450739018718\\
9	0.000427180869716021\\
10	0.000152840507874535\\
11	1.99591400720789e-06\\
12	1.00000008274037e-10\\
};
\addlegendentry{\tiny $l = 7$}


\end{axis}
\end{tikzpicture}%
}
\caption{\(\alpha = 10^{-1}\)}
\label{fig:res_mesh_ind_1e-1}
\end{subfigure}
\begin{subfigure}[t]{.32\textwidth}
\centering
\scalebox{.85}{
%
%
\definecolor{mycolor1}{rgb}{0.00000,0.44700,0.74100}%
\definecolor{mycolor2}{rgb}{0.85000,0.32500,0.09800}%
\definecolor{mycolor3}{rgb}{0.92900,0.69400,0.12500}%
\definecolor{mycolor4}{rgb}{0.49400,0.18400,0.55600}%
\begin{tikzpicture}

\begin{axis}[%
width=3.804cm,
height=5cm,
at={(0cm,0cm)},
scale only axis,
xmin=0,
xmax=120,
ymode=log,
ymin=1e-08,
ymax=1,
yminorticks=true,
yticklabels={,,},
axis background/.style={fill=white},
legend style={legend cell align=left, align=left, draw=white!15!black}
]
\addplot [color=mycolor1]
  table[row sep=crcr]{%
1	0.743318508264571\\
2	0.286717073967357\\
3	0.0566929612994478\\
4	0.0028056312045531\\
5	0.00162861588391916\\
6	0.00110553499016327\\
7	0.000862436784220058\\
8	0.0007983179562609\\
9	0.000742621949995786\\
10	0.000717304112977889\\
11	0.000693554777959359\\
12	0.000623992190508814\\
13	0.000591851702581088\\
14	0.000580905553455937\\
15	0.000454808683515776\\
16	0.000438575487936018\\
17	0.000433473166412804\\
18	0.000427967046529049\\
19	0.000423206790993651\\
20	0.000370976359132687\\
21	0.000356252213722999\\
22	0.000336824405381284\\
23	0.000333100124287847\\
24	0.000330402361404806\\
25	0.000317292312016024\\
26	0.000312916236738962\\
27	0.000311628284083619\\
28	0.000308879252250725\\
29	0.000305063012658553\\
30	0.000262811084243652\\
31	0.000261422037074802\\
32	0.000245238432281022\\
33	0.000244222960507086\\
34	0.000243594427112023\\
35	0.000181677230641341\\
36	0.000177871608875963\\
37	0.000177179762870096\\
38	0.000175923887987418\\
39	0.00017497952448766\\
40	0.0001661343228881\\
41	0.000164871965627537\\
42	0.000164346870927208\\
43	0.000163945560881944\\
44	0.000162705414439163\\
45	0.000162528777956542\\
46	0.000158042570359236\\
47	0.000157812900807089\\
48	0.000144646420673605\\
49	0.000144344115051268\\
50	0.00014373760713967\\
51	0.000142491336683755\\
52	0.000142276130398987\\
53	0.000141539896675522\\
54	0.000114594390632519\\
55	0.000114339773469942\\
56	0.000114133267664466\\
57	8.69509456695311e-05\\
58	8.59943729015419e-05\\
59	8.59485676557442e-05\\
60	5.11834428958899e-05\\
61	5.05062864517118e-05\\
62	5.04915093781991e-05\\
63	5.0284127206663e-05\\
64	4.97080547002399e-05\\
65	4.80332543329708e-05\\
66	4.79633227166328e-05\\
67	4.79386297156763e-05\\
68	4.78300105841677e-05\\
69	4.24770142294535e-05\\
70	4.24577652705094e-05\\
71	3.97676891063375e-05\\
72	3.96455751363115e-05\\
73	3.49742925300697e-05\\
74	3.4770832670214e-05\\
75	3.47304907423428e-05\\
76	3.45744512312111e-05\\
77	3.18124000644807e-05\\
78	3.12215984082029e-05\\
79	3.12027133300107e-05\\
80	2.6215403498786e-05\\
81	2.60542300847509e-05\\
82	2.58035541874264e-05\\
83	2.5610841883638e-05\\
84	2.54954320839836e-05\\
85	2.54732586657154e-05\\
86	2.54537284732452e-05\\
87	2.54189268804744e-05\\
88	1.73475342355595e-05\\
89	1.69521200385735e-05\\
90	1.6898920303942e-05\\
91	1.68620869076333e-05\\
92	1.65070097317421e-05\\
93	1.58647526788021e-05\\
94	1.58234852623389e-05\\
95	1.58211911931438e-05\\
96	7.50373730835224e-06\\
97	7.49525973315138e-06\\
98	6.30678130759954e-06\\
99	6.30328720596998e-06\\
100	6.02049687374248e-06\\
101	4.84322397085665e-06\\
102	4.76769195536728e-06\\
103	4.74761637308044e-06\\
104	1.57834052905495e-06\\
105	1.52944764878274e-06\\
106	1.4689251453294e-06\\
107	1.46224461522113e-06\\
108	1.4564517273781e-06\\
109	1.4242139604291e-06\\
110	1.37819439118697e-06\\
111	2.48531543019237e-08\\
112	1.23215725227055e-08\\
113	9.32627119282214e-09\\
114	5.88870244233597e-09\\
115	5.37223678737098e-10\\
116	2.26854469237425e-10\\
117	1.81217939959621e-10\\
118	1.23579521615902e-10\\
119	1.9725051453312e-11\\
120	1.e-20\\
};
\addlegendentry{\tiny $l = 9$}

\addplot [color=mycolor2]
  table[row sep=crcr]{%
1	0.743307753369925\\
2	0.286457673708496\\
3	0.0549925849712904\\
4	0.000743417186054932\\
5	0.000474574933497712\\
6	0.000394708639403094\\
7	0.000340994400940695\\
8	0.000299495577647208\\
9	0.000278407344270055\\
10	0.000261183451182957\\
11	0.000235576634728639\\
12	0.000206873114675557\\
13	0.000163896012422702\\
14	0.000158920172058367\\
15	0.000150987954464821\\
16	0.000138054957713695\\
17	0.000134669230525838\\
18	0.000132683541625014\\
19	0.000128020217720133\\
20	0.000126008547851844\\
21	0.000113685077705195\\
22	0.00011288095411097\\
23	0.000111350283086981\\
24	0.000109551556184855\\
25	0.000108098015673797\\
26	0.000101668424970176\\
27	8.83880526334306e-05\\
28	8.71639572193224e-05\\
29	8.56650181555313e-05\\
30	7.98544229557267e-05\\
31	7.833195689326e-05\\
32	7.6634960069339e-05\\
33	7.58084885722279e-05\\
34	7.40944566262336e-05\\
35	7.37769945135698e-05\\
36	7.21436976838692e-05\\
37	7.17719529082228e-05\\
38	5.62522594969028e-05\\
39	5.59267948609318e-05\\
40	5.48522386352897e-05\\
41	5.4656187738155e-05\\
42	4.73403990953088e-05\\
43	4.69357794298242e-05\\
44	4.63480225904277e-05\\
45	4.13231906278046e-05\\
46	4.10993483298226e-05\\
47	4.09399338727884e-05\\
48	3.55059334310355e-05\\
49	3.52585985847788e-05\\
50	3.50514693285633e-05\\
51	3.46828016378822e-05\\
52	3.44627114720296e-05\\
53	2.91918030279129e-05\\
54	2.89203289319979e-05\\
55	2.84551738202708e-05\\
56	2.82036601073586e-05\\
57	1.34740631362075e-05\\
58	1.25687990122254e-05\\
59	1.20754941612296e-05\\
60	1.18439522727755e-05\\
61	1.16068414871932e-05\\
62	1.15873873623717e-05\\
63	1.12273251814954e-05\\
64	1.095930842758e-05\\
65	1.08116667738196e-05\\
66	1.01157651738702e-05\\
67	9.60879571601861e-06\\
68	9.39417711470422e-06\\
69	8.34796244873412e-06\\
70	8.12005067345883e-06\\
71	8.06568052371184e-06\\
72	8.04128931801082e-06\\
73	7.93510248003546e-06\\
74	7.83719635154626e-06\\
75	7.78142291094769e-06\\
76	7.75100403410667e-06\\
77	7.36086536860456e-06\\
78	7.31155487462137e-06\\
79	6.15380458563602e-06\\
80	6.11292558777912e-06\\
81	6.02782662870519e-06\\
82	6.01100667908805e-06\\
83	6.00480395750863e-06\\
84	5.99035727363861e-06\\
85	5.60878908692497e-07\\
86	3.51458294103185e-07\\
87	1.65495419722494e-07\\
88	1.02042745146053e-07\\
89	6.19482300005803e-08\\
90	4.55680589546426e-08\\
91	4.45397781365098e-08\\
92	4.24339333879309e-08\\
93	7.39495661927769e-09\\
94	3.84332018640854e-09\\
95	1.04214680982095e-09\\
96	1.e-11\\
};
\addlegendentry{\tiny $l = 8$}

\addplot [color=mycolor3]
  table[row sep=crcr]{%
1	0.74306978775346\\
2	0.286094515974436\\
3	0.0557340698243721\\
4	0.00391285849892876\\
5	0.00230577003836231\\
6	0.00169515969119435\\
7	0.00140397602233393\\
8	0.00118324409103236\\
9	0.00105342633551987\\
10	0.00094806796324038\\
11	0.000911127409117633\\
12	0.000784436273026729\\
13	0.000641095188587897\\
14	0.000612370121496779\\
15	0.000574254542434524\\
16	0.000536949227210079\\
17	0.00050850239793607\\
18	0.000489628560953161\\
19	0.000471981203819759\\
20	0.000446556150685375\\
21	0.000435515327898733\\
22	0.000428870942415602\\
23	0.00042277378975622\\
24	0.000306028070816526\\
25	0.00029923446061678\\
26	0.000295569670867128\\
27	0.0002922996584106\\
28	0.000288231085537799\\
29	0.000285382270113699\\
30	0.000228164899316177\\
31	0.00022296532720667\\
32	0.000218491013679517\\
33	0.000217325424147308\\
34	0.000215255589968939\\
35	0.000182285467973692\\
36	0.000176593962781038\\
37	0.000170952282591202\\
38	0.000139381027014482\\
39	0.000137872474113392\\
40	0.000136167280440729\\
41	0.000127656984031565\\
42	0.00012663390547862\\
43	0.000124180870348795\\
44	0.000121569339015019\\
45	0.000121054320300181\\
46	9.80936807102144e-05\\
47	9.78537281355642e-05\\
48	9.66086636789189e-05\\
49	9.62213391801012e-05\\
50	9.6158393092946e-05\\
51	8.60990813403677e-05\\
52	7.88541358172634e-05\\
53	7.85929929796383e-05\\
54	7.84148156862362e-05\\
55	5.32536531177158e-05\\
56	5.30560408104584e-05\\
57	5.27878408478552e-05\\
58	5.264804958088e-05\\
59	3.49710767402339e-05\\
60	3.47165203599051e-05\\
61	3.40785280914663e-05\\
62	3.40053335266335e-05\\
63	3.34186114947684e-05\\
64	3.32694219860527e-05\\
65	2.92454971374222e-05\\
66	2.91808221933414e-05\\
67	2.89386372833433e-05\\
68	2.85714890952728e-05\\
69	2.98836008540831e-06\\
70	2.67118145003131e-06\\
71	8.34048093114464e-07\\
72	8.11585307531476e-07\\
73	3.14035178264338e-07\\
74	2.98252906930779e-07\\
75	1.02911069010253e-07\\
76	9.87163946419845e-09\\
77	2.38715892539321e-09\\
78	1e-11\\
};
\addlegendentry{\tiny $l = 7$}

\end{axis}
\end{tikzpicture}%
}
\caption{\(\alpha = 10^{-2}\)}
\label{fig:res_mesh_ind_1e-2}
\end{subfigure}
\begin{subfigure}[t]{.32\textwidth}
\centering
\scalebox{.85}{
\input{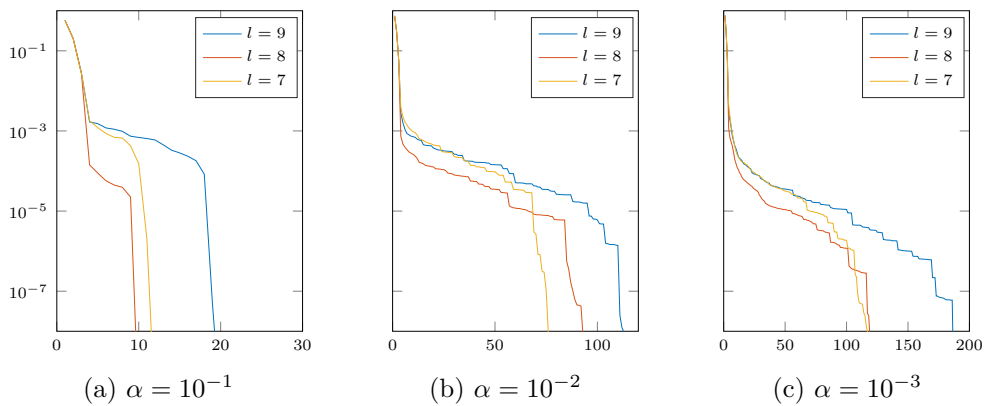}
}
\caption{\(\alpha = 10^{-3}\)}
\label{fig:res_mesh_ind_1e-3}
\end{subfigure}
\caption{Function residuals \(j(u^k) - j(\widehat{u}_\alpha)\) over iterations for
  different mesh levels \(l\).}
\label{fig:res_mesh_ind}
\end{figure}
Concerning the maximal support of the numerical solution throughout the iterations, we
observe that it seems to be dependent on \(\alpha\), but bounded by a similar constant
independent of the grid level.
\begin{figure}[htb]
\begin{subfigure}[t]{.33\textwidth}
\centering
\scalebox{.85}{
%
%
\definecolor{mycolor1}{rgb}{0.00000,0.44700,0.74100}%
\definecolor{mycolor2}{rgb}{0.85000,0.32500,0.09800}%
\definecolor{mycolor3}{rgb}{0.92900,0.69400,0.12500}%
\definecolor{mycolor4}{rgb}{0.49400,0.18400,0.55600}%
\begin{tikzpicture}

\begin{axis}[%
width=3.804cm,
height=5cm,
at={(0cm,0cm)},
scale only axis,
xmin=0,
xmax=30,
ymin=0,
ymax=50,
axis background/.style={fill=white},
legend style={at={(0.03,0.97)}, anchor=north west, legend cell align=left, align=left, draw=white!15!black}
]
\addplot [color=mycolor1]
  table[row sep=crcr]{%
1	0\\
2	1\\
3	2\\
4	3\\
5	4\\
6	5\\
7	6\\
8	6\\
9	6\\
10	7\\
11	8\\
12	8\\
13	7\\
14	7\\
15	7\\
16	7\\
17	7\\
18	5\\
19	3\\
20	4\\
};
\addlegendentry{\tiny $l = 9$}

\addplot [color=mycolor2]
  table[row sep=crcr]{%
1	0\\
2	1\\
3	2\\
4	3\\
5	4\\
6	4\\
7	5\\
8	6\\
9	6\\
10	5\\
};
\addlegendentry{\tiny $l = 8$}

\addplot [color=mycolor3]
  table[row sep=crcr]{%
1	0\\
2	1\\
3	2\\
4	3\\
5	4\\
6	5\\
7	6\\
8	7\\
9	7\\
10	7\\
11	7\\
12	7\\
};
\addlegendentry{\tiny $l = 7$}


\end{axis}
\end{tikzpicture}%
}
\caption{\(\alpha = 10^{-1}\)}
\label{fig:supp_mesh_ind_1e-1}
\end{subfigure}
\begin{subfigure}[t]{.32\textwidth}
\centering
\scalebox{.85}{
%
%
\definecolor{mycolor1}{rgb}{0.00000,0.44700,0.74100}%
\definecolor{mycolor2}{rgb}{0.85000,0.32500,0.09800}%
\definecolor{mycolor3}{rgb}{0.92900,0.69400,0.12500}%
\definecolor{mycolor4}{rgb}{0.49400,0.18400,0.55600}%
\begin{tikzpicture}

\begin{axis}[%
width=3.804cm,
height=5cm,
at={(0cm,0cm)},
scale only axis,
xmin=0,
xmax=120,
ymin=0,
ymax=50,
yticklabels={,,},
axis background/.style={fill=white},
legend style={at={(0.03,0.97)}, anchor=north west, legend cell align=left, align=left, draw=white!15!black}
]
\addplot [color=mycolor1]
  table[row sep=crcr]{%
1	0\\
2	1\\
3	2\\
4	3\\
5	4\\
6	5\\
7	6\\
8	7\\
9	8\\
10	9\\
11	10\\
12	11\\
13	12\\
14	12\\
15	12\\
16	12\\
17	13\\
18	14\\
19	15\\
20	15\\
21	15\\
22	16\\
23	17\\
24	18\\
25	17\\
26	18\\
27	19\\
28	18\\
29	19\\
30	19\\
31	20\\
32	21\\
33	22\\
34	23\\
35	21\\
36	20\\
37	20\\
38	21\\
39	22\\
40	22\\
41	22\\
42	23\\
43	24\\
44	24\\
45	25\\
46	25\\
47	26\\
48	26\\
49	27\\
50	27\\
51	27\\
52	27\\
53	27\\
54	22\\
55	23\\
56	24\\
57	23\\
58	22\\
59	22\\
60	20\\
61	21\\
62	22\\
63	22\\
64	22\\
65	23\\
66	24\\
67	24\\
68	24\\
69	24\\
70	25\\
71	25\\
72	25\\
73	23\\
74	23\\
75	23\\
76	23\\
77	23\\
78	23\\
79	24\\
80	24\\
81	24\\
82	24\\
83	23\\
84	23\\
85	24\\
86	24\\
87	24\\
88	22\\
89	22\\
90	22\\
91	23\\
92	23\\
93	21\\
94	21\\
95	22\\
96	21\\
97	22\\
98	21\\
99	22\\
100	22\\
101	22\\
102	22\\
103	22\\
104	22\\
105	22\\
106	23\\
107	23\\
108	24\\
109	23\\
110	22\\
111	22\\
112	22\\
113	23\\
114	24\\
115	24\\
116	24\\
117	25\\
118	26\\
119	27\\
120	28\\
};
\addlegendentry{\tiny $l = 9$}

\addplot [color=mycolor2]
  table[row sep=crcr]{%
1	0\\
2	1\\
3	2\\
4	3\\
5	4\\
6	5\\
7	6\\
8	7\\
9	8\\
10	9\\
11	10\\
12	11\\
13	11\\
14	12\\
15	13\\
16	13\\
17	14\\
18	15\\
19	16\\
20	17\\
21	17\\
22	17\\
23	18\\
24	19\\
25	20\\
26	19\\
27	19\\
28	20\\
29	21\\
30	21\\
31	22\\
32	21\\
33	22\\
34	22\\
35	23\\
36	22\\
37	23\\
38	23\\
39	23\\
40	22\\
41	23\\
42	23\\
43	24\\
44	24\\
45	24\\
46	25\\
47	26\\
48	26\\
49	26\\
50	26\\
51	26\\
52	26\\
53	26\\
54	25\\
55	25\\
56	25\\
57	23\\
58	23\\
59	23\\
60	23\\
61	23\\
62	24\\
63	25\\
64	25\\
65	25\\
66	25\\
67	26\\
68	26\\
69	26\\
70	26\\
71	26\\
72	26\\
73	26\\
74	26\\
75	27\\
76	27\\
77	27\\
78	28\\
79	26\\
80	26\\
81	26\\
82	25\\
83	26\\
84	27\\
85	27\\
86	27\\
87	28\\
88	29\\
89	29\\
90	28\\
91	29\\
92	30\\
93	30\\
94	30\\
95	31\\
96	32\\
};
\addlegendentry{\tiny $l = 8$}

\addplot [color=mycolor3]
  table[row sep=crcr]{%
1	0\\
2	1\\
3	2\\
4	3\\
5	4\\
6	5\\
7	6\\
8	7\\
9	8\\
10	9\\
11	10\\
12	11\\
13	12\\
14	12\\
15	12\\
16	13\\
17	14\\
18	14\\
19	15\\
20	16\\
21	17\\
22	17\\
23	18\\
24	17\\
25	16\\
26	17\\
27	18\\
28	19\\
29	20\\
30	19\\
31	19\\
32	19\\
33	19\\
34	20\\
35	20\\
36	20\\
37	21\\
38	21\\
39	22\\
40	21\\
41	22\\
42	22\\
43	22\\
44	21\\
45	22\\
46	22\\
47	23\\
48	23\\
49	24\\
50	25\\
51	24\\
52	25\\
53	26\\
54	27\\
55	25\\
56	26\\
57	26\\
58	26\\
59	26\\
60	26\\
61	26\\
62	27\\
63	27\\
64	28\\
65	28\\
66	28\\
67	28\\
68	28\\
69	26\\
70	27\\
71	25\\
72	26\\
73	26\\
74	26\\
75	26\\
76	27\\
77	28\\
78	29\\
};
\addlegendentry{\tiny $l = 7$}


\end{axis}
\end{tikzpicture}%
}
\caption{\(\alpha = 10^{-2}\)}
\label{fig:supp_mesh_ind_1e-2}
\end{subfigure}
\begin{subfigure}[t]{.32\textwidth}
\centering
\scalebox{.85}{
%
%
\definecolor{mycolor1}{rgb}{0.00000,0.44700,0.74100}%
\definecolor{mycolor2}{rgb}{0.85000,0.32500,0.09800}%
\definecolor{mycolor3}{rgb}{0.92900,0.69400,0.12500}%
\definecolor{mycolor4}{rgb}{0.49400,0.18400,0.55600}%
\begin{tikzpicture}

\begin{axis}[%
width=3.804cm,
height=5cm,
at={(0cm,0cm)},
scale only axis,
xmin=0,
xmax=200,
ymin=0,
ymax=50,
yticklabels={,,},
axis background/.style={fill=white},
legend style={at={(0.97,0.03)}, anchor=south east, legend cell align=left, align=left, draw=white!15!black}
]
\addplot [color=mycolor1]
  table[row sep=crcr]{%
1	0\\
2	1\\
3	2\\
4	3\\
5	4\\
6	5\\
7	6\\
8	7\\
9	8\\
10	9\\
11	10\\
12	11\\
13	12\\
14	13\\
15	14\\
16	15\\
17	16\\
18	17\\
19	17\\
20	17\\
21	18\\
22	17\\
23	18\\
24	18\\
25	19\\
26	20\\
27	20\\
28	20\\
29	21\\
30	21\\
31	21\\
32	22\\
33	23\\
34	23\\
35	24\\
36	25\\
37	25\\
38	24\\
39	24\\
40	24\\
41	25\\
42	26\\
43	26\\
44	26\\
45	27\\
46	26\\
47	27\\
48	28\\
49	29\\
50	30\\
51	31\\
52	32\\
53	33\\
54	33\\
55	33\\
56	34\\
57	34\\
58	34\\
59	35\\
60	36\\
61	36\\
62	37\\
63	36\\
64	36\\
65	34\\
66	34\\
67	34\\
68	34\\
69	35\\
70	36\\
71	37\\
72	36\\
73	37\\
74	37\\
75	38\\
76	38\\
77	39\\
78	40\\
79	40\\
80	39\\
81	40\\
82	41\\
83	40\\
84	40\\
85	40\\
86	40\\
87	39\\
88	38\\
89	36\\
90	37\\
91	37\\
92	38\\
93	37\\
94	38\\
95	37\\
96	38\\
97	39\\
98	40\\
99	40\\
100	40\\
101	39\\
102	40\\
103	40\\
104	40\\
105	39\\
106	39\\
107	40\\
108	41\\
109	40\\
110	41\\
111	41\\
112	40\\
113	39\\
114	40\\
115	41\\
116	42\\
117	42\\
118	42\\
119	42\\
120	42\\
121	42\\
122	43\\
123	43\\
124	43\\
125	43\\
126	42\\
127	41\\
128	39\\
129	39\\
130	39\\
131	40\\
132	39\\
133	39\\
134	39\\
135	40\\
136	40\\
137	40\\
138	41\\
139	41\\
140	42\\
141	42\\
142	40\\
143	40\\
144	41\\
145	41\\
146	41\\
147	41\\
148	41\\
149	41\\
150	41\\
151	42\\
152	43\\
153	42\\
154	42\\
155	42\\
156	42\\
157	42\\
158	42\\
159	42\\
160	42\\
161	43\\
162	42\\
163	42\\
164	43\\
165	44\\
166	44\\
167	43\\
168	43\\
169	43\\
170	40\\
171	41\\
172	42\\
173	40\\
174	40\\
175	40\\
176	41\\
177	41\\
178	42\\
179	41\\
180	41\\
181	42\\
182	43\\
183	44\\
184	45\\
185	44\\
186	45\\
187	42\\
188	43\\
189	44\\
190	45\\
191	45\\
192	45\\
193	46\\
194	47\\
195	47\\
196	48\\
197	49\\
198	50\\
};
\addlegendentry{\tiny $l = 9$}

\addplot [color=mycolor2]
  table[row sep=crcr]{%
1	0\\
2	1\\
3	2\\
4	3\\
5	4\\
6	5\\
7	6\\
8	7\\
9	8\\
10	9\\
11	10\\
12	11\\
13	12\\
14	13\\
15	14\\
16	15\\
17	15\\
18	16\\
19	17\\
20	18\\
21	19\\
22	20\\
23	21\\
24	22\\
25	22\\
26	22\\
27	22\\
28	23\\
29	23\\
30	23\\
31	23\\
32	23\\
33	24\\
34	24\\
35	25\\
36	25\\
37	25\\
38	26\\
39	27\\
40	28\\
41	28\\
42	29\\
43	29\\
44	30\\
45	31\\
46	32\\
47	33\\
48	34\\
49	35\\
50	35\\
51	36\\
52	35\\
53	35\\
54	36\\
55	35\\
56	33\\
57	34\\
58	35\\
59	35\\
60	35\\
61	35\\
62	36\\
63	35\\
64	35\\
65	33\\
66	33\\
67	33\\
68	34\\
69	34\\
70	35\\
71	34\\
72	35\\
73	34\\
74	35\\
75	36\\
76	36\\
77	35\\
78	35\\
79	36\\
80	36\\
81	36\\
82	36\\
83	36\\
84	37\\
85	37\\
86	37\\
87	36\\
88	37\\
89	38\\
90	38\\
91	39\\
92	38\\
93	37\\
94	37\\
95	35\\
96	35\\
97	35\\
98	36\\
99	37\\
100	37\\
101	37\\
102	35\\
103	35\\
104	35\\
105	35\\
106	36\\
107	37\\
108	37\\
109	36\\
110	37\\
111	37\\
112	37\\
113	37\\
114	37\\
115	37\\
116	38\\
117	38\\
118	39\\
119	39\\
120	40\\
121	41\\
122	42\\
123	42\\
124	41\\
125	41\\
126	42\\
127	43\\
128	44\\
129	45\\
};
\addlegendentry{\tiny $l = 8$}

\addplot [color=mycolor3]
  table[row sep=crcr]{%
1	0\\
2	1\\
3	2\\
4	3\\
5	4\\
6	5\\
7	6\\
8	7\\
9	8\\
10	9\\
11	10\\
12	11\\
13	12\\
14	13\\
15	14\\
16	15\\
17	16\\
18	17\\
19	18\\
20	19\\
21	20\\
22	20\\
23	20\\
24	19\\
25	20\\
26	20\\
27	21\\
28	21\\
29	21\\
30	22\\
31	23\\
32	23\\
33	23\\
34	23\\
35	23\\
36	23\\
37	23\\
38	24\\
39	25\\
40	26\\
41	27\\
42	28\\
43	26\\
44	27\\
45	28\\
46	29\\
47	30\\
48	31\\
49	32\\
50	32\\
51	33\\
52	32\\
53	32\\
54	31\\
55	31\\
56	32\\
57	33\\
58	34\\
59	34\\
60	32\\
61	33\\
62	34\\
63	34\\
64	34\\
65	32\\
66	31\\
67	31\\
68	30\\
69	31\\
70	32\\
71	32\\
72	32\\
73	33\\
74	33\\
75	34\\
76	35\\
77	34\\
78	35\\
79	35\\
80	36\\
81	37\\
82	34\\
83	34\\
84	33\\
85	32\\
86	33\\
87	34\\
88	35\\
89	36\\
90	35\\
91	35\\
92	36\\
93	35\\
94	36\\
95	37\\
96	38\\
97	39\\
98	37\\
99	37\\
100	37\\
101	37\\
102	38\\
103	38\\
104	37\\
105	38\\
106	38\\
107	36\\
108	37\\
109	37\\
110	37\\
111	38\\
112	39\\
113	40\\
114	40\\
115	40\\
116	41\\
117	41\\
118	42\\
119	43\\
120	43\\
};
\addlegendentry{\tiny $l = 7$}


\end{axis}
\end{tikzpicture}%
}
\caption{\(\alpha = 10^{-3}\)}
\label{fig:supp_mesh_ind_1e-3}
\end{subfigure}
\caption{Support size \(\# \supp u^k\) over iterations for
  different mesh levels \(l\).}
\label{fig:supp_mesh_ind}
\end{figure}
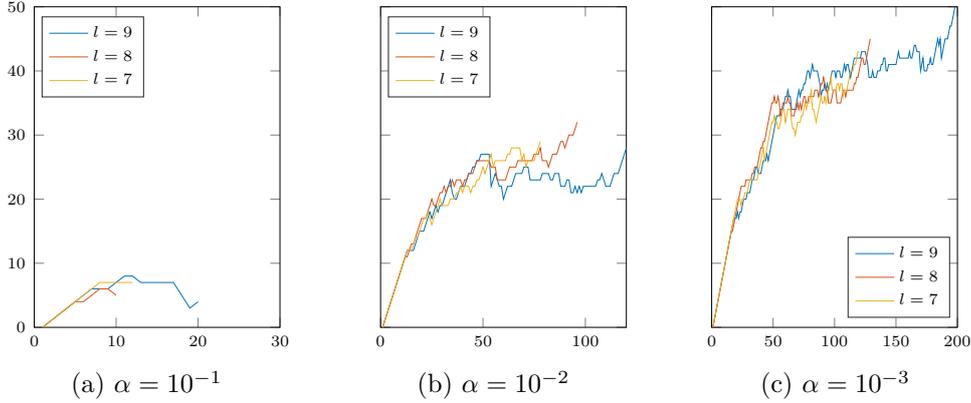

\appendix
\section{Sparse minimization with finite rank operators}
\label{app:minimization_finite_rank}

Let \(H_1\) be a separable real Hilbert space, and \(\Mgen = \Cgen^*\) be the associated space
of vector measures. Introduce the solution operator
\[
S\colon \Mgen \to H_2,
\]
where \(H_2\) is another separable real Hilbert space.
\(S\) is assumed to be linear, and weak-\(*\) to weak continuous (the weak-\(*\)
topology on the dual of the separable space \(\Cgen\) can be normed, therefore, this is
the same as the sequential equivalent). Moreover, \(S\) can be written as the Banach
space dual of a continuous operator
\[
S^*\colon H_2 \to \Cgen.
\]

In this section, we give some results for the two abstract minimization problems relevant
for this paper. Most of these results are slight generalizations of known results, which
we could not directly find in the literature. We consider the problem
\begin{align}
\label{eq:Palpha_gen}
\tag{\ensuremath{P_\alpha}}
\min_{u\in\Mgen}&\; \left[\frac{1}{2\alpha} \norm{S u - p_d}^2_{H_2} + \norm{u}_{\Mgen}\right],
\end{align}
for given \(p_d \in H_2\) and \(\alpha > 0\).
Note, that in contrast to~\eqref{eq:problem_convex}, we have multiplied the objective
function by \(1/\alpha\), which obviously does not change the solution set, but leads to a
more convenient form of the dual problem below.
Moreover, we consider the associated minimum norm problem
\begin{align}
\label{eq:Pzero_gen}
\tag{\ensuremath{P_0}}
\min_{u\in\Mgen}&\; \norm{u}_{\Mgen} \quad\text{subject to } S u = p_d,
\end{align}
for some \(p_d = S u^\star\), \(u^\star \in \Mgen\).
It is know that under the general assumptions on \(S\), both problems have solutions.
This can be verified with the direct method of the calculus of variations.
Moreover, the dual problem of~\eqref{eq:Palpha_gen},
\begin{align}
\label{eq:Dalpha_gen}
\tag{\ensuremath{D_\alpha}}
\max_{y \in H_2}&\; \left[(p_d, y)_{H_2} - \frac{\alpha}{2} \norm{y}^2_{H_2}\right]
\quad\text{subject to } \norm{S^* y}_{\Cgen} \leq 1,
\end{align}
has a unique solution, and the strong duality \(\max \eqref{eq:Dalpha_gen}
= \min \eqref{eq:Palpha_gen}\) holds; see~\cite[Proposition~3.5]{BrediesPikkarainen:2013}
(the proof is only given for \(H_1 = \R^n\), but works unmodified in the general case).
For~\eqref{eq:Pzero_gen}, the dual problem is given by
\begin{align}
\label{eq:Dzero_gen}
\tag{\ensuremath{D_0}}
\max_{y \in H_2}&\; (p_d, y)
\quad\text{subject to } \norm{S^* y}_{\Cgen} \leq 1.
\end{align}
Since \(p_d = S u^\star\), strong duality holds with \(\sup \eqref{eq:Dzero_gen}
= \min \eqref{eq:Pzero_gen}\); see~\cite[Proposition~13]{DuvalPeyre:2015} (the proof is
only given for \(H_1 = \R\) and \(D\) equal to the torus, but works unmodified in the
general setting).
\begin{proposition}[{\cite[Proposition~13]{DuvalPeyre:2015}}]
\label{prop:strong_duality}
Let \(p_d = S u^\star\), \(u^\star \in \Mgen\).
Then, strong duality (see, e.g., \cite[Chapter~3.4]{EkelandTemam:1999}) holds for the
problem~\eqref{eq:Pzero_gen} and the dual problem~\eqref{eq:Dzero_gen}.
If the dual problem admits a solution, any pair of solutions \((u^\dagger,y^\dagger)\) to
both problems is characterized by the subdifferential inclusion \(S^* y^\dagger
\in \partial\norm{u^\dagger}_{\Mgen}\).
\end{proposition}
In general, \eqref{eq:Dzero_gen} does not necessarily have a solution.
However, if \(S\) is a finite rank operator (the range of \(S\) or \(S^*\) is finite
dimensional), the dual problem~\eqref{eq:Dzero_gen} admits a solution. This result is mentioned
and used in~\cite{DuvalPeyre:2015}; however, since no proof is given, we provide one for
the general setting above.
\begin{proposition}
\label{prop:existence_dual_zero}
Suppose that \(\Ran S\) is finite dimensional and $p_d\in \Ran S$. Then, the dual problem~\eqref{eq:Dzero_gen}
admits a solution.
Suppose additionally that the adjoint \(S^*\colon H_2 \to \Cgen\) is injective. Then, the above result
holds for any \(p_d \in H_2\) and the solution set of~\eqref{eq:Dzero_gen} is bounded.
\end{proposition}
\begin{proof}
  We first assume that \(S^*\colon H_2 \to \Cgen\) is injective. Note that this implies \(H_2\) is finite dimensional. In this case, \eqref{eq:Dzero_gen} can be reformulated
  as a semi-infinite optimization problem, and the result can be deduced as an application
  of the general result~\cite[Theorem~5.99]{BonnansShapiro:2000} (injectivity of \(S^*\)
  is equivalent to the regularity condition mentioned there).
  However, in our case, it can be also shown directly.
  In fact, any maximizing sequence for \eqref{eq:Dzero_gen}  is bounded: Take by
  contradiction \(\{y_k\}\) with \(\norm{S^*y_k}_{\Cgen} \leq 1\) and
  \(\norm{y_k}_{H_2}\rightarrow \infty\). Considering the renormed sequence \(\{\tilde{y}_k\}_{k\in \N}\) with
  \(\tilde{y}_k = y_k / \norm{y_k}_{H_2}\) there exists a subsequence denoted by
  the same symbol and a \(\hat{y} \in H_2\) with \(\tilde{y}_k \to \hat{y}\) and
  \(\norm{\hat y}_{H_2}=1\) (since \(H_2\) is finite dimensional). Consequently there holds
  \begin{align*}
   \norm{y_k}_{H_2} \norm{S^*\tilde{y}_k}_{\Cgen}
   =\norm{S^*y_k}_{\Cgen} \leq 1.
  \end{align*}
  From this we directly conclude that \(\norm{S^*\hat{y}}_{\Cgen}=0\) since $S^\ast$ is bounded. Then the injectivity of \(S^*\) implies a contradiction to
  \(\norm{\hat{y}}_{H_2}=1\) . Consequently, any
  minimizing sequence is bounded, and by using the
  continuity of \(S^*\), it follows that there exits
  at least one optimal solution to~\eqref{eq:Dzero_gen}. Boundedness of the solution set
  follows in the same way.

  Now, we address the general case, where \(S^*\) is not necessarily injective, and
  show that it can be reduced to the previous case. Consider the problem
  \begin{equation}
    \label{eq:dual_modified}
    \max_{y \in \Ran S} \; (y,p_d)
    \quad\text{subject to } \norm{S^* y}_{\Cgen} \leq 1.
  \end{equation}
  Since \(\Ran S\) is finite dimensional (and therefore a closed subspace), we
  have \((\Ran S)^\bot = \Ker S^*\), and \(H_2 = \Ran S \oplus \Ker S^*\). For any \(y \in
  H_2\) we have \(y = y_1 +  y_0\) with \(y_1 \in \Ran S = (\Ker S^*)^\bot\) and \(y_0 \in
  \Ker S^*\). Let $u^\star\in \Mgen$ be an element with $Su^\star=p_d$ which exists according to our assumptions. Then we have
  \[
    (p_d,y) = \pair{u^\star, S^* y_1}=(p_d,y_1),
    \quad\text{and } \norm{S^* y}_{\Cgen} = \norm{S^* y_1}_{\Cgen}
  \]
  which implies that~\eqref{eq:dual_modified} and~\eqref{eq:Dzero_gen} have the same
  value. Moreover, the restricted operator \(S^*\rvert_{\Ran S} \colon \Ran S \to
  \Cgen\) is injective. Using the result from before,~\eqref{eq:dual_modified} admits a solution, and for any
  solution \(y_1\) and any \(y_0 \in \Ker S^*\), \(y = y_1 + y_0\) is a solution of~\eqref{eq:Dzero_gen}.
\QED
\end{proof}

\section{Extremal solutions}
\label{app:extremal_solutions}

Since the dual problems~\eqref{eq:Dalpha_gen} and~\eqref{eq:Dzero_gen} fall into the category of
\emph{semi-infinite} optimization problems, it follows that solutions of~\eqref{eq:Palpha_gen}
and~\eqref{eq:Pzero_gen} consisting of
finitely many Dirac delta functions exist; see, e.g.,
\cite[Section~5.4.2]{BonnansShapiro:2000}.

For the convenience of the reader, we provide a direct proof, which also leads to an
algorithmic strategy for reducing the support of any suboptimal point of~\eqref{eq:Palpha_gen}
or~\eqref{eq:Pzero_gen}.
To this purpose, we analyze the corresponding solution sets,
which we denote for \(\alpha \geq 0\) by
\[
U_{p_d,\alpha} = \set{u \in \Mgen \;|\; u \text{ solves~\eqref{eq:Palpha_gen} for \(\alpha>0\)
    or~\eqref{eq:Pzero_gen} for \(\alpha=0\)}}.
\]
This is a convex bounded subset of \(\Mgen\).
Furthermore the following properties are easily derived.
\begin{proposition}
Let \(\widehat{u} \in U_{p_d,\alpha}\) be arbitrary, and \(\widehat{p} = S\widehat{u}\).
For all elements \(u \in U_{p_d,\alpha}\) we have
\begin{align*}
\SO u = \widehat{p}, \qquad
\norm{u}_{\Mgen} = \norm{\widehat{u}}_{\Mgen}.
\end{align*}
\end{proposition}
\begin{proof}
The statement is clear for \(\alpha=0\), where \(\widehat{p} = p_d\).
For \(\alpha > 0\) the first part follows from the strict convexity of the tracking term and the
linearity of \(S\). Therefore, the value of the first term of the
objective assumes a unique value for all optimal solutions. By the optimality follows that
also the second term must be of the same value for all optimal solutions.
\QED
\end{proof}
As a corollary, we obtain a characterization of \(U_{p_d,\alpha}\).
\begin{corollary}
\label{cor:opt}
Let \(\widehat{u} \in U_{p_d,\alpha}\) be arbitrary, and \(\widehat{p} = S\widehat{u}\).
It holds,
\begin{align*}
U_{p_d,\alpha} &= \set{u \in \Mgen
  \;|\; \SO u = \widehat{p}
  \text{ and } \norm{u}_{\Mgen} = \norm{\widehat{u}}_{\Mgen}}
\\
&= \set{u \in \Mgen
  \;|\; \SO u = \widehat{p}
  \text{ and } \norm{u}_{\Mgen} \leq \norm{\widehat{u}}_{\Mgen}}.
\end{align*}
\end{corollary}
Now, we recall the concept of extremal points of convex set: A point in the convex set
\(U_{p_d,\alpha}\) is called \emph{extremal}, if it can not be written as a nontrivial
convex combination of other elements of \(U_{p_d,\alpha}\). Furthermore, we have the
theorem of Krein and Milman.
\begin{proposition}\label{prop:krein_milman}
The closure (in the sense of the weak-\(*\) topology)
of the convex combinations of the extremal points of \(U_{p_d,\alpha}\) is
equal to \(U_{p_d,\alpha}\), i.e.,
\[
U_{p_d,\alpha} = \widebar{\conv\set{u \in U_{p_d,\alpha} \;|\; u \text{ extremal }}}^{\text{weak}-*}
\]
\end{proposition}
\begin{proof}
Corollary \ref{cor:opt}, the Banach-Alaoglu Theorem and the weak-\(*\) continuity of $S$ imply that \(U_{p_d,\alpha}\) is
compact with respect to the weak-\(*\) topology. Then the assertion is a direct application of the theorem of Krein-Milman; see,
e.g.,~\cite[Theorem~2.19]{BonnansShapiro:2000}.
\QED
\end{proof}
Furthermore, if \(S\) is a finite rank operator, the extremal points can be characterized
as follows (cf., e.g., \cite[Proposition~2.177]{BonnansShapiro:2000}).
\begin{theorem}\label{thm:extremal_dirac}
Suppose that \(\dim \Ran S = N_S < \infty\).
The extremal points of \(U_{p_d,\alpha}\) can be written as a linear combinations of no
more than \(N_S\) Dirac delta functions:
\[
\set{u \in U_{p_d,\alpha} \;|\; u \text{ extremal }}
\subset \left\{\sum_{j=1}^{N_S} \coeff{u}_j \delta_{x_j} \;\Big|\; \coeff{u}_j \in H_1,\; x_j \in D\right\}.
\]
\end{theorem}
\begin{proof}
Let \(u \in U_{p_d,\alpha}\) be extremal. The proof will be done by contradiction. Assume,
therefore, that \(\supp u\) consists of more than \(N_S\) points. Then, there exists
a disjoint partition \(\set{D_n}_{n = 1,\ldots,N_S+1}\) of the set \(D\) with the
properties
\[
\abs{u}(D_n) > 0 \quad\text{for all } n = 1,\ldots,N_S+1.
\]
Define for \(n = 1,\ldots,N_S+1\) the restrictions
\[
u_n = u\rvert_{D_n} \in \Mgen.
\]
It is clear that \(\norm{u_n}_{\Mgen} = \abs{u}(D_n) > 0\). Now, we consider
the renormalized measures and their image under \(S\), i.e.
\begin{align*}
v_n &= \frac{u_n}{\norm{u_n}_{\Mgen}}, \\
w_n &= Sv_n \in \Ran S \subset H_2,
\end{align*}
and look for a nontrivial solution \(\lambda \in \R^{N_S+1} \setminus\set{0}\) of the
system of linear equations
\[
\sum_{n=1}^{N_S+1} \lambda_n Sv_n = \sum_{n=1}^{N_S+1} \lambda_n w_n = 0 \in \Ran S.
\]
Since the number of equations is one smaller than the number of variables, such a
solution exists.
Without restriction, we may assume \(\sum_{n=1,\ldots,N_S+1} \lambda_n \geq
0\) (otherwise, we take the negative of \(\lambda\)). We define
\[
  \tau = \max_{n=1,\ldots,N_S+1} \frac{\abs{\lambda_n}}{\norm{u_n}_{\Mgen}}
\]
and \(u_{+}\) and \(u_{-}\) as
\[
u_{\pm} = u \pm \frac{1}{\tau} \sum_{n=1}^{N_S+1} \lambda_n v_n
 = \sum_{n=1}^{N_S+1} \left(1 \pm \frac{\lambda_n}{\tau\,\norm{u_n}_{\Mgen}} \right) u_n.
\]
Clearly, \(u_{+} \neq u_{-} \neq u\).
By construction and linearity of \(S\) we have \(Su_{\pm} = Su =
\widehat{p}\). Furthermore, we directly verify that
\begin{multline*}
\norm{u_{\pm}}_{\Mgen}
= \int_{D} \d \abs{u_\pm}
= \sum_{n=1}^{N_S+1} \int_{D_n} \d \abs{u_\pm} \\
= \sum_{n=1}^{N_S+1} \left(\norm{u_n}_{\Mgen} \pm \frac{\lambda_n}{\tau}\right)
= \norm{u}_{\Mgen} \pm \frac{1}{\tau}\sum_{n=1}^{N_S+1} \lambda_n
\end{multline*}
since $\abs{\lambda_n}/\tau\leq \norm{u_n}_{\Mgen}$. Since
\(\sum_{n=1,\ldots,N_S+1} \lambda_n \geq 0\) we have
\(\norm{u_{-}}_{\Mgen} \leq \norm{u}_{\Mgen}\), and \(u_{-}\) is an optimal
solution of~\eqref{eq:problem_convex}, i.e., \(u_{-} \in U_{\alpha,p_d}\)
(Corollary~\ref{cor:opt}). Moreover, we see that it must hold
\[
\sum^{N_S+1}_{n=1} \lambda_n = 0,
\]
since the norm cannot be strictly smaller,
since \(u \in U_{\alpha,p_d}\). It follows that also \(u_{+}\) is optimal.
We conclude the proof with the observation that
\[
u = \frac{1}{2} u_{+} + \frac{1}{2} u_{-},
\]
which contradicts the assumption that \(u\) is extremal in \(U_{\alpha,p_d}\).
\QED
\end{proof}
The given proof can be modified into a constructive procedure to remove excess points from
the support of an existing (suboptimal) solution of~\eqref{eq:problem_convex}.
\begin{proposition}
\label{prop:remove_diracs}
Suppose that \(\dim \Ran S = N_S < \infty\).
Let \(u = \sum_{n=1,\ldots,P} \coeff{u}_n \delta_{x_n}\) be a arbitrary with \(P \in \mathbb{N}\),
\(u_n \in H_1\), \(x_n \in D\) (pairwise distinct). Then, there exists a
\(u^{new} = \sum_{n=1,\ldots,P} \coeff{u}^{new}_n \delta_{x_n}\) such that
\[
\norm{u^{new}}_{\Mgen} \leq \norm{u}_{\Mgen},
\quad Su^{new} = Su,
\]
and all but \(N_S\) of the coefficients \(\coeff{u}^{new}_n\) are equal to zero.
\end{proposition}
\begin{proof}
The proof is done by induction on \(P\). We only perform the step \(N_S + 1\) to \(N_S\).
As in the previous proof, we define
\[
u_n = u\rvert_{\set{x_n}} = \coeff{u}_n \delta_{x_n},
\text{ and } w_n = S(\coeff{v}_n \delta_{x_n}), \text{ where } \coeff{v}_n =
\frac{\coeff{u}_n}{\norm{\coeff{u}_n}_{H_1}}.
\]
We find the nontrivial solution of \(\sum_{n=1,\ldots,N_S+1} \lambda_n w_n = 0\) with
\(\sum_{n=1,\ldots,N_S+1} \lambda_n \geq 0\).
Now, in contrast to the previous proof, we set
\[
\tau = \max_{n=1,\ldots,N_S+1} \frac{\lambda_n}{\norm{\coeff{u}_n}_{H_1}} \geq 0.
\]
We set
\[
u_{new} = u - \frac{1}{\tau} \sum_{n=1}^{N_S+1} \lambda_n \coeff{v}_n\delta_{x_n}
= \sum_{n=1}^{N_S+1} \left(1-\frac{\lambda_n}{\tau\norm{\coeff{u}_n}_{H_1}}\right) \coeff{u}_n \delta_{x_n}
\]
Thus, the coefficients of \(u^{new}\) are given as \(\coeff{u}^{new}_n =
[1-\lambda_n/(\tau\norm{\coeff{u}_n}_{H_1})]\coeff{u}_n\). It holds that \(\norm{u^{new}}_{\Mgen}
 = \norm{u}_{\Mgen} - \sum_{n=1,\ldots,N_S+1} \lambda_n/\tau
 \leq \norm{u}_{\Mgen}\) since $\lambda_n/\tau\leq \norm{u_n}_{H_1}$
and we finish the proof with the observation that
\[
  \coeff{u}^{new}_{\widehat{n}} = 0 \quad\text{for }
  \widehat{n} \in \argmax_{n=1,\ldots,N_S+1} \frac{\lambda_n}{\norm{\coeff{u}_n}_{H_1}}.
\qedhere
\qquad\QED
\]
\end{proof}

\section{Weak-\(*\) convergence of discrete measures}
\label{app:weak_closedness}
We prove the closedness of sets comprising vector measures
supported on a uniformly bounded number of support points with respect to the weak-\(*\)
topology on $\Mgen$.

\begin{proposition}
\label{prop:compact_Diracs}
Let $D$ be compact. For any $N_d \in \N$ the set
\begin{align*}
P ^{N_d}=\left\{\sum_{j=1}^{N_d} \coeff{u}_j \delta_{x_j} \;\Big|\; \coeff{u}_j\in H_1, x_j\in D \right\}
\end{align*}
is weak-\(*\) closed.
\end{proposition}
\begin{proof}
Let an arbitrary weak-\(*\) convergent sequence $\set{u_k}_{k\in \N} \subset P^{N_d}$ with
limit $\widehat{u}$ be given. For each $k\in \N$ there exist
$\coeff{u}^k_j \in H_1$, $x^k_j\in D$, \(j= 1,\ldots,N_d\) with
\begin{align*}
u_k = \sum_{j=1}^{N_d} \coeff{u}^k_j \delta_{x^k_j}
\quad\text{and}\quad \norm{u_k}_{\Mgen}=\sum_{j=1,\ldots,N_d} \norm{\coeff{u}^k_j}_{H_1}\leq C,
\end{align*}
for some $C>0$. Introducing the vectors
$\coeff{u}^k=(\coeff{u}^k_1,\ldots,\coeff{u}^k_{N_d} )^T\in H^{N_d}_1$ and
$x^k=(x^k_1,\ldots,x^k_{N_d} )^T\in D^{N_d}$, there exist a subsequence of
$(\coeff{u}^k,x^k) \in H_1^{N_d}\times D^{N_d}$ denoted by the same symbol and
$(\coeff{u},x) \in H_1^{N_d}\times D^{N_d}$ with $\coeff{u}^k \rightharpoonup^* \coeff{u}$
and $x^k \rightarrow x$ due to the compactness of $D$ and the boundedness of
$\coeff{u}^k$. Defining
\begin{align*}
u = \sum_{j=1,\ldots,N_d} \coeff{u}_j \delta_{x_j},
\end{align*}
we arrive at
\begin{align*}
\pair{\phi,{u}}
  = \lim_{k\rightarrow \infty} \sum_{j=1,\ldots,N_d} (\coeff{u}^k_j, \phi(x^k_j))_{H_1}
  = \lim_{k\rightarrow \infty} \pair{\phi,u_k}
  = \pair{\phi,\widehat{u}}
\end{align*}
for all $\phi\in \Cgen$ since $\coeff{u}^k_j \rightharpoonup \coeff{u}_j$ and
$\norm{\phi(x^k_j)-\phi(x_j)}_{H_1}\rightarrow 0$. Due to the uniqueness of the weak-\(*\) limit
we get $\widehat{u}={u}\in P^{N_d}$ yielding the weak-* closedness of $P^{N_d}$.
\QED
\end{proof}
As a corollary each accumulation point of a sequence of measures with uniformly bounded support size is also finitely supported.
\begin{corollary}
Let $D$ be compact. Consider a sequence $u_k\in \Mgen$ with $\#
\supp \abs{u_k}\leq N_d$ for some $N_d \in \N$. Then every accumulation point $\widehat{u}$ of
$u_k$ fulfills $\#\supp \abs{\widehat{u}}\leq N_d$.
\end{corollary}
\begin{proof}
Since every measure of support less that \(N_d\) can be written as a sum over \(N_d\)
Dirac delta functions (by possibly adding additional Dirac delta functions with zero
coefficient), applying Proposition~\ref{prop:compact_Diracs} yields the result.
\QED
\end{proof}
\begin{align*}
\end{align*}



\bibliographystyle{siam}
\bibliography{helmholtz_sparsity}

\end{document}